\documentclass[twoside]{scrartcl}
\usepackage{scrlayer-scrpage}

\usepackage[utf8]{inputenc}
\usepackage[T1]{fontenc}
\usepackage[english]{babel}
\usepackage[unicode=true,plainpages=false]{hyperref}

\usepackage[left=2.5cm, right=3cm, top=2cm, bottom=3.5cm]{geometry}

\usepackage[shortlabels]{enumitem}

\usepackage{amsmath}
\usepackage{amssymb}
\usepackage{amsthm}
\usepackage{mathrsfs}
\usepackage{yhmath}
\usepackage{wasysym}
\usepackage{mathtools}

\usepackage[capitalise]{cleveref}

\usepackage{caption}
\usepackage{subcaption}

\usepackage{graphicx}
\usepackage{pgf,tikz}
\usepackage{tikz-cd}
\usepackage{rotating}

\usepackage{lmodern} 
\usepackage{tabularx} 
\usepackage[draft=false,babel,kerning=true,spacing=true]{microtype} 

\usepackage{todonotes}

\usepackage{xcolor}

\newcommand{\cycl}{\mathrm{cycl}}

\allowdisplaybreaks

\let\cprod\times

\makeatletter
\newcommand*{\relrelbarsep}{.386ex}
\newcommand*{\relrelbar}{%
  \mathrel{%
    \mathpalette\@relrelbar\relrelbarsep
  }%
}
\newcommand*{\@relrelbar}[2]{%
  \raise#2\hbox to 0pt{$\m@th#1\relbar$\hss}%
  \lower#2\hbox{$\m@th#1\relbar$}%
}
\providecommand*{\rightrightarrowsfill@}{%
  \arrowfill@\relrelbar\relrelbar\rightrightarrows
}
\providecommand*{\leftleftarrowsfill@}{%
  \arrowfill@\leftleftarrows\relrelbar\relrelbar
}
\providecommand*{\xrightrightarrows}[2][]{%
  \ext@arrow 0359\rightrightarrowsfill@{#1}{#2}%
}
\providecommand*{\xleftleftarrows}[2][]{%
  \ext@arrow 3095\leftleftarrowsfill@{#1}{#2}%
}
\makeatother

\DeclarePairedDelimiter{\abs}{\lvert}{\rvert}

\newcommand{\noloc}{\nobreak\mskip6mu plus1mu\mathpunct{}\nonscript\mkern-\thinmuskip{:}\mskip2mu\relax} 

\newcommand{\calA}{\mathcal{A}}
\newcommand{\Sp}{\mathrm{Sp}}

\renewcommand{\Pr}{\mathrm{Pr}}

\newcommand{\isom}{\cong}

\newcommand{\from}{\mathbin{\leftarrow}}
\newcommand{\xto}[1]{\mathbin{\xrightarrow{#1}}} 
\newcommand{\isoto}{\xto\sim}

\DeclareMathOperator{\Hom}{Hom}

\DeclareMathOperator{\Fun}{Fun}

\newcommand{\isect}{\mathbin{\cap}}
\newcommand{\bigisect}{\bigcap}
\newcommand{\union}{\mathbin{\cup}}
\newcommand{\bigunion}{\bigcup}
\newcommand{\dunion}{\mathbin{\sqcup}}
\newcommand{\bigdunion}{\bigsqcup}

\renewcommand{\emptyset}{\varnothing}

\renewcommand{\subset}{\subseteq}

\newcommand{\comp}{\mathbin{\circ}}
\DeclareMathOperator{\id}{id}

\newcommand{\restrict}[2]{{#1}|_{#2}}

\newcommand{\injto}{\mathrel{\hookrightarrow}}
\newcommand{\surjto}{\mathrel{\twoheadrightarrow}}

\newcommand{\bigdsum}{\bigoplus}
\newcommand{\cplbigdsum}[1]{\widehat\bigdsum_{#1} \,}
\newcommand{\cd}{\mathrm{cd}} 

\newcommand{\Z}{{\mathbb{Z}}}
\newcommand{\N}{\mathbb{N}}
\newcommand{\Fld}{\mathbb{F}}
\newcommand{\Q}{\mathbb{Q}}

\newcommand{\F}[1]{\mathbb{F}_{#1}}

\DeclareMathOperator{\End}{End}

\newcommand{\tensor}{\otimes}

\DeclareMathOperator{\Spec}{Spec}

\newcommand{\Spf}{\mathrm{Spf}\,}

\DeclareMathOperator{\IHom}{\underline{\Hom}}

\newcommand{\et}{{\mathrm{et}}}

\newcommand{\proet}{{\mathrm{proet}}} 
\newcommand{\qproet}{{\mathrm{qproet}}} 
\newcommand{\aff}{{\mathrm{aff}}} 

\newcommand{\cplt}{\,\widehat{}\,}
\newcommand{\ri}{\mathcal O} 
\newcommand{\mm}{\mathfrak m} 

\DeclareMathOperator{\Spa}{Spa}
\DeclareMathOperator{\Spd}{Spd}

\newcommand{\opp}{{\mathrm{op}}} 

\newcommand{\solid}{{\scalebox{0.5}{$\square$}}}
\DeclareMathOperator{\Mod}{Mod}

\newcommand{\Perfd}{\mathrm{Perfd}}
\newcommand{\AffPerfd}{\Perfd^{\mathrm{aff}}}

\DeclareMathOperator{\D}{\mathcal D} 

\newcommand{\infcatinf}{\mathcal Cat_\infty}
\newcommand{\catcat}{\mathcal Cat}

\newcommand{\SymMonCat}{\CAlg(\infcatinf)}

\DeclareMathOperator{\AnRing}{AnRing}

\DeclareMathOperator{\AdicRing}{AdicRing}

\newcommand{\oc}{{\mathrm{oc}}} 
\newcommand{\cpl}{{\mathrm{cpl}}} 

\DeclareMathOperator{\fib}{fib}

\newcommand{\ex}{{\mathrm{ex}}} 

\DeclareMathOperator{\Ind}{Ind} 

\DeclareMathOperator{\vStacks}{vStack}

\newcommand{\dimtrg}{\mathrm{dim.trg}} 

\newcommand{\nuc}{{\mathrm{nuc}}} 
\newcommand{\tr}{{\mathrm{tr}}} 

\newcommand{\wl}{{\mathrm{wl}}} 

\DeclareMathOperator{\CAlg}{CAlg} 

\theoremstyle{plain}
\newtheorem{theorem}{Theorem}[section]
\newtheorem{theorem*}{Theorem}
\newtheorem{proposition}[theorem]{Proposition}
\newtheorem{proposition*}[theorem*]{Proposition}
\newtheorem{corollary}[theorem]{Corollary}
\newtheorem{lemma}[theorem]{Lemma}

\theoremstyle{definition}
\newtheorem{definition}[theorem]{Definition}
\newtheorem{definition*}[theorem*]{Definition}
\newtheorem{example}[theorem]{Example}

\newtheorem{remark}[theorem]{Remark}
\newtheorem{remarks}[theorem]{Remarks}

\newtheorem{hypothesis*}[theorem*]{Hypothesis}

\numberwithin{equation}{theorem}
\numberwithin{figure}{section}
\numberwithin{table}{section}

\newlist{thmenum}{enumerate}{1}
\setlist[thmenum]{label=(\roman*), ref=\thetheorem.(\roman*)}
\crefalias{thmenumi}{theorem}

\newlist{propenum}{enumerate}{1}
\setlist[propenum]{label=(\roman*), ref=\theproposition.(\roman*)}
\crefalias{propenumi}{proposition}

\newlist{corenum}{enumerate}{1}
\setlist[corenum]{label=(\roman*), ref=\thecorollary.(\roman*)}
\crefalias{corenumi}{corollary}

\newlist{lemenum}{enumerate}{1}
\setlist[lemenum]{label=(\roman*), ref=\thelemma.(\roman*)}
\crefalias{lemenumi}{lemma}

\newlist{exampleenum}{enumerate}{1}
\setlist[exampleenum]{label=(\alph*), ref=\theexamples.(\alph*)}
\crefalias{exampleenumi}{example}

\newlist{remarksenum}{enumerate}{1}
\setlist[remarksenum]{label=(\roman*), ref=\theremarks.(\roman*)}
\crefalias{remarksenumi}{remark}

\newlist{defenum}{enumerate}{1}
\setlist[defenum]{label=(\alph*), ref=\thedefinition.(\alph*)}
\crefalias{defenumi}{definition}

\title{Descent for solid quasi-coherent sheaves on perfectoid spaces}
\author{Johannes Anschütz  \and Lucas Mann}
\date{\today}

\begin{document}

\maketitle

\begin{abstract}
  We prove $v$-descent for solid quasi-coherent sheaves on perfectoid spaces as a key technical input for the development of a $6$-functor formalism with values in solid quasi-coherent sheaves on relative Fargues--Fontaine curves. 
\end{abstract}

\tableofcontents

\section{Introduction}
\label{sec:introduction}

Let $p$ be a prime, and let $\Perfd$ be the category of perfectoid spaces over $\Z_p$. We equip $\Perfd$ with the $v$-topology, i.e., the Grothendieck topology generated by surjections of affinoid perfectoid spaces. The main result of this paper is the following theorem:

\begin{theorem}[{\cref{rslt:v-hyperdescent-for-O+-modules}, \cref{sec:main-theor-d_hats-main-descent-theorem}}]
  \label{sec:introduction-1-main-theorem-introduction}
  There exists a unique hypercomplete $v$-sheaf of $\infty$-categories
  \[
    \Perfd^{\mathrm{op}}\to \mathcal{C}at_\infty,\ X\mapsto \D^a_{\hat\solid}(\ri^+_X),
  \]
  such that for each affinoid perfectoid space $X=\Spa(A,A^+)$ whose tilt $X^\flat$ admits a map of finite $\dimtrg$ to some totally disconnected space, we have
  \begin{align*}
    \D^a_{\hat\solid}(\ri^+_X)\cong \D^a_{\hat\solid}(A^+),
  \end{align*}
  compatibly with pullback. 
\end{theorem}

Here, $\D_{\hat\solid}^a(A^+)$ refers to a(n almost version of a) slight modification of the category $\D_\solid(A^+)$ of solid $A^+$-modules introduced in \cite{condensed-mathematics} and \cite{andreychev-condensed-huber-pairs} (the exact definition is \cref{sec:adic-rings-1-definition-adic-analytic-ring}, \cref{sec:defin-d_hats-2-definition-of-almost-plus-category}). Roughly the potentially non-complete compact generators of $\D_\solid(A^+)$ get replaced in $\D_{\hat\solid}(A^+)$ by their $\pi$-adic completions for a pseudo-uniformizer $\pi\in A$. \cref{sec:introduction-1-main-theorem-introduction} formally implies $v$-hyperdescent of the functor $X\mapsto \D_{\hat\solid}(\ri_X)$ on perfectoid spaces because $\D_{\hat\solid}(\ri_X):= \mathrm{Mod}_{\ri_X}(\D^a_{\hat\solid}(\ri^+_X))$.

\Cref{sec:introduction-1-main-theorem-introduction} will be used in \cite{anschuetz_mann_lebras_6_functors_for_solid_sheaves_on_fargues_fontaine_curves} as the essential technical ingredient for setting up a $6$-functor formalism with values in (modified) solid quasi-coherent sheaves on the Fargues--Fontaine curve. This $6$-functor formalism implies finiteness and duality results for the pro-\'etale cohomology of general pro-\'etale $\Q_p$-local systems on smooth rigid-analytic varieties, and we refer to \cite{anschuetz_mann_lebras_6_functors_for_solid_sheaves_on_fargues_fontaine_curves} for more details and motivation for the study of these assertions (and hence of \cref{sec:introduction-1-main-theorem-introduction}). We note that our proof of \cref{sec:introduction-1-main-theorem-introduction} would not simplify if the pro-\'etale topology is considered instead of the $v$-topology. Moreover, the flexibility of the descent statement beyond the totally disconnected case is critical to establish the existence of the $6$-functor formalism, e.g., the projection formula, and some of its properties, e.g., cohomological smoothness for smooth rigid-analytic varieties over $\mathbb{C}_p$.

We now give a rough sketch of the proof of \cref{sec:introduction-1-main-theorem-introduction}, which at the same time will summarize the content of the different sections in this paper. First of all, \cref{sec:introduction-1-main-theorem-introduction} is an $\ri^+$-analog of a main theorem of \cite{mann-mod-p-6-functors}, namely \cite[Theorem 1.8.3]{mann-mod-p-6-functors}, which we recall due to its importance for this paper:

\begin{theorem}[{\cite[Theorem 1.8.2., Theorem 1.8.3., Theorem 3.5.21]{mann-mod-p-6-functors}}]
  \label{sec:introduction-1-mod-p-theorem-of-lucas}
  There exists a unique hypercomplete $v$-sheaf of $\infty$-categories
  \[
    \Perfd_\pi^{\mathrm{op}}\to \mathcal{C}at_\infty,\ X\mapsto \D^a_{\solid}(\ri^+_X/\pi),
  \]
  such that for each $+$-bounded affinoid perfectoid space $X=\Spa(A,A^+)$ with uniformizer $\pi$ we have $\D^a_{\solid}(\ri^+_X/\pi)\cong \D^a_{\solid}(A^+/\pi)$ compatibly with pullback.
\end{theorem}

Here, $\Perfd_\pi$ denotes the category of perfectoid spaces $X$ with a choice of a pseudo-uniformizer $\pi$ and morphisms respecting $\pi$. We note that $\D_\solid^a(A^+/\pi)$ is (an almost version of) the solid category for the discrete ring $A^+/\pi$ introduced in \cite{condensed-mathematics}. We use the name ``$+$-bounded'' for what is called ``$p$-bounded'' in \cite{mann-mod-p-6-functors}, which in turn is a technical condition defined in \cite[\S3.5]{mann-mod-p-6-functors}. Part of our paper is a reformulation of $+$-boundedness in terms of the much simpler condition of finite $\dimtrg$ appearing in \cref{sec:introduction-1-main-theorem-introduction}; we refer the reader to \cref{sec:compare-p-bounded-and-+-bounded} for this reformulation and ignore the difference in the following exposition.

\cref{sec:introduction-1-main-theorem-introduction} and \cref{sec:introduction-1-mod-p-theorem-of-lucas} are tightly related. On the one hand, \cref{sec:defin-d_hats-2-comparison-with-almost-definition-of-lucas-thesis} implies that
\[
  \mathrm{Mod}_{\ri^+_X/\pi}(\D^a_{\hat\solid}(\ri^+_X))\cong \D^a_{\solid}(\ri^+_X/\pi)
\]
for each perfectoid space $X$ with a pseudo-uniformizer $\pi$, which shows that \cref{sec:introduction-1-main-theorem-introduction} implies \cref{sec:introduction-1-mod-p-theorem-of-lucas}. On the other hand our definition of $\D^a_{\hat\solid}(A^+)$ is made to use \cref{sec:introduction-1-mod-p-theorem-of-lucas} in our proof of \cref{sec:introduction-1-main-theorem-introduction}. We note that this reduction does not seem possible for the usual category $\D^a_{\solid}(A^+)$ instead of $\D^a_{\hat\solid}(A^+)$ -- we heavily use that the compact generators of the latter are $\pi$-complete.

In \cref{sec:defin-d_hats} we define the modification $\D_{\hat\solid}(A)$ of $\D_\solid(A)$ for any adic analytic ring $A$ (in the general sense of \cref{sec:adic-rings-1-definition-adic-analytic-ring}) by $I$-completing the compact generators of $\D_\solid(A)$ with respect to some ideal of definition $I\subseteq \pi_0(A)$, and study its basic properties. In particular, we prove the following comparison statements.
\begin{theorem}[{\cref{sec:defin-d_hats-1-equivalence-of-nuclear-subcategory}, \cref{sec:adic-rings-3-modified-version-agrees-with-old-one-for-finite-type-stuff}, \cref{sec:defin-d_hats-2-complete-and-discrete-mod-i-implies-nuclear}}]
  \label{sec:introduction-1-comparison-of-solid-and-hat-solid}
  Let $A$ be an adic analytic ring, and $\alpha^\ast\colon \D_\solid(A)\to \D_{\hat\solid}(A)$ the natural functor.
  \begin{thmenum}
  \item If $A^+$ is of finite type over $\Z_p$, then $\alpha^\ast$ is an equivalence.
  \item The functor $\alpha^\ast$ induces an equivalence $\D_\solid(A)^\nuc\cong \D_{\hat\solid}(A)^\nuc$ on nuclear objects.
  \item $\D_\solid(A)^\nuc\subseteq \D_\solid(A)$ is the full subcategory generated under colimits by $I$-complete objects $M\in \D_\solid(A)$, which are discrete mod $I$. 
  \end{thmenum}
\end{theorem}

In \cref{rslt:descendability-of-adic-rings} we prove a descent theorem for $\D_{\hat\solid}(-)$ as a functor on adic rings, by reducing the descent question modulo all $I^n$:

\begin{theorem}[{\cref{rslt:descendability-of-adic-rings}}]
  \label{sec:introduction-3-adic-descent-introduction}
Let $A\to B$ be an adic morphism of adic rings, and $I\subseteq \pi_0(A)$ an ideal of definition. Assume that $A/I^n\to B/I^n, n\geq 0,$ is descendable of index independent of $n$ (in the sense of \cite[Definition 2.6.7]{mann-mod-p-6-functors}). Then $A\to B$ satisfies descent for $\D_{\hat\solid}(-)$.  
\end{theorem}

 Again, such a statement does not seem possible for $\D_\solid(-)$. Morever, following \cite{andreychev-condensed-huber-pairs} we globalize $\D_{\hat\solid}$ to stably uniform analytic adic spaces in \cref{sec:glob-stably-unif} with an eye towards our applications in \cite{anschuetz_mann_lebras_6_functors_for_solid_sheaves_on_fargues_fontaine_curves}.

\begin{theorem}[{\cref{sec:glob-stably-unif-1-analytic-and-etale-descent-for-hat-solid-version}}]
  \label{sec:introduction-2-descent-for-hat-solid-on-stable-uniform-spaces-introduction}
  The functor $\Spa(B,B^+)\mapsto \mathrm{Mod}_B(\D_{\hat\solid}(B^\circ,B^+))$ satisfies analytic (resp.\ \'etale) descent on stably uniform (resp.\ sousperfectoid) affinoid analytic adic spaces. 
\end{theorem}

In \cref{sec:defin-d_hats-a+} we move to perfectoid spaces, where we define $\D_{\hat\solid}^a(A^+)$ for any affinoid perfectoid space $X=\Spa(A,A^+)$ (\cref{sec:defin-d_hats-2-definition-of-almost-plus-category}). Then we apply the complete descent results from \cref{sec:complete-descent} to obtain $v$-descent of $\D_{\hat\solid}^a(A^+)$ on totally disconnected perfectoid spaces (\cref{rslt:v-hyperdescent-for-O+-modules}). From here, we can define $\D_{\hat\solid}^a(\ri_X^+)$ for any perfectoid space $X$ (or even small $v$-stack on $\Perfd$) by $v$-descent. This proves the existence part of \cref{sec:introduction-1-main-theorem-introduction}, with the exception of the identification $\D^a_{\hat\solid}(\mathcal{O}^+_X)\cong \D^a_{\hat\solid}(A^+)$ if $X=\Spa(A,A^+)$ is an affinoid perfectoid space whose tilt has a map of finite $\dimtrg$ to a totally disconnected space. This question we adress next.

There exists an evident functor
\[
 \widetilde{(-)}\colon \D_{\hat\solid}^a(A^+)\to \D_{\hat\solid}^a(\ri_X^+), 
\]
and this functor admits a right adjoint, denoted by $\Gamma(X,-)$. Let $\pi\in A$ be a pseudo-uniformizer. Using \cref{sec:introduction-1-mod-p-theorem-of-lucas}, it is fairly easy to see that $\widetilde{(-)}$ is an equivalence if (and only if) the abstractly defined category $\D_{\hat\solid}^a(\ri_X^+)$ is generated under colimits by (suitably bounded) $\pi$-complete objects, and $\Gamma(X,-)$ preserves colimits (see \cref{sec:bound-cond-1-criterion-for-sheafification-being-an-equivalence}). Both of these properties of $\D^a_{\hat\solid}(\ri^+_X)$ are proven by a detour through a category $\D_\nuc(X,\Z_p[[\pi]])$ of (overconvergent) nuclear sheaves of $\Z_p[[\pi]]$-modules on $X_\qproet$ that we discuss in \cref{sec:nuclear-sheaves} following \cite{mann-nuclear-sheaves} and \cite{fargues-scholze-geometrization}. The discussion of $\D_\nuc(X,\Z_p[[\pi]])$ involves a rather long and detailed analysis of solid, $\omega_1$-solid and overconvergent objects in $\D(X_\qproet,\Z_p[[\pi]])$ for a general $p$-bounded spatial diamond. The criticial assertion in this section concerns the nuclear objects in the category of $\omega_1$-solid objects $\D_\solid(X,\Z_p[[\pi]])_{\omega_1}\subseteq \D(X_\qproet,\Z_p[[\pi]])$. We stress that nuclearity here refers to the abstract notion from \cite[Lecture VIII]{condensed-complex-geometry}, and not to the more geometric notion introduced in \cite{mann-nuclear-sheaves}.

\begin{lemma}[{\cref{sec:nucl-objects-overc-2-nuclear-objects-are-overconvergent}, \cref{sec:nucl-objects-overc-1-d-nuc-generated-by-complete-objects}}]
  \label{sec:introduction-4-properties-of-nuclear-objects-introduction}
  Let $\ell$ be a prime, and let $X$ be an $\ell$-bounded affinoid perfectoid space, and $\Lambda$ an adic profinite $\Z_\ell$-algebra, e.g., $\Lambda=\Z_p[[\pi]]$ if $\ell=p$. Then the category $\D_\nuc(X,\Lambda)$ of nuclear objects in $\D_\solid(X,\Lambda)_{\omega_1}$ is generated under colimits by overconvergent, complete, right-bounded objects.
\end{lemma}

Especially the generation by complete objects is non-obvious. Together with the good cohomological properties of $\omega_1$-solid $\Z_p[[\pi]]$-sheaves on $X_\qproet$ this is the source of the desired properties of $\D^a_{\hat\solid}(\mathcal{O}^+_X)$. The precise implementation of these properties involves \cref{rslt:comparison-of-oc-nuc-sheaves-to-O+-modules-on-p-bd-perfd}, which realizes $\D^a_{\hat\solid}(\mathcal{O}^+_X)$ as a category of $\mathcal{O}^{+a}$-modules in the abstract $\D_\nuc(\Z_p[[\pi]])$-linear category
\[
  \D_\nuc(X,(A^+)^a_{\hat\solid}):=\D_\nuc(X,\Z_p[[\pi]])\otimes_{\D_\nuc(\Z_p[[\pi]])} \D^a_{\hat\solid}(A^+).
\] This latter category is well-behaved thanks to the rigidity of $\D_\nuc(\Z_p[[\pi]])$ (\cref{sec:mathc-valu-geom-remark-rigidity-of-d-nuc-r}) and dualizability of $\D^a_{\hat\solid}(A^+)$ as a $\D_\nuc(\Z_p[[\pi]])$-module (\cref{rslt:properties-of-A+-linear-oc-nuc-sheaves}). However, the equivalence
\[
  \D^a_{\hat\solid}(\ri^+_X) \cong \mathrm{Mod}_{\mathcal{O}^{+a}}(\D_\nuc(X,(A^+)^a_{\hat\solid})).
\]
of \cref{rslt:comparison-of-oc-nuc-sheaves-to-O+-modules-on-p-bd-perfd} needs as an input the equivalence $\D^a_{\hat\solid}(\mathcal{O}^+_{\overline{Z}^{/X}})\cong \D^a_{\hat\solid}(B^+,A^+)$ for the relative compactification $\overline{Z}^{/X}$ of a strictly totally disconnected space $Z=\Spa(B,B^+)$ which is quasi-pro-\'etale over $X$. As $\overline{Z}^{/X}$ is not necessarily totally disconnected, this does not follow from the descent of totally disconnected spaces proven in \cref{rslt:v-hyperdescent-for-O+-modules}, and we supply an argument in \cref{rslt:compute-Dqcohri-on-rel-compactification-of-td-space} using an analysis of families of Riemann-Zariski spaces (\cref{rslt:descendability-of-open-cover-of-ZR-space-family}).

\subsection{Acknowledgements}
\label{sec:acknowledgements}

We thank Arthur-C\'esar Le Bras heartily for all his feedback, and for all discussions in an early stage of this project. For personal, Monadic reasons he unfortunately had to leave this project in a later stage. We moreover thank Gregory Andreychev, Johan de Jong, Juan Esteban Rodr\'iguez Camargo and Peter Scholze for discussions related to this paper. This paper was finished as a part of the DFG-Sachbeihilfe 534205068. Moreover, the second author was partially supported by ERC Consolidator Grant 770936:NewtonStrat.

\subsection{Notations and conventions}
\label{sec:notat-conv}

For technical convenience we fix an implicit cut-off cardinal $\kappa$ (in the sense of \cite[Section 4]{etale-cohomology-of-diamonds}), and assume all our perfectoid spaces, and condensed sets to be $\kappa$-small. In particular, for a Huber pair $(A,A^+)$ its associated category $\D_\solid(A,A^+)$ (\cite[Theorem 3.28]{andreychev-condensed-huber-pairs}) is generated by a \textit{set} of compact objects. Passing to the filtered colimit over all $\kappa$'s implies the descent statement in general.

We work implicitly ``condensed'' and ``animated''. More precisely, a ring $R$ will mean a condensed animated ring. It is called static if $R\cong \pi_0(R)$ (as a condensed ring), and discrete if $R\cong \underline{R(\ast)_{\mathrm{disc}}}$, where $R(\ast)$ is the ``underlying animated ring'', or equivalently, $R$ is discrete if the functor $S\mapsto R(S)$ on profinite sets maps cofiltered inverse limits to filtered colimits.

Given a ring $R$ and an ideal $I\subseteq \pi_0(R(\ast))$ generated by elements $f_1,\ldots, f_n\in I$, we set $M/I:=\Z\otimes_{\Z[x_1,\ldots, x_n]}M$ for any $R$-module $M$ where the tensor product is implicitly derived. Here, the map $\Z[x_1,\ldots, x_n]\to R$ making $M$ into an $\Z[x_1,\ldots, x_n]$-module is classified by the elements $f_1,\ldots, f_n$. In particular, the quotient $M/I$ depends on the choice of elements, and we will carefully make this clear when this may create confusion. Concretely, the quotient $M/I$ is calculated by the tensor product of the complexes $M\xrightarrow{f_i} M$, $i=1,\ldots, n$.

Given a $\Z[x]$-linear category $\mathcal{C}$ with sequential limits we call $X\in \mathcal{C}$ $x$-complete if the inverse limit
\[
  \ldots \to X \xrightarrow{x} X
\]
vanishes.
  
\section{\texorpdfstring{$\D_{\hat\solid}(A)$}{D\_sld(A)} for adic rings}
\label{sec:d_hats-adic-rings}

When proving descent for solid modules over $p$-complete rings, one of the main challenges is that in general the compact generators of $\D_\solid(A)$ are not $p$-complete, even if $A$ itself is so. Indeed, by the very definition the compact generators of $\D_\solid(A)$ are colimits over the category of maps $A^\prime\to A$ with $A^\prime$ a finitely generated $\Z$-algebra (\cite[Definition 3.20]{andreychev-condensed-huber-pairs}). This defect breaks the usual strategy of reducing descent questions to the reductions mod $p^n$ and therefore makes it hard to get any good general descent results on $p$-adically complete rings and general solid modules. To overcome this issue, we present a slight modification of $\D_\solid(A)$, denoted $\D_{\hat\solid}(A)$ in the following (\cref{def:modified-modules-over-adic-ring}), where we replace the compact generators of $\D_\solid(A)$ by their $p$-adic completions. This construction may seem a bit artificial, but the results in this section show that we do not lose much and still have a very tight relation to $\D_\solid(A)$. With this modified version of solid modules we then explain how descent questions can indeed be reduced modulo $p^n$ (\cref{rslt:descendability-of-adic-rings}).

Of course the construction of $\D_{\hat\solid}(A)$ works for any adically complete rings (instead of just $p$-complete rings), so we perform it in the natural generality. Hence, we start by discussing adic analytic rings.

\subsection{Adic rings}
\label{sec:adic-rings}

We fix the following terminology. Note that we adhere to our conventions from \cref{sec:notat-conv}.

We let $\Z_\solid$ be the solid analytic ring from \cite[Definition 5.1]{condensed-mathematics} and $\Z[T]_\solid$ the analytic ring from \cite[Theorem 8.2]{condensed-mathematics}. If $A$ is any solid ring, and $S\to \pi_0(A)(\ast)$ a map of sets, then we denote by $(A,S)_\solid$ the analytic ring over $\Z_\solid$ such that an object $M\in \mathrm{Mod}_A(D(\Z_\solid))$ is $(A,S)_\solid$-complete if and only if $M$ is $\Z[s]_\solid$-complete for any $s\in S$ with induced map $\Z[s]\to A$ of rings. For $(A,S)=(A,A^+)$ a classical complete Huber pair, this analytic ring structure has been analyzed in \cite{andreychev-condensed-huber-pairs}. We note that by \cite[Appendix to lecture XII]{scholze-analytic-spaces} and \cite[Proposition 3.32]{andreychev-condensed-huber-pairs} $(A,S)_\solid\cong (A,\widetilde{S})_\solid$ if $\widetilde{S}\subseteq \pi_0(A)(\ast)$ denotes the smallest integrally closed subring containing each topologically nilpotent element and the image of $S$.\footnote{By definition, an element $a\in \pi_0(A)(\ast)$ is topologically nilpotent if the associated map $\Z[T]\to A$ factors over $\Z[[T]]$ as a map of rings. As $A$ is solid and $\Z[[T]]\otimes_{(\Z[T],\Z)_\solid}\Z[[T]]\cong \Z[[T]]$ this factorization is unique if it exists.}

\begin{definition}
  \label{sec:adic-rings-1-definition-adic-analytic-ring}
\begin{defenum}
	\item We say that a ring $A$ is \emph{adic} if there is some finitely generated ideal $I \subset \pi_0 (A)(\ast)$ such that $A/I$ is discrete and $A$ is $I$-adically complete. We call any such ideal $I$ an \emph{ideal of definition} of $A$.

	\item An \emph{adic analytic ring} is an analytic ring $\calA$ of the form $\calA=(A, A^+)_\solid$, where $A$ is an adic ring and $A^+ \subset \pi_0(A)(\ast)$ is a subring. We denote by
	\begin{align*}
		\AdicRing \subset \AnRing
	\end{align*}
	the full subcategory spanned by the adic analytic rings. Given an adic analytic ring $\calA$ we usually write $\underline{\calA}=\calA[\ast]$ for its underlying adic ring.

	\item A map $f\colon A \to B$ of adic rings is called \emph{adic} if for some ideal of definition $I$ of $A$, $f(I)$ generates an ideal of definition of $B$. A map of adic analytic rings is called adic if the map of underlying adic rings is adic.
\end{defenum}
\end{definition}

Given an adic ring $A$ with ideal of definition $I$, we denote by $\D(A):=\mathrm{Mod}_A(\D(\mathrm{Cond}(\mathrm{Ab}))$ its (stable $\infty$-)category of modules in condensed abelian groups, and we let $\D(A)_{\hat I} \subset \D(A)$ be the full subcategory of $I$-adically complete $A$-modules, i.e., those which are $x$-complete for any $x\in I$. This notion does not depend on the choice of $I$ or is generators:

\begin{lemma}
  \label{sec:adic-rings-1-complete-modules-independent-of-ideal-of-definition}
Let $A$ be an adic ring.
\begin{lemenum}
	\item \label{rslt:adic-completeness-is-well-defined} For any two ideals of definition $I$, $I'$ of $A$ we have $\D(A)_{\hat I} = \D(A)_{\hat{I'}}$.
	\item A map $f\colon A \to B$ is adic if and only if for every ideal of definition $I$ of $A$, $f(I)\subseteq \pi_0(B)(\ast)$ generates an ideal of definition of $B$.
\end{lemenum}
\end{lemma}
\begin{proof}
We first prove (i), so let $I$ and $I'$ be given. Without loss of generality we can assume $I \subset I'$ because the intersection of two ideals of definition is an ideal of definition. Pick any $a \in I'$. Then $A/I$ is $a$-adically complete (as a finite limit of copies of $A$'s), and hence the induced map $\Z[x] \to A/I$ factors over $\Z[[x]] \to A/I$. Since $A/I$ is discrete and in particular nuclear as a $\Z_\solid$-module, we have $\Hom_\Z(\Z[[x]], A/I) = \varinjlim_n \Hom_\Z(\Z[x]/x^n, A/I)$, which implies that $a$ is nilpotent in $\pi_0(A/I)(\ast)$. But then any $I$-adically complete $A$-module is also $a$-adically complete. Since $a \in I'$ was arbitrary, we immediately arrive at the claimed identity.

We now prove (ii), so let $f\colon A \to B$ be an adic map of adic rings and $I$ some ideal of definition of $A$. By assumption there is some ideal of definition $I'$ of $A$ such that $f(I')$ generates an ideal of definition of $B$. If $I''$ is any ideal of definition of $A$ containing $I'$ then by the proof of (i), every element in $I''$ has some power which lies in $I'$; hence also $f(I'')$ generates an ideal of definition of $B$. We apply this to $I'' = I' + I$. But by reversing this argument we also deduce that $f(I)$ generates an ideal of definition of $B$, as desired. The converse direction is clear.
\end{proof}

\begin{definition}
  \label{sec:adic-rings-2-definition-complete-category-for-adic-analytic-ring}
Let $A$ be an adic ring. We say that an $A$-module $M \in \D(A)$ is \emph{adically complete} if it is $I$-adically complete for some (equivalently every) ideal of definition $I$ of $A$. We denote by $\D_\cpl(A) \subset \D(A)$ the full subcategory spanned by the adically complete $A$-modules and by
\begin{align*}
	(-)_\cpl\colon \D(A) \to \D_\cpl(A)
\end{align*}
the left-adjoint of the inclusion. For an analytic adic ring $\mathcal A$ we similarly denote $\D_\cpl(\mathcal A) := \D(\mathcal A) \isect \D_\cpl(\underline{\mathcal A})$.
\end{definition}

Here, the existence of $(-)_\cpl$ is guaranteed by the adjoint functor theorem, which applies here as all categories here are presentable (recall that we have fixed a cut-off cardinal $\kappa$ in \cref{sec:notat-conv}). 

\subsection{Definition of \texorpdfstring{$\D_{\hat\solid}(A)$}{D\_sld(A)}}
\label{sec:defin-d_hats}

With a good notion of adic rings and adically complete modules at hand, we can now construct the promised modification of $\D(\mathcal A)$. The definition is very simple, albeit somewhat artificial:

\begin{definition} \label{def:modified-modules-over-adic-ring}
Given $\mathcal A$ an adic analytic ring, let $\mathcal C_{\mathcal A} \subset \D_\cpl(\mathcal A)$ be the full subcategory generated under finite (co)limits and retracts from the essential image of the full subcategory $\D(\mathcal A)^\omega$ of compact $\mathcal A$-modules under $(-)_\cpl$. Then we define
\begin{align*}
	\D(\hat{\mathcal A}) := \Ind(\mathcal C_{\mathcal A}).
\end{align*}
Writing $\mathcal A = (A, A^+)_\solid$, we also use the notation $\D_{\hat\solid}(A, A^+) := \D(\hat{\mathcal A})$.
\end{definition}

The following results show that $\D(\hat{\mathcal A})$ is quite close to $\D(\mathcal A)$ and in particular inherits all of its nice properties:

\begin{lemma} \label{rslt:basic-properties-of-modified-modules}
Let $\mathcal A$ be an adic analytic ring.
\begin{lemenum}
	\item $\D(\hat{\mathcal A})$ is a stable $\infty$-category that can naturally be equipped with a closed symmetric monoidal structure and a $t$-structure.

	\item \label{rslt:definition-of-base-change-to-modified-modules} There is a natural colimit-preserving, symmetric monoidal and right $t$-exact functor $\alpha^*\colon \D(\mathcal A) \to \D(\hat{\mathcal A})$.

	\item \label{rslt:properties-of-forgetful-functor-from-modified-modules} The functor $\alpha^*$ has a right adjoint $\alpha_*\colon \D(\hat{\mathcal A}) \to \D(\mathcal A)$ which is $t$-exact, conservative, preserves all small limits and colimits and commutes with truncations.

	\item \label{rslt:alpha-push-pull-of-compact-module-over-adic-ring} For every compact $\mathcal A$-module $P \in \D(\mathcal A)$ we have $\alpha_* \alpha^* P = P_\cpl$.
\end{lemenum}
\end{lemma}
\begin{proof}
Let $\mathcal C_{\mathcal A} \subset \D_\cpl(\mathcal A)$ be as in \cref{def:modified-modules-over-adic-ring}. Recall that there is a symmetric monoidal structure $\hat\tensor_{\mathcal A}$ on $\D_\cpl(\mathcal A)$ given by the completed tensor product.\footnote{This follows from the fact that $\underline{\Hom}(M,N)$ is adically complete whenever $N\in \D(A)$ is adically complete.} Since compact objects in $\D(\mathcal A)$ are stable under $\tensor_{\mathcal A}$ (here we use that $\calA$ is an analytic ring over $\Z_\solid$) it follows that $\mathcal C_{\mathcal A}$ is stable under $\hat\tensor_{\mathcal A}$, so that we naturally get a symmetric monoidal structure on $\mathcal C_{\mathcal A}$. By \cite[Corollary 4.8.1.14]{lurie-higher-algebra} we can uniquely extend this symmetric monoidal structure to $\D(\hat{\mathcal A}) = \Ind(\mathcal C_{\mathcal A})$ so that it preserves colimits in each argument.

The functor $\alpha^*$ is the natural functor colimit $\D(\mathcal A) = \Ind(\D(\mathcal A)^\omega) \to \Ind(\mathcal C_{\mathcal A}) = \D(\hat{\mathcal A})$ induced by the completion functor $\D(\mathcal A)^\omega \to \mathcal C_{\mathcal A}$. By \cite[Corollary 4.8.1.14]{lurie-higher-algebra} $\alpha^*$ can be upgraded uniquely to a symmetric monoidal functor. The right adjoint $\alpha_*$ exists by the adjoint functor theorem (using that we have fixed a cut-off cardinal). Since $\alpha^*$ preserves compact objects (by definition) it follows that $\alpha_*$ commutes with filtered colimits and hence with all colimits. It is also easy to see that $\alpha_*$ is conservative: Given any $M \in \D(\hat{\mathcal A})$ with $\alpha_* M = 0$, pick any compact $P \in \D(\mathcal A)$; then $0 = \Hom(P, \alpha_* M) = \Hom(\alpha^* P, M)$, so it is enough to show that the family of functors $\Hom(\alpha^* P, -)$ is conservative (with $P$ ranging over compact $\mathcal A$-modules). But this follows immediately from the fact that the $\alpha^* P$ form compact generators of $\D(\hat{\mathcal A})$ by construction.

We now prove (iv), so fix any compact $\mathcal A$-module $P$. Note that $\alpha^* P$ is adically complete: For any ideal of definition $I$ for $\underline{\mathcal A}$ and any $x \in I$ we need to check that $\varprojlim_x \alpha^* P = 0$ with $\varprojlim$ taken in $\D(\hat{\calA})$, and for this it is enough to check that $\varprojlim_x \Hom(Q, \alpha^* P) = 0$ for every $Q \in \mathcal C_{\mathcal A}$; but this statement depends only on $\mathcal C_{\mathcal A}$, which is a full subcategory of $\D_\cpl(\mathcal A)$. Thus, we have checked that $\alpha^\ast P$ is adically complete. Since $\alpha_*$ preserves all small limits, we deduce that also $\alpha_* \alpha^* P$ is adically complete, i.e. lies in $\D_\cpl(\mathcal A)$. Hence the unit $P \to \alpha_* \alpha^* P$ induces a map $P_\cpl \to \alpha_* \alpha^* P$. To show that this map is an isomorphism, it is enough to do so after applying $\Hom(Q, -)$ for any compact $Q \in \D(\mathcal A)$; we have:
\begin{align*}
	\Hom(Q, P_\cpl) = \Hom(Q_\cpl, P_\cpl) = \Hom(\alpha^* Q, \alpha^* P) = \Hom(Q, \alpha_* \alpha^* P),
\end{align*}
where in the second step we used that $\mathcal C_{\mathcal A} \subset \D(\mathcal A)$ is a full subcategory.

We have now proved everything apart from the claims about the $t$-structure. We define $\D_{\ge0}(\hat{\mathcal A}) \subset \D(\hat{\mathcal A})$ to be the full subcategory spanned by those objects $M \in \D(\hat{\mathcal A})$ such that $\alpha_* M \in \D_{\ge0}(\mathcal A)$. This subcategory is clearly stable under colimits and extensions and by (iv) it contains $\alpha^* \mathcal A[S]$ for all profinite sets $S$. It follows that the $\alpha^* \mathcal A[S]$ are compact generators of $\D_{\ge0}(\hat{\mathcal A})$, so by \cite[Proposition 1.4.4.11.(i)]{lurie-higher-algebra} this subcategory defines a $t$-structure on $\D(\hat{\mathcal A})$. Since every $M \in \D_{\ge0}(\mathcal A)$ is a (sifted) colimit of the free generators $\mathcal A[S]$ it follows easily from (iv) that $\alpha^*$ is right $t$-exact (because $(-)_\cpl$ is so). Moreover, $\alpha_*$ is right $t$-exact by construction and it is left $t$-exact because $\alpha^*$ is right $t$-exact. Finally, one checks easily that $\alpha_*$ commutes with $\tau_{\ge0}$ by passing to left adjoints; from the usual truncation triangle it then follows that $\alpha_*$ also commutes with $\tau_{\le0}$.
\end{proof}

\begin{remark}
  \label{sec:adic-rings-1-remark-on-generalized-analytic-rings}
By \cref{rslt:basic-properties-of-modified-modules} the adjunction of $\alpha^*$ and $\alpha_*$ is monadic and hence identifies $\D(\hat{\mathcal A})$ with a category of modules over some monad on $\D(\mathcal A)$. Thus one may think of $\hat{\mathcal A}$ as some form of generalized analytic ring and $\D(\hat{\mathcal A})$ as its category of modules. One may attempt to formalize this idea in order to get a general theory of generalized analytic rings, but we will not pursue this further. We note that $\alpha_\ast\colon \D(\hat{\mathcal{A}})\to \D(\mathcal{A})$ is in general not $\D(\mathcal{A})$ linear, or equivalently $\alpha_\ast$ does not satisfy the projection formula (otherwise for each compact object $P\in \mathcal{A}$ one has $P\otimes \alpha_\ast(\alpha^\ast(\underline{\mathcal{A}})\cong \alpha_\ast(\alpha^\ast P)\cong P_{\cpl}$). Thus, modules over the monad $\alpha_\ast(\alpha^\ast(-))$ do in general not identify with modules under some ring object in $\D(\mathcal{A})$.
\end{remark}

Before we continue, let us make everything more functorial. We have the following straightforward result:

\begin{lemma} \label{rslt:functoriality-of-modified-modules}
The construction $\mathcal A \mapsto \D(\hat{\mathcal A})$ defines a functor $\AdicRing^\opp \to \SymMonCat$ and $\alpha^*\colon \D(\mathcal A) \to \D(\hat{\mathcal A})$ defines a natural transformation of such functors.
\end{lemma}
\begin{proof}
The assignment $\mathcal A \mapsto \D(\mathcal A)$ defines a functor $\AdicRing^\opp \to \SymMonCat$ (see e.g. \cite[Proposition 2.3.26]{mann-mod-p-6-functors}). By looking at the associated cocartesian family of $\infty$-operads and restricting to a full subcategory of that, we can construct the functor $\mathcal A \mapsto \D_\cpl(\mathcal A)$ (see \cite[Lemma 2.2.22]{mann-mod-p-6-functors} for the general strategy of this argument). By restricting to even small full subcategories we obtain the functor which associates to each $\mathcal A$ the full subcategory $\mathcal C_{\mathcal A} \subset \D_\cpl(\mathcal A)$ from \cref{def:modified-modules-over-adic-ring}. By the remark right before \cite[Proposition 4.8.1.10]{lurie-higher-algebra} the assignment $\mathcal C \mapsto \Ind(\mathcal C)$ defines an endofunctor of symmetric monoidal $\infty$-categories. This finishes the construction of the functor $\mathcal A \mapsto \D(\hat{\mathcal A})$.

The construction of the natural transformation $\alpha^*$ works in a very similar fashion, by making everything relative over $\Delta^1$.
\end{proof}

\begin{definition}
  \label{sec:adic-rings-2-definition-completion-and-pullback-functors}
\begin{defenum}
	\item For every $\mathcal A \in \AdicRing$ we denote by $- \tensor_{\mathcal A} \hat{\mathcal A}\colon \D(\mathcal A) \to \D(\hat{\mathcal A})$ the functor $\alpha^*$ from \cref{rslt:definition-of-base-change-to-modified-modules}.

	\item For every map $\mathcal A \to \mathcal B$ in $\AdicRing$ we denote by $- \tensor_{\hat{\mathcal A}} \hat{\mathcal B}\colon \D(\hat{\mathcal A}) \to \D(\hat{\mathcal B})$ the induced functor from \cref{rslt:functoriality-of-modified-modules}.
\end{defenum}
\end{definition}

We will write $\alpha^\ast_{\mathcal{A}}, \alpha_{\mathcal{A},\ast}$ in case we need to clarify the dependence of $\alpha^\ast,\alpha_\ast$ on $\mathcal{A}$. If $\mathcal{A}\to \mathcal{B}$ is a morphism of adic analytic rings, then $\alpha_{\mathcal{B}}^\ast((-)\otimes_{\mathcal{A}}\mathcal{B})\cong \alpha^\ast_{\mathcal{A}}(-)\otimes_{\hat{\mathcal{A}}} \hat{\mathcal{B}}$ by \cref{rslt:functoriality-of-modified-modules}.

We remark that $-\otimes_{\calA} \hat{\calA}$ and $-\otimes_{\hat{\calA}}\hat{\mathcal{B}}$ are symmetric monodial. While for general adic analytic rings $\mathcal A$, $\D(\hat{\mathcal A})$ differs from $\D(\mathcal A)$, in praxis these two categories are often the same:

\begin{lemma}
  \label{sec:adic-rings-3-modified-version-agrees-with-old-one-for-finite-type-stuff}
Let $(A, A^+)_\solid$ be an adic analytic ring. If either $A$ is discrete or $A^+$ is finitely generated over $\Z$, then $\D(\hat{\mathcal A}) = \D(\mathcal A)$.
\end{lemma}
\begin{proof}
We use the notation from \cref{rslt:basic-properties-of-modified-modules}. Since $\alpha_*$ is conservative, in order to prove the desired equivalence of categories it is enough to show that $\alpha^*$ is fully faithful, i.e. that for all $M \in \D(\mathcal A)$ the unit of the adjunction $M \isoto \alpha_* \alpha^* M$ is an isomorphism. Since $\alpha_*$ and $\alpha^*$ preserve colimits, this reduces to the case that $M = P$ is compact, in which case by \cref{rslt:alpha-push-pull-of-compact-module-over-adic-ring} the claim reduces to showing that $P$ is adically complete. This is evident if $A$ is discrete, so it remains to treat the case where $A^+$ is finitely generated.

We can assume that $A^+$ is a finitely generated polynomial algebra over $\Z$ and that there is a map $A^+ \to A$. By choosing generators of an ideal of definition of $A$ we construct a map $A^+[x_\bullet] := A^+[x_1, \dots, x_n] \to A$, which then automatically factors as $A^+[[x_\bullet]] \to A$. Note that for every profinite set $S$ we have
\begin{align*}
	(A, A^+)_\solid[S] = A \tensor_{A^+_\solid} A^+_\solid[S] = A \tensor_{(A^+[[x_\bullet]],A^+)_\solid} (A^+[[x_\bullet]],A^+)_\solid[S]
\end{align*}
On the other hand, $(A^+[[x_\bullet]],A^+)_\solid[S] \isom \prod_I A^+[[x_\bullet]]$ for some set $I$ (as $A^+$ is finitely generated over $\Z$), so in particular it is $(x_1,\dots, x_n)$-adically complete. It follows from \cite[Proposition 2.12.10]{mann-mod-p-6-functors} that $(A,A^+)_\solid[S]$ is $(x_1, \dots, x_n)$-adically complete, i.e. adically complete as an $A$-module. This proves the claim.
\end{proof}

\begin{lemma} \label{rslt:A-mod-I-modified-modules-equals-usual-modules}
Let $\mathcal A$ be an adic analytic ring with ideal of definition $I$. Then $- \tensor_{\mathcal A} \hat{\mathcal A}$ induces an equivalence of categories
\begin{align*}
	\D(\mathcal A/I) = \Mod_{\underline{\mathcal A}/I}(\D(\hat{\mathcal A})).
\end{align*}
\end{lemma}
Here, the $\D(\calA/I)$ denotes analytic ring structure on $\underline{\calA}/I$ with its induced analytic ring structure along $\underline{\calA}\to \underline{\calA}/I$.
\begin{proof}
Let us write $\alpha^* := - \tensor_{\mathcal A} \hat{\mathcal A}$. We have $\D(\mathcal A/I) = \Mod_{\underline{\mathcal A}/I}(\D(\mathcal A))$, so since $\alpha^*$ is symmetric monoidal, it induces a functor $\D(\mathcal A/I) \to \Mod_{\underline{\mathcal A}/I}(\D(\hat{\mathcal A}))$. Its right adjoint is given by $\alpha_*$, which is conservative, so the desired equivalence reduces to showing that $\alpha^*$ is fully faithful on $\underline{\mathcal A}/I$-modules, i.e. that $\id \isoto \alpha_* \alpha^*$ is an isomorphism. This in turn can be checked on compact generators, i.e. for $P = (\mathcal A/I)[S]$. But such a $P$ is also compact as $\mathcal A$-module (by our conventions \cref{sec:notat-conv}), hence by \cref{rslt:alpha-push-pull-of-compact-module-over-adic-ring} we have $\alpha_*\alpha^* P = P_\cpl = P$, as desired.
\end{proof}

Next we study adic completeness in $\D(\hat{\mathcal A})$. Since this category is $\D(\mathcal A)$-linear via the symmetric monoidal functor $- \tensor_{\mathcal A} \hat{\mathcal A}$, there is a good notion of adically complete modules in $\D(\hat{\mathcal A})$:

\begin{definition}
  \label{sec:adic-rings-2-definition-completeness-in-modified-modules}
Let $\mathcal A$ be an adic analytic ring. An $\hat{\mathcal A}$-module $M \in \D(\hat{\mathcal A})$ is called \emph{adically complete} if for some (equivalently every) ideal of definition $I$ of $\underline{\mathcal A}$ and every $x \in I$ we have $\varprojlim_x M = 0$. We denote by $\D_\cpl(\hat{\mathcal A}) \subset \D(\hat{\mathcal A})$ the full subcategory spanned by the adically complete modules. Similarly, we let $\D_\cpl(\mathcal{A})\subseteq \D(\mathcal{A})$ be the full subcategory of adically complete objects in $\D(\mathcal{A})$.\footnote{We show in \cref{sec:defin-texorpdfstr-equivalence-of-complete-objects} that $\D_\cpl(\mathcal A)$ and $\D_\cpl(\hat{\mathcal A})$ agree.}
\end{definition}

The following result is a first indicator for why we work with $\hat{\mathcal A}$ in this paper: The complete objects in $\D(\hat{\mathcal A})$ behave much better with respect to pullbacks and tensor products than those in $\D(\mathcal A)$.

\begin{lemma} \label{rslt:basic-properties-of-adically-complete-modules}
Let $\mathcal A$ be an adic analytic ring.
\begin{lemenum}
	\item An $\hat{\mathcal A}$-module is adically complete if and only if the underlying $\mathcal A$-module is adically complete. All the equivalent characterizations from \cite[Lemma 2.12.4]{mann-mod-p-6-functors} apply to adically complete modules in $\D(\hat{\mathcal A})$, Most notably, an $\hat{\calA}$-module $M$ is adically complete if and only if all its homotopy objects $\pi_i(M)$, $i\in \Z$, are.

	\item The $t$-structure on $\D(\hat{\mathcal A})$ restricts to a $t$-structure on $\D_\cpl(\hat{\mathcal A})$.

	\item \label{rslt:complete-modules-stable-under-tensor} $\D^-_\cpl(\hat{\mathcal A})$ is stable under the symmetric monoidal structure $\tensor_{\hat{\mathcal A}}$ on $\D(\hat{\calA})$ from \cref{rslt:basic-properties-of-modified-modules}.

	\item \label{rslt:adic-base-change-preserves-adic-completeness} Let $\mathcal A \to \mathcal B$ be an adic map of adic analytic rings. Then $- \tensor_{\hat{\mathcal A}} \hat{\mathcal B}$ restricts to a functor $\D^-_\cpl(\hat{\mathcal A}) \to \D^-_\cpl(\hat{\mathcal B})$.
\end{lemenum}
\end{lemma}
\begin{proof}
We use the notation $\alpha^*$ and $\alpha_*$ from \cref{rslt:basic-properties-of-modified-modules}. Then (i) follows easily from \cref{rslt:properties-of-forgetful-functor-from-modified-modules} and (ii) is just a reformulation of the fact that $M$ is adically complete if and only if all $\pi_i M$ are adically complete (which is part of (i)).

We now prove (iv), so let $M \in \D^-_\cpl(\hat{\mathcal A})$ be given. We need to show that $M \tensor_{\hat{\mathcal A}} \hat{\mathcal B}$ is adically complete. For some $n \ge 0$ we can find an adic map $\mathcal A_0 := (\Z[[x_1, \dots, x_n]], \Z)_\solid \to \mathcal A$. By using the bar resolution for the monadic adjunction between base-change and forgetful functor along $\mathcal A_0 = \hat{\mathcal A_0} \to \hat{\mathcal A}$ (using \cref{sec:adic-rings-3-modified-version-agrees-with-old-one-for-finite-type-stuff}) we can write $M$ as a uniformly right-bounded geometric realizations of modules of the form $M_0 \tensor_{\mathcal A_0} \hat{\mathcal A}$ for some $M_0 \in \D(\mathcal A_0)$. Since adic completion is bounded and hence commutes with uniformly right-bounded geometric realizations, adic completeness is stable under uniformly right-bounded geometric realizations. We can thus reduce the claim to the case $\mathcal A = \mathcal A_0$. We can assume that $M$ is connective and thus write it as a geometric realization of objects of the form $\bigdsum_{i\in I} \mathcal A[S_i]$ for some profinite sets $S_i$. By using again the fact that adic completeness is stable under uniformly right-bounded geometric realizations, and also under $\omega_1$-filtered colimits, we reduce to the case that $M = \cplbigdsum{k\in\N} \mathcal A[S_k]$ for some profinite sets $S_k$. Note that it is enough to show that $- \tensor_{\mathcal A} \hat{\mathcal B}$ preserves $x$-complete objects for all $x \in I$, so from now on we work only with $x$-completions. Then we can write $M = \varinjlim_\alpha \prod_k x^{\alpha_k} \mathcal A[S_k]$, where the colimit is taken over all monotonous sequences $\alpha\colon \Z_{\ge0} \to \Z_{\ge0}$ converging to $\infty$. We now claim that the natural map
\begin{align*}
	M \tensor_{\mathcal A} \hat{\mathcal B} = \varinjlim_\alpha ((\prod_k x^{\alpha_k} \mathcal A[S_k]) \tensor_{\mathcal A} \hat{\mathcal B}) \isoto \varinjlim_\alpha \prod_k x^{\alpha_k} (\mathcal A[S_k] \tensor_{\mathcal A} \hat{\mathcal B})
\end{align*}
is an isomorphism. The argument is similar to the proof of \cite[Lemma 3.7.(ii)]{mann-nuclear-sheaves}, where the crucial input is that the constituents $(\prod_k x^{\alpha_k} \mathcal A[S_k]) \tensor_{\mathcal A} \hat{\mathcal B}$ are $x$-adically complete -- this is true because $\prod_k \mathcal A[S_k]$ is compact in $\D(\mathcal A)$ and thus $(\prod_k x^{\alpha_k} \mathcal A[S_k]) \tensor_{\mathcal A} \hat{\mathcal B}$ is compact in $\D(\hat{\mathcal{B}})$ (as the right adjoint to $-\otimes_{\calA}\hat{\mathcal{B}}$ commutes with colimits). Now compact objects in $\D(\hat{\mathcal B})$ are adically complete (by \cref{rslt:alpha-push-pull-of-compact-module-over-adic-ring}). To finish the proof of (iv) it remains to see that the natural map
\begin{align*}
	\varinjlim_\alpha \prod_k x^{\alpha_k} (\mathcal A[S_k] \tensor_{\mathcal A} \hat{\mathcal B}) \isoto \cplbigdsum{k} (\mathcal A[S_k] \tensor_{\mathcal A} \hat{\mathcal B})
\end{align*}
is an isomorphism, or equivalently that the left-hand side is $x$-adically complete. This can be checked after applying the forgetful functor to $\D(\mathcal A)$ and on homology, where it is a straightforward computation. This finishes the proof of (iv).

It remains to prove (iii), so let $M, N \in \D^-_\cpl(\hat{\mathcal A})$ be given. By the same bar resolution argument as in the proof of (iv) we can assume that $M$ and $N$ come via pullback from adically complete right-bounded $\mathcal A_0 := (\Z[[x_1, \dots, x_n]], \Z)_\solid$-modules for some $n \ge 0$. But then by (iv) the claim reduces to the case $\mathcal A = \mathcal A_0$, where it follows from \cite[Proposition 2.12.10]{mann-mod-p-6-functors}.
\end{proof}

\begin{example}
  \label{sec:adic-rings-1-example-boundedness-for-completeness-is-needed}
  The boundedness assumption in \cref{rslt:adic-base-change-preserves-adic-completeness} is necessary: If $M=N=\bigoplus\limits_{i\in \Z} \underline{\calA}[i]$, then $M\otimes_{\hat\calA}N$ is not adically complete (if $\underline{\calA}$ is not discrete), even though $M,N$ are adically complete. Namely, the tensor product contains the direct summand $\bigoplus\limits_{i\in \Z}\underline{\calA}$.

  Similarly, the preservation of completeness is wrong for the base change $-\otimes_{\calA} \mathcal{B}$ of adic analytic rings with the usual definition of $D_\solid(-)$.
\end{example}

In the presence of finite Tor dimension, the preservation of adic completeness can be strengthened to unbounded complexes:

\begin{corollary} \label{rslt:tor-dim-of-adic-maps}
Let $\mathcal A \to \mathcal B$ be an adic map of adic analytic rings such that for some ideal of definition $I$ the map $\mathcal A/I \to \mathcal B/I$ has finite Tor dimension. Then $\hat{\mathcal A} \to \hat{\mathcal B}$ has finite Tor dimension and the functor $- \tensor_{\hat{\mathcal A}} \hat{\mathcal B}$ preserves adic completeness.
\end{corollary}
\begin{proof}
We first prove the claim about Tor dimension, so let $M \in \D^\heartsuit(\hat{\mathcal A})$ be given. Then we can write $M$ as a filtered colimit of adically complete modules; by truncating we can assume that these adically complete modules lie in $\D^\heartsuit$. We can thus assume that $M$ itself is adically complete, hence the same is true for $N := M \tensor_{\hat{\mathcal A}} \hat{\mathcal B}$ by \cref{rslt:adic-base-change-preserves-adic-completeness}. But by assumption $N/I$ is bounded to the left (independent of $M$), hence so is $N$.

The second claim follows from the finite Tor dimension by writing any adically complete module as a filtered colimit over its truncations $\tau_{\ge n}$ for $n \to -\infty$.
\end{proof}

Unsurprisingly, the difference between $\mathcal A$ and $\hat{\mathcal A}$ vanishes on complete objects; more precisely we have:

\begin{lemma}
  \label{sec:defin-texorpdfstr-equivalence-of-complete-objects}
  Let $\mathcal{A}$ be an adic analytic ring. Then the functor $\alpha_\ast\colon \D(\hat{\mathcal{A}})\to \D(\mathcal{A})$ restricts to an equivalence $\D_\cpl(\hat{\mathcal{A}})\cong \D_\cpl(\mathcal{A})$.
\end{lemma}
\begin{proof}
  By \cref{rslt:basic-properties-of-adically-complete-modules} the restriction $\alpha_\ast\colon \D_\cpl(\hat{\mathcal{A}})\to \D_\cpl(\mathcal{A})$ is well-defined. By general theory of completion, it admits as a left adjoint the functor $\hat{\alpha}^\ast:=(-)_{\cpl}\circ \alpha^\ast$. As both $\alpha_\ast, \hat{\alpha}^\ast$ commute with the reduction $\mathcal{A}\to \mathcal{A}/I$ for some (finitely generated) ideal of definition $I$, the unit/counit of this adjunction are isomorphisms by \cref{rslt:A-mod-I-modified-modules-equals-usual-modules}. This finishes the proof.
\end{proof}

For any adic analytic ring $\mathcal A$, $\D(\hat{\mathcal A})$ is a compactly generated closed symmetric monoidal category with compact unit object $\underline{\mathcal A}$, hence the definition of nuclear objects from \cite[Lecture VIII]{condensed-complex-geometry} applies:

\begin{definition}
We denote by $\D(\hat{\mathcal A})^\nuc \subset \D(\hat{\mathcal A})$ the full subcategory spanned by the nuclear objects, as defined in \cite[Definition 8.5]{condensed-complex-geometry}.
\end{definition}

\begin{proposition}
  \label{sec:defin-d_hats-1-equivalence-of-nuclear-subcategory}
For every adic analytic ring $\mathcal A$, the functor $- \tensor_{\mathcal A} \hat{\mathcal A}$ induces an equivalence of symmetric monoidal categories
\begin{align*}
	\D(\mathcal A)^\nuc = \D(\hat{\mathcal A})^\nuc.
\end{align*}
\end{proposition}
\begin{proof}
We use the notation $\alpha^*=(-)\tensor_{\mathcal{A}}\hat{\mathcal{A}}$ and $\alpha_*$ from \cref{rslt:basic-properties-of-modified-modules}. We first show that $\alpha_*$ preserves nuclear modules (for $\alpha^\ast$ this follows by symmetric monoidality). Since $\alpha_*$ preserves colimits, this boils down to showing that for any compact objects $P^\prime, Q^\prime \in \D(\hat{\mathcal A})$ and any trace-class map $f\colon P^\prime \to Q^\prime$ in $\D(\hat{\mathcal A})$, the induced map $\alpha_* f\colon \alpha_* P^\prime \to \alpha_* Q^\prime$ is trace-class. This follows if we can show that the natural map
\begin{align*}
	(\alpha_* P^\prime)^\vee \tensor \alpha_* Q^\prime \isoto \alpha_* ((P^\prime)^\vee \tensor Q^\prime)
\end{align*}
is an isomorphism in $\D(\mathcal A)$. As this claim is stable under finite colimits and retracts in $P^\prime, Q^\prime$ we may assume that $P^\prime=\alpha^\ast P, Q^\prime=\alpha^\ast Q$ for two compact objects $P,Q\in \D(\mathcal{A})$. 
By \cref{rslt:basic-properties-of-modified-modules} $\alpha^\ast \alpha_\ast \alpha^\ast P\cong \alpha^\ast P$ as $P\in \D(\mathcal{A})$ is compact. Using adjunctions and $\alpha_\ast \alpha^\ast \underline{\mathcal{A}}\cong \underline{\mathcal{A}}$ shows $(\alpha_\ast\alpha^\ast P)^\vee\cong P^\vee$; in particular this object is right bounded and adically complete.\footnote{Here, we are using that $\mathcal{A}$ is living over $\Z_\solid$, and that the compact projective generators for $\Z_\solid$ are internally projective to obtain the right-boundedness of $P^\vee$.} Similarly, $(\alpha^\ast P)^\vee$ is right-bounded because $\alpha_\ast((\alpha^\ast P)^\vee)\cong P^\vee$ by adjunctions and symmetric monoidality of $\alpha^\ast$ while $P^\vee$ is right-bounded. It follows from \cref{rslt:complete-modules-stable-under-tensor} that the right-hand side of the above claimed isomorphism is adically complete, and it follows from \cite[Proposition 2.12.10]{mann-mod-p-6-functors} that the left-hand side is adically complete (here we use that $(\alpha_* \alpha^* P)^\vee = P^\vee$ is discrete modulo any ideal of definition as can be checked directly). Hence the above isomorphism can be checked modulo any ideal of definition of $\underline{\mathcal A}$, where it follows from \cref{rslt:A-mod-I-modified-modules-equals-usual-modules}.

We have established that $\alpha_*$ restricts to a functor $\D(\hat{\mathcal A})^\nuc \to \D(\mathcal A)^\nuc$, which is automatically conservative. Thus in order to prove the claimed equivalence of categories it is now enough to show that $\alpha^*$ is fully faithful on nuclear modules, i.e. for any nuclear $M$ the natural map $M \isoto \alpha_* \alpha^* M$ is an isomorphism. Since $\alpha_* \alpha^*$ preserves colimits, this reduces to showing the following: Let $(P_n)_n$ be a sequence of compact objects in $\D(\mathcal A)$ with trace-class transition maps; then $\varinjlim_n P_n = \varinjlim_n (P_n)_\cpl$. To prove this, it suffices to check that the map $P_n \to P_{n+1}$ factors as $P_n \to (P_n)_\cpl \to P_{n+1}$. But this follows easily from the definition of trace-class maps by observing that $(P_n)^\vee = ((P_n)_\cpl)^\vee$ as $\underline{\mathcal{A}}$ is complete.
\end{proof}

The following result provides a concrete description of nuclear $\underline{\mathcal{A}}$-modules. Namely, they are exactly the ones that can be written as colimits of ``Banach'' modules:

\begin{lemma}
  \label{sec:defin-d_hats-1-examples-for-nuclear-modules}
  Let $\mathcal{A}$ be an adic analytic ring with ideal of definition $I$.
  \begin{lemenum}
  \item\label{sec:defin-d_hats-2-complete-and-discrete-mod-i-implies-nuclear} If $M\in \D(\mathcal{A})$ is adically complete and $M/I$ is discrete, then $M$ is nuclear.
    \item The category $\D^\nuc(\hat{\mathcal{A}})=\D^\nuc(\mathcal{A})$ is generated under colimits by adically complete objects $M\in \D(\mathcal{A})$ with $M/I$ discrete.
  \end{lemenum}
\end{lemma}
\begin{proof}
  Let $M$ as in (i). Writing $M$ as a colimit of its truncations $\tau_{\geq n}M, n\to -\infty$, we can reduce to the case that $M$ is right-bounded. Let $P\in \D(\mathcal{A})$ be compact. We need to see that the map
  \[
    P^\vee\otimes M\to \underline{\Hom}(P,M)
  \]
  is an isomorphism. The right hand side is adically complete, because $M$ is, and the left hand side is adically complete by \cite[Proposition 2.12.10.(ii)]{mann-mod-p-6-functors} as $P^\vee=(P_{\mathrm{cpl}})^\vee$ and $M$ are adically complete, and $M/I$ is discrete. Hence, it suffices to check the claim modulo $I$, where the claim follows from discreteness of $M/I$ as in \cite[Proposition 2.9.7.(i)]{mann-nuclear-sheaves}.

  Now we show (ii). Let $M\in \D(\mathcal{A})$ be nuclear. By \cite[Theorem 8.6.(2)]{condensed-complex-geometry} we may assume that $M$ is basic nuclear. Then it is sufficient (by the last part of the proof of \cref{sec:defin-d_hats-1-equivalence-of-nuclear-subcategory}) to show the following: If $f\colon P_1\to P_2$ is any trace class morphism between compact objects in $\D(\mathcal{A})$, then $f_{\cpl}\colon P_{1,\cpl}\to P_{2,\cpl}$ factors over the completion $Q$ of $\underline{P_{2,\cpl}(\ast)}$. Here, $\underline{(-)}$ denotes the functor from $\underline{\mathcal{A}}(\ast)$-modules to $\D(\mathcal{A})$, which is left adjoint to the functor $N\mapsto N(\ast)$. This last claim follows if we can show
  \[
    \Hom(P_1,Q)\cong (P_1^\vee\otimes Q)(\ast)\cong (P_1^\vee\otimes P_{2,\cpl})(\ast), 
  \]
  with the first isomorphism implied by nuclearity of $Q$ as shown in (i). As the functor $(-)(\ast)$ preserves $I$-adic completeness both sides are adically complete as $A(\ast)$-modules, by arguments as in \cref{sec:defin-d_hats-1-equivalence-of-nuclear-subcategory}. Hence, it suffices to check the statement modulo $I$. But $P_1^\vee$ is discrete modulo $I$, and more generally for any discrete $\mathcal{A}/I$-module $D$ the map
  \[
    (D\otimes Q)(\ast)\cong (D\otimes P_{2,\cpl})(\ast)
  \]
  is an isomorphism. Indeed, this statement commutes with colimits in $D$ and is clear in the case $D=\mathcal{A}/I$. This finishes the proof.
\end{proof}

As the full subcategory $\D(\hat{\mathcal{A}})^\nuc\subseteq \D(\hat{\mathcal{A}})$ is stable under all colimits, it admits a right adjoint $(-)_\nuc\colon \D(\hat{\mathcal{A}})\to \D(\hat{\mathcal{A}})^\nuc$. The next result provides an explicit description of $(-)_\nuc$, similar to \cite[Proposition 3.12]{mann-nuclear-sheaves}. Note that it is not clear to us how to explicitly compute the analogous nuclearization functor for $\mathcal A$ in place of $\hat{\mathcal A}$.

\begin{proposition}
  \label{sec:defin-d_hats-1-nuclearization}
  Let $\mathcal{A}$ be an adic analytic ring.
  \begin{propenum}
  \item $(-)_\nuc\colon \D(\hat{\mathcal{A}})\to \D(\hat{\mathcal{A}})^\nuc$ preserves all colimits.
  \item If $P\in \D(\hat{\mathcal{A}})$ is compact, then $P_\nuc$ is naturally isomorphic to the completion of $\alpha^\ast(\underline{\alpha_\ast P(\ast)})$.
  \end{propenum}
\end{proposition}
\begin{proof}
  Sending $P\in \D(\hat{\mathcal{A}})$ compact to the completion $\alpha^\ast(\underline{\alpha_\ast P(\ast)})^\wedge$ extends uniquely to a colimit preserving functor
  \[
    F\colon \D(\hat{\mathcal{A}})\to \D(\hat{\mathcal{A}})^\nuc
  \]
  (using \cref{sec:defin-d_hats-2-complete-and-discrete-mod-i-implies-nuclear} to see that value is indeed nuclear). Let $\iota\colon \D(\mathcal{A})^\nuc\to \D(\mathcal{A})$ be the fully faithful inclusion. As the compact generators of $\D(\hat{\mathcal{A}})$ are complete (by \cref{rslt:basic-properties-of-adically-complete-modules}) there exists a natural transformation, $\iota\circ F\to \mathrm{Id}$. Indeed, it suffices to construct it on compact objects, and there it suffices to construct a natural map $\alpha^\ast(\underline{\alpha_\ast P(\ast)})\to P$, which exists as $\alpha^\ast\circ \underline{(-)}$ is left adjoint to $(\alpha_\ast(-))(\ast)$. Applying $(-)_\nuc$ to this transformation yields a natural transformation $F\to (-)_\nuc$. Let $\mathcal{L}=\varinjlim\limits_{i\in J} \mathcal{L}_i$ in $\D(\hat{\mathcal{A}})$ with $\mathcal{L}_i$ compact, and let $\mathcal{M}\in \D(\hat{\mathcal{A}})$. Then
  \[
    \Hom_{\D(\hat{\mathcal{A}})}(\mathcal{L},\iota\circ F(\mathcal{M}))=\varprojlim\limits_{i\in J}\Hom_{\D(\hat{\mathcal{A}})}(\mathcal{L}_i,\iota\circ F(\mathcal{M}))=\varprojlim\limits_{i\in J}(\mathcal{L}_i^\vee\otimes \iota\circ F(\mathcal{M}))(\ast)
  \]
  using nuclearity of $\iota\circ F(M)$ (and the notations $(-)(\ast)=\Hom_{\D(\hat{\mathcal{A}})}(\underline{\mathcal{A}},-), (-)^\vee=\underline{\Hom}_{\D(\hat{\mathcal{A}})}(-,\underline{\mathcal{A}})$). On the other hand,
  \[
    \Hom_{\D(\hat{\mathcal{A}})}(\mathcal{L},\iota(\mathcal{M}_\nuc))=\varprojlim\limits_{i\in J}\Hom_{\D(\hat{\mathcal{A}})}(\mathcal{L}_i,\mathcal{M}).
  \]
  We have to see that both inverse limits agree if $\mathcal{L}$ is nuclear, or even basic nuclear, i.e., we can assume that $\mathcal{L}=\varinjlim\limits_{n\in \mathbb{N}} \mathcal{L}_n$ with $\mathcal{L}_n$ compact and the transition maps trace-class. This assumption implies that each map
  \[
    \Hom_{\D(\hat{\mathcal{A}})}(\mathcal{L}_{n+1},\mathcal{M})\to \Hom_{\D(\hat{\mathcal{A}})}(\mathcal{L}_n,\mathcal{M})
  \]
  factors over the space of trace class maps $(\mathcal{L}_n^\vee\otimes \mathcal{M})(\ast)$ from $\mathcal{L}_n$ to $\mathcal{M}$ (\cite[Lemma 8.2.(3)]{condensed-complex-geometry}). Hence, it suffices to show that, similarly to the last assertion in the proof of \cref{sec:defin-d_hats-2-complete-and-discrete-mod-i-implies-nuclear}, the map
  \[
    (P^\vee\otimes \iota\circ F(\mathcal{M}))(\ast)\to (P^\vee\otimes M)(\ast)
  \]
  is an isomorphism for any $P, M\in \D(\hat{\mathcal{A}})$ with $P$ compact. As both sides commute with colimits in $M$ we may assume that $M$ is compact. Then both sides are $I$-adically complete over $A(\ast)$ for some ideal of definition $I\subseteq A(\ast)$, which reduces us to the case that $I=\{0\}$ by passing to the reduction mod $I$. In this case $\alpha^\ast,\alpha_\ast$ are inverse equivalences (by \cref{rslt:A-mod-I-modified-modules-equals-usual-modules}), and the assertion follows from the last assertion in the proof of \cref{sec:defin-d_hats-2-complete-and-discrete-mod-i-implies-nuclear}.
\end{proof}

From \cref{sec:defin-d_hats-1-equivalence-of-nuclear-subcategory} and \cref{sec:defin-d_hats-1-examples-for-nuclear-modules} we can produce a weaker form of a projection formula for $\alpha_\ast$.

\begin{corollary}
  \label{sec:defin-d_hats-1-projection-formula-for-nuclear}
  Let $\mathcal{A}$ be an adic analytic ring. Let $M\in \D(\mathcal{A})^\nuc$. Then the natural map
  \[
    M\otimes \alpha_{\ast}(N)\to \alpha_\ast(\alpha^\ast M\otimes N)
  \]
  is an isomorphism for any $N\in \D(\hat{\mathcal{A}})$.
\end{corollary}
\begin{proof}
  Both sides commute with colimits in $N$. Hence, we may assume that $N\in \D(\hat{\mathcal{A}})$ is compact (and hence complete by \cref{rslt:basic-properties-of-modified-modules}). As $M\cong M_\nuc$ and the nuclearization commutes with colimits, and sends compacts of bounded complete objects, which are discrete modulo some ideal of definition $I$, we may assume that $M$ is complete and discrete modulo $I$. Then the left hand side is complete by \cite[Proposition 2.12.10.(ii)]{mann-mod-p-6-functors}, and it is by reduction to the discrete case enough to show that the right hand side is complete, too.  By \cref{rslt:basic-properties-of-adically-complete-modules} it is sufficient to show that $\alpha^\ast M$ is complete (as $N$ is complete and right bounded). Now, $\alpha^\ast M\in \D(\hat{\mathcal{A}})$ is nuclear, which implies $\alpha_\ast(\alpha^\ast M)\cong M$ by \cref{sec:defin-d_hats-1-equivalence-of-nuclear-subcategory}. In particular, $\alpha^\ast M$ is complete by \cref{rslt:basic-properties-of-adically-complete-modules}. 
\end{proof}

We now prove a relative variant of \cref{sec:adic-rings-3-modified-version-agrees-with-old-one-for-finite-type-stuff}. We use the terminology that a morphism $\mathcal{A}=(A,A^+)_\solid\to \mathcal{B}=(B,B^+)_\solid$ of adic analytic rings is of $+$-finite type, if $(B,B^+)_\solid\cong (B,A\cup \{S\})_\solid$ for some finite set $S$.

\begin{proposition}
  \label{sec:defin-d_hats-1-modified-version-in-relative-+-finite-type-case}
  Let $f\colon \mathcal{A}=(A,A^+)_\solid\to \mathcal{B}=(B,B^+)_\solid$ be an adic morphism of adic analytic rings of $+$-finite type.
  \begin{propenum}
  \item The diagram
\[\begin{tikzcd}
	{\D(\hat{\mathcal{A}})} & {\D(\hat{\mathcal{B}})} \\
	{\D(\mathcal{A})} & {\D(\mathcal{B})}
	\arrow["{-\otimes_{\hat{\mathcal{A}}}\hat{\mathcal{B}}}", from=1-1, to=1-2]
	\arrow["{\alpha_{\mathcal{A},\ast}}"', from=1-1, to=2-1]
	\arrow["{-\otimes_{\mathcal{A}}\mathcal{B}}", from=2-1, to=2-2]
	\arrow["{\alpha_{\mathcal{B},\ast}}", from=1-2, to=2-2]
      \end{tikzcd}\]
    naturally commutes, i.e., base change holds for $\alpha$.
  \item The natural functor
    \[
      \D(\hat{\mathcal{A}})\otimes_{\D(\mathcal{A})}\D(\mathcal{B})\isoto \D(\hat{\mathcal{B}})
    \]
    is an equivalence.
  \end{propenum}
\end{proposition}
\begin{proof}
  For (i) we have to see that the natural morphism
  \[
    \mathcal{B}\otimes_{\mathcal{A}}\alpha_{\mathcal{A},\ast}(M)\to \alpha_{\mathcal{B},\ast}(\hat{\mathcal{B}}\otimes_{\hat{\mathcal{A}}}M) 
  \]
  is an isomorphism for any $M\in \D(\hat{\mathcal{A}})$. Both sides commute with colimits in $M$, which reduces to the case $M=\alpha^\ast_{\mathcal{A}}P$ for some compact object $P\in \D(\mathcal{A})$. Then the left hand side equates to $\mathcal{B}\otimes_{\mathcal{A}}P_\cpl$ by \cref{rslt:basic-properties-of-modified-modules} while the right hand side is
  \[
    \alpha_{\mathcal{B},\ast}(\hat{\mathcal{B}}\otimes_{\hat{\mathcal{A}}}\alpha_A^\ast(P))\cong \alpha_{\mathcal{B},\ast}(\alpha_{\mathcal{B}}^\ast(\mathcal{B}\otimes_{\mathcal{A}}P)\cong (\mathcal{B}\otimes_{\mathcal{A}}P)_\cpl
  \]
  by \cref{rslt:functoriality-of-modified-modules} and \cref{rslt:basic-properties-of-modified-modules}, and this object is complete. Hence, it is sufficient to show that $\mathcal{B}\otimes_{\mathcal{A}} P_\cpl$ is complete because then one can check the statement modulo some ideal of definition, where it is clear. If $S=\emptyset$, then the statement follows from \cite[Proposition 2.12.10]{mann-mod-p-6-functors}. Using induction one reduces to the case that $S=\{s\}$. Then
  \[
    \mathcal{B}\otimes_{\mathcal{A}} N=\underline{\Hom}_{\Z[T]}(\Z((T^{-1}))/\Z[T][-1],\underline{\mathcal{B}}\otimes_{\mathcal{A}}N)
  \]
  for any $N\in \D(\mathcal{A})$ by the right adjoint assertion to  \cite[Observation 8.11]{condensed-mathematics}. Here, $\Z[T]\to \underline{\mathcal{B}}$ is the map classified by $s$. Again using \cite[Proposition 2.12.10]{mann-mod-p-6-functors} we can conclude that
  \[
    \begin{matrix}
      \mathcal{B}\otimes_{\mathcal{A}} N_\cpl & \cong & \underline{\Hom}_{\Z[T]}(\Z((T^{-1}))/\Z[T][-1],\underline{\mathcal{B}}\otimes_{\mathcal{A}}N_\cpl) \\
      & \cong & \underline{\Hom}_{\Z[T]}(\Z((T^{-1}))/\Z[T][-1],(\underline{\mathcal{B}}\otimes_{\mathcal{A}}N)_\cpl) \\
      & \cong & (\mathcal{B}\otimes_{\mathcal{A}}N)_\cpl    
    \end{matrix}
\]
for any $N\in \D(\mathcal{A})$.

Now we prove (ii). The natural functor
\[
  \Phi\colon \D(\hat{\mathcal{A}})\otimes_{\D(\mathcal{A})}\D(\mathcal{B})\to \D(\hat{\mathcal{B}})
\]
commutes with colimits (by definition), and hence has a right adjoint $\Psi$. As the image of the functor $\alpha_{\mathcal{B}}^\ast\colon \D(\mathcal{B})\to \D(\hat{\mathcal{B}})$ generates the target under colimits, the same is true for $\Phi$. This implies formally that the functor $\Psi$ is conservative. Hence, it is sufficient to show that $\Phi$ is fully faithful. We may reduce to the cases $B^+=A^+$ or $B=A$ as the statement is stable under composition of $+$-finite morphisms. We first deal with the case $B^+=A^+$. Then $\D(\mathcal{B})\cong \mathrm{Mod}_B(\D(\mathcal{A})$, and hence $\D(\hat{\mathcal{A}})\otimes_{\D(\mathcal{A})}\D(\mathcal{B})\cong \mathrm{Mod}_{\alpha^\ast_A B}(\D(\hat{\mathcal{A}})$. To get fully faithfulness of $\Phi$ it suffices to see that
\[
  \mathrm{Hom}_{\D(B)}((B\otimes_{A}P)_\cpl, (B\otimes_{A}Q)_\cpl)\cong \mathrm{Hom}_{\D(\hat{\mathcal{A}}), \alpha^\ast_{A} B}(\alpha^\ast _AB\otimes \alpha^\ast P, \alpha^\ast_A B\otimes \alpha^\ast Q) 
\]
for any $P,Q\in \D(\mathcal{A})$ compact. The left hand side simplifies to $\Hom_{\D(A)}(P, (B\otimes_A Q)_\cpl)$, while the right hand side is by adjunctions isomorphic to $\Hom_{\D(A)}(P,\alpha_{A,\ast}(\alpha^\ast_A B\otimes \alpha^\ast_A(Q)))$. Now, $(B\otimes_A Q)_\cpl\cong B\otimes_A Q_\cpl$ by \cite[Proposition 2.12.10.(ii)]{mann-mod-p-6-functors}, which agrees with $\alpha_{A,\ast}(\alpha^\ast_A B\otimes \alpha^\ast_A(Q))\cong B\otimes \alpha_{A,\ast}\alpha^\ast_A(Q)$ by \cref{sec:defin-d_hats-1-projection-formula-for-nuclear} and \cref{rslt:basic-properties-of-modified-modules}. This finishes the proof in the case that $B^+=A^+$.
Hence, assume that $B=A$. As $f\colon \mathcal{A}\to \mathcal{B}$ is of $+$-finite type this implies that the functor $f^\ast:=\mathcal{B}\otimes_{\mathcal{A}}(-)\to \D(\mathcal{A})\to \D(\mathcal{B})$ admits a fully faithful left adjoint $f_!$ satisfying the projection formula. Indeed, it is sufficient to treat the case $S=\{s\}$, where it follows by base change from the discussion in \cite[Lecture 8]{condensed-mathematics}.
We first check that $\hat{f}^\ast:=\hat{\mathcal{B}}\otimes_{\hat{\mathcal{A}}}(-)\colon \D(\hat{\mathcal{A}})\to \D(\hat{\mathcal{B}})$ commutes with products, so that $\hat{f}^\ast$ admits a left adjoint $\hat{f}_!$. As $\alpha_{B,\ast}$ is conservative, it is sufficient to show that $\alpha_{B,\ast}\hat{f}^\ast\cong f^\ast\alpha_{A,\ast}$ (by the proven assertion (i)) commutes with products. This is clear as $f^\ast, \alpha_{A,\ast}$ are right adjoints. Applying (i) again shows that $\hat{f}_!\alpha_{A}^\ast\cong \alpha_B^\ast f_!$. As the essential image of $\alpha_B^\ast$ generates $\D(\hat{\mathcal{B}})$ under colimits this implies that the natural transformation
\[
  \mathrm{Id}\to \hat{f}^\ast\hat{f}_!
\]
is an isomorphism because $f^\ast f_!\cong \mathrm{Id}$. Similarly, one checks that $\hat{f}_!$ satisfies the projection formula. This implies that $\hat{f}^\ast$ is the open localization associated with the idempotent algebra $\mathrm{cone}(\hat{f}_!(1)\to 1)\in \D(\hat{\mathcal{A}})$ (\cite[Proposition 6.5]{condensed-complex-geometry}). However, this algebra is the pullback of the idempotent algebra $\mathrm{cone}(f_!(1)\to 1)$ because $\alpha^\ast_B f_!\cong \hat{f}_! \alpha^\ast_A$. By \cref{sec:glob-stably-unif-2-closed-open-immersions-are-stable-under-base-change-for-lurie-tensor-product} this implies the claim.
\end{proof}

\subsection{Complete descent}
\label{sec:complete-descent}

Next we discuss base-change in the setting of adic rings, in particular with our new modified version of modules. As we will show below, base-change holds in great generality as long as the involved maps are adic.

\begin{lemma}
Let $(A, A^+)_\solid$ be an adic analytic ring.
\begin{lemenum}
	\item \label{rslt:adic-add-element-to-+-ring} For every element $a \in \pi_0 A$ with induced map $\Z[x] \to A$ we have
	\begin{align*}
		(A, A^+)_\solid \tensor_{(\Z[x], \Z)_\solid} \Z[x]_\solid = (A, A^+[a])_\solid.
	\end{align*}

	\item \label{rslt:adic-polynomial-ring} We have
	\begin{align*}
		(A, A^+)_\solid \tensor_{\Z_\solid} \Z[x]_\solid = (A\langle x \rangle, A^+[x])_\solid.
	\end{align*}
\end{lemenum}
\end{lemma}
\begin{proof}
Part (i) follows easily from the definitions by using that $A$ is solid over any discrete ring (because it is a limit of discrete rings). For (ii) we note that on modules the functor $- \tensor_{\Z_\solid} \Z[x]_\solid$ preserves limits (see \cite[Lemma 2.9.5]{mann-mod-p-6-functors}), so in particular it preserves adic completeness. Then the claim reduces to the observations that the right-hand side is an analytic ring and that the statement is true modulo $I$.
\end{proof}

For adic analytic rings we have the following characterization of steady maps (see \cite[Definition 12.13]{scholze-analytic-spaces}) in terms of adicness:

\begin{lemma}
  \label{sec:complete-descent-2-steady-maps-of-adic-analytic-rings-are-exactly-the-adic-ones}
A map of adic analytic rings is steady if and only if it is adic.
\end{lemma}
\begin{proof}
Let $f\colon (A, A^+)_\solid \to (B, B^+)_\solid$ be a map of adic analytic rings. First assume that $f$ is adic. In order to show that $f$ is steady, it is enough to show that $(A, A^+)_\solid \to (B, A^+)_\solid$ and $(B, A^+)_\solid \to (B, B^+)$ are steady. For the second map this follows from \cref{rslt:adic-add-element-to-+-ring} and \cite[Proposition 2.9.7.(ii)]{mann-mod-p-6-functors} by stability of steadiness under base change. The first map is an induced analytic ring structure, so the statement reduces to showing that $B$ is nuclear in $\D_\solid(A, A^+)$ (see \cite[Corollary 2.3.23]{mann-mod-p-6-functors}), which follows directly from \cref{sec:defin-d_hats-2-complete-and-discrete-mod-i-implies-nuclear} by adicness of $f$.
We now prove the converse, so assume that $f$ is steady and let $I$ be an ideal of definition for $A$. Consider the map $g\colon (A, A^+) \to (A\langle x \rangle, A^+[x])$. Then by \cref{rslt:adic-polynomial-ring} the base-change of $g$ along $f$ is given by the map $g'\colon (B, B^+) \to (B\langle x \rangle, B^+[x])$. Note that ``$\langle x \rangle$'' means adically completed polynomials with respect to the adic topology on $A$ respectively $B$. By steadiness of $f$ we have $B\langle x \rangle = A\langle x\rangle \tensor_{(A, A^+)_\solid} (B, B^+)_\solid$. Now $B$ is adically complete as an $A$-module. Indeed, to see this we may assume that $B$ is discrete, in which case it is an $A/I$-module for some ideal of definition $I\subseteq \pi_0A(\ast)$ by nuclearity of $B$ as an $A$-module, and we have shown the completeness of $B$ as an $A$-module. Hence by \cite[Proposition 2.12.10]{mann-mod-p-6-functors} $A\langle x\rangle \tensor_{(A, A^+)_\solid} (B, B^+)_\solid$ computes the $I$-adic completion of $B[x]$ (note that this tensor product does not depend on $B^+$ because the $I$-adic completion of $B[x]$ is solid for any choice of $B^+$). For this $I$-adic completion to be the same as the $J$-adic completion of $B[x]$ for some ideal of definition $J$ of $B$ we must have that $f$ is adic. Indeed, base change to $\mathcal{A}/I$ reduces to the case $I=\{0\}$, and then $B[x]\cong \bigoplus_{n\geq 0} B\cdot x^{n}$ can be $J$-adically complete if and only if $J$ is nilpotent.\footnote{Let $b\in J$. Then $\pi_0(B[x])$ must contain some element $x$ whose coefficient in front of $x^n$ is $b^n$ as follows by the universal case of the ring $\Z[[b]]$. This forces $b$ to be nilpotent, and thus $J$ as well because $J$ is finitely generated.}
\end{proof}

\begin{lemma} \label{rslt:pushout-of-adic-analytic-rings}
Let $(B, B^+)_\solid \from (A, A^+)_\solid \to (C, C^+)_\solid$ be a diagram of adic analytic rings and adic maps. Then
\begin{align*}
	(B, B^+)_\solid \tensor_{(A, A^+)_\solid} (C, C^+)_\solid = (B \hat\tensor_A C, B^+ \tensor_{A^+} C^+)_\solid. 
\end{align*}
In particular this tensor product is still adic.
\end{lemma}
\begin{proof}
We can reduce to the cases $A^+ = B^+$ and $A = B$. The second case follows easily from \cref{rslt:adic-add-element-to-+-ring}. The first case follows from \cite[Proposition 2.12.10]{mann-mod-p-6-functors}.
\end{proof}

\begin{remark}
  \label{sec:complete-descent-1-definition-completed-tensor-product}
  In general, the single notation $\hat{\mathcal{A}}$ is not well-defined, but should be seen as a convenient replacement for the datum $(\mathcal{A},\D(\hat{\mathcal{A}}))$. As an example, given a diagram $\mathcal B \from \mathcal A \to \mathcal C$ of adic analytic rings and adic maps, we denote $\hat{\mathcal B} \tensor_{\hat{\mathcal A}} \hat{\mathcal C} := (\mathcal B \tensor_{\mathcal A} \mathcal C)\cplt$, i.e., it denotes the datum $(\mathcal{B}\tensor_{\mathcal{A}}\mathcal{C}, \D((\mathcal B \tensor_{\mathcal A} \mathcal C)\cplt))$.
\end{remark}

\begin{corollary} \label{rslt:steadiness-for-modified-modules}
Let $\mathcal B \from \mathcal A \to \mathcal C$ be a diagram of adic analytic rings and adic maps. Then for every $M \in \D(\hat{\mathcal B})$ the natural morphism
\begin{align*}
	M \tensor_{\hat{\mathcal A}} \hat{\mathcal C} \isoto M \tensor_{\hat{\mathcal B}} (\hat{\mathcal B} \tensor_{\hat{\mathcal A}} \hat{\mathcal C})
\end{align*}
is an isomorphism, i.e. the following base-change diagram commutes:
\begin{center}\begin{tikzcd}
	\D(\hat{\mathcal B} \tensor_{\hat{\mathcal A}} \hat{\mathcal C}) \arrow[d] & \D(\hat{\mathcal B}) \arrow[l] \arrow[d]\\
	\D(\hat{\mathcal C}) & \D(\hat{\mathcal A}) \arrow[l]
\end{tikzcd}\end{center}
\end{corollary}
\begin{proof}
  The existence of the natural morphism is formally implied by adjunctions as
  \[
    (-\tensor_{\hat{\mathcal{A}}} \hat{\mathcal{B}})\tensor_{\hat{\mathcal{B}}} (\hat{\mathcal{B}}\tensor_{\hat{\mathcal{A}}}\hat{\mathcal{C}})\cong (-\tensor_{\hat{\mathcal{A}}} \hat{\mathcal{C}})\tensor_{\hat{\mathcal{C}}} (\hat{\mathcal{B}}\tensor_{\hat{\mathcal{A}}}\hat{\mathcal{C}})
   \] by \cref{rslt:functoriality-of-modified-modules}.
Both sides of the claimed isomorphism commute with all colimits, so we can w.l.o.g. assume that $M$ is compact. In particular $M$ is right-bounded and adically complete, hence by \cref{rslt:adic-base-change-preserves-adic-completeness} both sides of the claimed isomorphism are adically complete. But then we can check the claimed isomorphism modulo any ideal of definition, so by \cref{rslt:A-mod-I-modified-modules-equals-usual-modules} we reduce to the well-known claim that maps of discrete Huber pairs are steady (see \cite[Proposition 2.9.7.(ii)]{mann-mod-p-6-functors}).
\end{proof}

We are finally in the position to discuss descent in the setting of $\hat{\mathcal A}$-modules. As promised we will see that descendability can be checked modulo all powers of an ideal of definition.

\begin{definition}
A map $\mathcal A \to \mathcal B$ of adic analytic rings is called \emph{adically descendable of index $\le d$} if it is adic and for some ideal of definition $I$ and all $n \ge 1$ the map $\mathcal A/I^n \to \mathcal B/I^n$ is descendable of index $\le d$ in the sense of \cite[Definition 2.6.7]{mann-mod-p-6-functors}.
\end{definition}

\begin{remark}
If $\mathcal A \to \mathcal B$ is a descendable morphism of adic analytic rings then it is in particular adically descendable because descendability is stable under the base-change along $\mathcal A \to \mathcal A/I^n$ (see \cite[Lemma 2.6.9]{mann-mod-p-6-functors}). The converse seems to be wrong unless the compact objects in $\D(\mathcal A)$ are adically complete. This is fixed by $\D(\hat{\mathcal A})$, as the following result shows.
\end{remark}

\begin{example}
  \label{sec:complete-descent-1-example-for-adically-descendable-maps}
  A typical example for an adically descendable map $\mathcal{A} \to \mathcal B$ of index $\leq 2$ is an adic $I$-completely faithfully flat and $I$-completely finitely presented map, i.e., $I$ is an ideal of definition of $\mathcal A$ and $\underline{\mathcal{A}}/I^n\to \underline{\mathcal{B}}/I^n$ is faithfully flat and finitely presented (on $\pi_0$), cf.\ \cite[Lemma 2.10.6]{mann-mod-p-6-functors} for the static case, and \cite[Theorem 4.15]{mikami2023fppfdescent} for the general case. 
\end{example}

\begin{theorem} \label{rslt:descendability-of-adic-rings}
Let $\mathcal A \to \mathcal B$ be an adically descendable morphism of adic analytic rings. Then modules descend along $\hat{\mathcal A} \to \hat{\mathcal B}$, i.e. the functor
\begin{align*}
	\D(\hat{\mathcal A}) \isoto \varprojlim_{n\in\Delta} \D(\hat{\mathcal B} \tensor_{\hat{\mathcal A}} \dots \tensor_{\hat{\mathcal A}} \hat{\mathcal B})
\end{align*}
is an equivalence of categories.
\end{theorem}
\begin{proof}
Let $\End^L(\D(\hat{\mathcal A}))$ denote the category of $\D(\hat{\mathcal A})$-enriched colimit-preserving functors $\D(\hat{\mathcal A}) \to \D(\hat{\mathcal A})$ (cf. \cite[\S2.5, \S A.4]{mann-mod-p-6-functors}). It comes equipped with the composition monoidal structure and there is an embedding $\eta\colon \D(\hat{\mathcal A}) \injto \End^L(\D(\hat{\mathcal A}))$ via $M \mapsto - \tensor_{\hat{\mathcal A}} M$. As in \cite[Lemma 2.5.6]{mann-mod-p-6-functors} the map $f\colon \mathcal A \to \mathcal B$ induces a pair of adjoint functors
\begin{align*}
	f^\natural\colon \End^L(\D(\hat{\mathcal A})) \rightleftarrows \End^L(\D(\hat{\mathcal B})) \noloc f_\natural,
\end{align*}
where $f_\natural$ acts on underlying functors as $F \mapsto f_* \comp F \comp f^*$; here $f_*$ denotes the forgetful functor and $f^* := - \tensor_{\hat{\mathcal A}} \hat{\mathcal B}$ is the base-change. Note that $f^\natural$ exists by the adjoint functor theorem in our case, as we chose a cutoff cardinal $\kappa$ in the beginning so that $\End^L(\D(\hat{\mathcal A}))$ is presentable. Also, via compatibility with $\eta$ one sees that $f^\natural \id = \id$.

Note that for every $x \in \pi_0 A$ and every $F \in \End^L(\D(\hat{\mathcal A}))$, the pointwise multiplication by $x$ induces an endomorphism of $F$; indeed, via the embedding $\eta$ we can define this endomorphism on $\id$ and then also on $F = \id \comp F$. It thus makes sense to speak of adically complete endofunctors in $\End^L(\D(\hat{\mathcal A}))$. We now observe:
\begin{itemize}
	\item[(a)] An endofunctor $F \in \End^L(\D(\hat{\mathcal A}))$ is adically complete if and only if for every compact $P \in \D(\hat{\mathcal A})$, $F(P) \in \D(\hat{\mathcal A})$ is adically complete.
\end{itemize}
Indeed, adic completeness of enriched endofunctors can be checked on the underlying ordinary functors (see \cite[Lemma A.4.5.(ii)]{mann-mod-p-6-functors}), i.e., in the category $\Fun^L(\D(\hat{\mathcal A}), \D(\hat{\mathcal A}))$ of colimit-preserving functors $\D(\hat{\mathcal A}) \to \D(\hat{\mathcal A})$. If $\mathcal C_{\mathcal A} \subset \D(\hat{\mathcal A})$ denotes the full subcategory of compact objects, then by the universal property of $\Ind$-completions we have
\begin{align*}
	\Fun^L(\D(\hat{\mathcal A}), \D(\hat{\mathcal A})) = \Fun^\ex(\mathcal C_{\mathcal A}, \D(\hat{\mathcal A})),
\end{align*}
where the right-hand side denotes exact functors $\mathcal C_{\mathcal A} \to \D(\hat{\mathcal A})$. Since limits of exact functors between stable categories are exact, limits on the right-hand side are computed pointwise. But this immediately implies (a). From (a) we immediately deduce that $\id \in \End^L(\D(\hat{\mathcal A}))$ is adically complete; here it is crucial to work with $\D(\hat{\mathcal A})$ instead of $\D(\mathcal A)$!

Now let
\begin{align*}
	\mathcal K := \fib(\id \to f_\natural \id) \in \End^L(\D(\hat{\mathcal A})),
\end{align*}
where the map $\id \to f_\natural \id = f_\natural f^\natural \id$ comes from the unit of the adjunction. We now claim:
\begin{itemize}
	\item[(b)] For some $d \ge 0$ the natural map $\mathcal K^d \to \id$ is zero.
\end{itemize}
To prove this, let $d' \ge 0$ be such that each map $f_n\colon \mathcal A/I^n \to \mathcal B/I^n$ is descendable of index $\le d'$. We claim that $d = 2d'$ works. To see this, denote $g_n\colon \mathcal A \to \mathcal A/I^n$ the projection, so that we get adjoint functors
\begin{align*}
	g_n^\natural\colon \End^L(\D(\hat{\mathcal A})) \rightleftarrows \End^L(\D(\mathcal A/I^n)) \noloc g_{n\natural}.
\end{align*}
(Here we implicitly use \cref{rslt:A-mod-I-modified-modules-equals-usual-modules}). The adic completeness of $\id$ implies that the natural map $\id \isoto \varprojlim_n g_{n\natural} \id$ is an isomorphism. We deduce the following identity of spectra:
\begin{align*}
	\Hom(\mathcal K^{d'}, \id) = \varprojlim_n \Hom(\mathcal K^{d'}, g_{n\natural} \id) = \varprojlim_n \Hom(g_n^\natural \mathcal K^{d'}, \id)
\end{align*}
The induced map $\Hom(\mathcal K^{d'}, \id) \to \Hom(g_n^\natural \mathcal K^{d'}, g_n^\natural \id)$ is the one coming from the functoriality of $g_n^\natural$. By \cite[Proposition A.4.17]{mann-mod-p-6-functors} the functor $g_n^\natural$ is naturally a monoidal functor, so that if we denote $\mathcal K_n := \fib(\id \to f_{n\natural} \id) \in \End^L(\D(\mathcal A/I^n)$ then $g_n^\natural \mathcal K^{d'} = \mathcal K_n^{d'}$. We deduce that the natural map $\mathcal K^{d'} \to \id$ gets sent to the natural map $\mathcal K_n^{d'} \to \id$ under $g_n^\natural$. The latter map is zero by choice of $d'$. Now (b) follows completely analogous to the argument in \cite[Proposition 2.7.2]{mann-mod-p-6-functors}.

With (b) at hand, the claimed descent of modules is now formal: By the same procedure as in \cite[Proposition 2.6.5]{mann-mod-p-6-functors} we deduce from (b) that $\id$ can be obtained from $f_\natural \id$ in a finite number of steps, each of which consists of a composition, a finite limit or a retract. Using \cref{rslt:pushout-of-adic-analytic-rings,rslt:steadiness-for-modified-modules}, we can thus apply the argument from \cite[Proposition 2.6.3]{mann-mod-p-6-functors} verbatim.
\end{proof}

We also need a way to pass to filtered colimits, similar to \cite[\S2.7]{mann-mod-p-6-functors}. This is fairly straightforward since adic completions are countable limits and therefore commute with $\omega_1$-filtered colimits (in the following result the superscript $\omega_1$ refers to $\omega_1$-compact objects):

\begin{lemma} \label{rslt:adicness-preserved-under-w1-filtered-colimits}
Let $(\mathcal A_i)_{j\in J}$ be an $\omega_1$-filtered diagram of adic analytic rings and adic transition maps. Then $\mathcal A := \varinjlim_j \mathcal A_j$ is an adic analytic ring and the natural functor
\begin{align*}
	\varinjlim_j \D(\hat{\mathcal A_j})^{\omega_1} \isoto \D(\hat{\mathcal A})^{\omega_1}
\end{align*}
is an equivalence of categories.
\end{lemma}
\begin{proof}
We can assume that $J$ has some initial element $0 \in J$. Let $I$ be some ideal of definition of $\underline{\mathcal A}_0$. Then each $\underline{\mathcal A}_j$ is $I$-adically complete and since countable limits commute with $\omega_1$-filtered colimits in $\D(\mathrm{Cond}(\mathrm{Ab}))$, we deduce that $\underline{\mathcal A} = \varinjlim_j \underline{\mathcal A}_j$ is $I$-adically complete. Clearly, $\underline{\mathcal{A}}$ is discrete mod $I$ as all transition maps are adic. This shows that $\mathcal A$ is indeed an adic analytic ring. By a similar argument we see that for any $\omega_1$-compact $P \in \D(\hat{\mathcal A_0})^{\omega_1}$ the natural map
\begin{align*}
	\varinjlim_j (P \tensor_{\hat{\mathcal A_0}} \hat{\mathcal A_j}) \isoto P \tensor_{\hat{\mathcal A_0}} \hat{\mathcal A}
\end{align*}
is an isomorphism in $\D(\hat{\mathcal A_0})$. Indeed, this claim is stable under colimits in $P$, so by \cite[Lemma A.2.1]{mann-mod-p-6-functors} we can w.l.o.g. assume that $P = \mathcal A_0[S]_{\hat I}$ for some profinite set $S$. Then the claim boils down to showing that the map $\varinjlim_j \mathcal A_j[S]_{\hat I} \isoto \mathcal A[S]_{\hat I}$ is an isomorphism. This is true without $I$-completions by \cite[Proposition 2.3.15.(ii)]{mann-mod-p-6-functors} and then follows for the $I$-completed version because $I$-completions commute with $\omega_1$-filtered colimits. With the above isomorphism at hand, the rest of the argument works exactly as in \cite[Lemma 2.7.4]{mann-mod-p-6-functors}.
\end{proof}

\begin{definition}
  \label{sec:complete-descent-1-definition-weakly-adically-descendable}
A map $f\colon \mathcal A \to \mathcal B$ of adic analytic rings is called \emph{weakly adically descendable} if it is an iterated $\omega_1$-filtered colimit of adically descendable maps of adic analytic rings along adic maps.\footnote{I.e., $f$ lies in the smallest subcategory of adic morphisms of adic analytic rings, which contains the adically descendable maps and is stable under $\omega_1$-filtered colimits.} If all the maps in this $\omega_1$-filtered colimit are descendable of index $\le d$, for some fixed $d \ge 0$, then we say that $f$ has \emph{index $\le d$}.
\end{definition}

\begin{proposition}
\begin{propenum}
	\item If $\mathcal A \to \mathcal B$ is weakly adically descendable then it is adic and modules descend along $\hat{\mathcal A} \to \hat{\mathcal B}$.

	\item Every adically descendable morphism of adic analytic rings is weakly adically descendable, and weakly adically descendable morphisms are stable under adic base-change and $\omega_1$-filtered colimits.

	\item \label{rslt:filtered-colim-of-adic-desc-maps-is-weakly-desc} A completed filtered colimit of adically descendable morphisms of index $\le d$ along adic maps of adic analytic rings is weakly adically descendable of index $\le 2d$.
\end{propenum}
\end{proposition}
\begin{proof}
Part (i) follows from \cref{rslt:descendability-of-adic-rings,rslt:adicness-preserved-under-w1-filtered-colimits} by the same argument as in \cite[Proposition 2.7.5]{mann-mod-p-6-functors} (use also \cref{rslt:pushout-of-adic-analytic-rings} in order to reduce to the case that all maps in the $\omega_1$-filtered colimit have the same source). In part (ii) the only non-trivial claim is the one about base-change, and this follows from \cref{rslt:pushout-of-adic-analytic-rings} and \cite[Lemma 2.6.9]{mann-mod-p-6-functors}. Part (iii) follows by the same argument as in \cite[Theorem 2.7.8.(iii)]{mann-mod-p-6-functors}.
\end{proof}

\subsection{Globalization for stably uniform adic spaces}
\label{sec:glob-stably-unif}

We finish our general discussion of the modified category $\D_{\hat\solid}(\mathcal{A})$ for an adic analytic ring $\mathcal{A}$ with a globalization to (classical) stably uniform analytic adic spaces. Let $(A,A^+)$ be a classical complete analytic Huber pair, and $X:=\Spa(A,A^+)$. Recall that $(A,A^+)$ is called uniform if $A^\circ$ is a ring of definition, and that $(A,A^+)$ is called stably uniform if for each rational open $U\subseteq X$ the analytic Huber pair $(\mathcal{O}_X(U),\mathcal{O}_X^+(U))$ is uniform, in which case $(A,A^+)$ is automatically sheafy (\cite[Definition 5.2.4., Theorem 5.2.5.]{scholze-berkeley-lectures}). In \cite[Theorem 4.1]{andreychev-condensed-huber-pairs} Andreychev has proven that if $(A,A^+)$ is sheafy (not necessarily uniform), then the functor $U\mapsto \D_\solid(\mathcal{O}_X(U),\mathcal{O}_X^+(U))$ on rational opens $U\subseteq X$ satisfies descent for the analytic topology on $X$.

\begin{definition}
  \label{sec:glob-stably-unif-1-definition-of-d-hat-solid-a-a+}
  \begin{defenum}
\item  For any classical complete analytic uniform Huber pair $(B,B^+)$ set
  \[
    \D_{\hat\solid}(B,B^+):=\mathrm{Mod}_{B}(\D_{\hat\solid}(B^\circ, B^+)).
  \]
  Here, we implicitly view $B,B^\circ$ as condensed static rings (although $B^+$ is only treated as a discrete ring). Note that we need uniformity to achieve that $B^\circ$ is an adic ring in the sense of \cref{sec:adic-rings-1-definition-adic-analytic-ring}.
\item Given a map $f\colon Y=\Spa(B,B^+)\to X=\Spa(A,A^+)$ of stably uniform analytic adic spaces, we let $f^\ast:=\widehat{(B^\circ,B^+)}_\solid\tensor_{\widehat{(B^\circ,B^+)}_\solid}(-)\colon \D_{\hat\solid}(X):=\D_{\hat\solid}(A,A^+)\to \D_{\hat\solid}(Y):=\D_{\hat\solid}(B,B^+)$, $f_\ast\colon \D_{\hat\solid}(Y)\to \D_{\hat\solid}(X)$ be the induced pair of adjoint functors. If we want to distinguish these functors two functors $f^\ast, f_\ast$ from their counterparts $f^\ast=(B,B^+)_\solid\tensor_{(A,A^+)_\solid}(-)\colon \D_\solid(X):=\D_{\solid}(A,A^+)\to \D_{\solid}(Y):=\D_{\solid}(B,B^+)$, we denote them by $\hat{f}^\ast, \hat{f}_\ast$.
  \item We let $\alpha_X^\ast\colon \D_{\solid}(X)=\mathrm{Mod}_A(\D_{\solid}(A^\circ,A^+))\to \D_{\hat\solid}(X)$ be the functor induced on module categories by the symmetric monoidal functor $\widehat{(A^\circ,A^+)_\solid}\otimes_{(A^\circ, A^+)_\solid}(-)$ from \cref{def:modified-modules-over-adic-ring}. We let $\alpha_{X,\ast}$ be the right adjoint of $\alpha_X^\ast$.
  \end{defenum}
  \end{definition}

  Let $\mathrm{Sym}$ be the $\infty$-category of cocomplete closed symmetric monoidal stable $\infty$-categories with morphisms given by cocontinuous, symmetric monoidal functors, cf.\ \cite[Definition 6.3]{condensed-complex-geometry}.
  We recall from \cite[Proposition 6.5]{condensed-complex-geometry} that a morphism $g^\ast\colon C\to D$ in $\mathrm{Sym}$ is called an open (resp.\ closed) immersion if $g^\ast$ has a fully faithful left adjoint $g_!$ (resp.\ colimit preserving fully faithful right adjoint $g_\ast$), which satisfies the projection formula. If $g^\ast$ is a closed immersion, then $g_\ast(1_D)\in C$ is an idempotent algebra, and $D\cong \mathrm{Mod}_{g_\ast(1_D)}(C)$. If $g^\ast$ is an open immersion, then the cofiber $R:=[g_!(1_D)\to 1_C]\in C$ is an idempotent algebra, and $D$ is isomorphic to the quotient $C/\mathrm{Mod}_{R}(C)$ (via $g^\ast$). We record the following stability of open/closed immersions under base change with respect to the Lurie tensor product.

  \begin{lemma}
    \label{sec:glob-stably-unif-2-closed-open-immersions-are-stable-under-base-change-for-lurie-tensor-product}
    Let $h^\ast \colon C\to C^\prime$, $g^\ast\colon C\to D$ be morphisms in $\mathrm{Sym}$. If $g^\ast$ is an open (resp.\ closed) immersion, the same is true for the base change $k^\ast\colon C^\prime\to D^\prime:=C^\prime\otimes_C D$.
  \end{lemma}
  \begin{proof}
    For closed immersions this is clear as $\mathrm{Mod}_R(C)\otimes_C C^\prime\cong \mathrm{Mod}_{h^\ast(R)}(C^\prime)$ for any monoid object $R\in C$, and $h^\ast$ preserves idempotent algebras. Thus, assume that $g^\ast$ is an open immersion, and let $R:=[g_!(1_D)\to 1_C]\in C$ be the associated idempotent algebra. As the base change $C^\prime\otimes_C(-)$ preserves adjunctions between colimit preserving, $C$-linear functors, the base change $k_!$ of $g_!$ is a fully left adjoint of $k^\ast$. By $C^\prime$-linearity one checks that $k_!$ satisfies the projection formula. In particular, $k^\ast$ is an open immersion. Moreover, we note that as in the proof of \cite[Proposition 6.5]{condensed-complex-geometry} the projection formula for $k^\ast$ implies that the kernel of the functor $k^\ast$ is exactly the category of $h^\ast(R)$-modules in $C^\prime$ because $h^\ast(R)$ is the cofiber of $k_!(1_{D^\prime})\to 1_{C^\prime}$ by exactness of $h^\ast$. 
  \end{proof}

  \begin{remark}
    \label{sec:glob-stably-unif-3-tensoring-preserves-fiber-sequences}
    The last part of the proof of \cref{sec:glob-stably-unif-2-closed-open-immersions-are-stable-under-base-change-for-lurie-tensor-product} shows that base change preserves the ``excision sequence'' $D\xto{g_!} C\to \mathrm{Mod}_R(C)$ for an open immersion $g^\ast\colon C\to D$.
  \end{remark}

We now show that rational localizations yield open immersion in $\mathrm{Sym}$ for the completed solid category:
  
\begin{lemma}
  \label{sec:glob-stably-unif-1-rational-opens-still-define-localizations}
  Let $X=\Spa(A,A^+)$ be a stably uniform analytic adic space, and $j\colon U:=\Spa(B,B^+)\to X$ a rational open.
  Then $\hat{j}^\ast\colon \D_{\hat\solid}(X)\to \D_{\hat\solid}(U)$ is an open immersion in $\mathrm{Sym}$. In fact, $\hat{j}^\ast$ is the base change of $j^\ast\colon \D_{\hat\solid}(X)\to \D_{\hat\solid}(U)$ along the functor $\alpha_X^\ast\colon \D_{\solid}(X)\to \D_{\hat\solid}(X)$.
\end{lemma}
\cref{sec:glob-stably-unif-2-closed-open-immersions-are-stable-under-base-change-for-lurie-tensor-product} and the proof below shows that $\alpha^\ast_X\circ j_!\cong \hat{j}_{!}\circ \alpha^\ast_U $ for the left adjoints $j_!, \hat{j}_!$ of $j^\ast,\hat{j}^\ast$.
\begin{proof}
  This follows from \cref{sec:defin-d_hats-1-modified-version-in-relative-+-finite-type-case} and \cref{sec:glob-stably-unif-2-closed-open-immersions-are-stable-under-base-change-for-lurie-tensor-product}: the morphism $(A^\circ, A^+)_\solid\to (B^\circ, B^+)_\solid$ is of $+$-finite type, and therefore
  \begin{equation}
    \label{eq:1-tensor}
        \D_{\hat{\solid}}(A^\circ, A^+)\otimes_{\D_\solid(A^\circ,A^+)}\D_\solid(\mathcal{O}_X(U)^\circ, \mathcal{O}_X^+(U))\isoto \D_{\hat{\solid}}(\mathcal{O}_X(U)^\circ, \mathcal{O}_X^+(U)).
  \end{equation}
  We claim that the natural morphism $A\otimes_{A^\circ} B^\circ\to B$ is an isomorphism in $\D_\solid(A^\circ,A^+)$. Namely, by \cite[Theorem 1.6]{andreychev-condensed-huber-pairs} the claim is local on $X=\Spa(A,A^+)$, which reduces to the case that $A$ is Tate, where the claim is clear. Given this claim, we can conclude
  \[
    \D_{\hat{\solid}}(X)\otimes_{\D_\solid(X)}\D_\solid(U)\isoto \D_{\hat{\solid}}(U)
  \]
  by tensoring (\cref{eq:1-tensor}) over $\D_\solid(A^\circ, A^+)$ with $\mathrm{Mod}_A(\D_{\solid}((A^\circ, A^+))$. From here \cref{sec:glob-stably-unif-2-closed-open-immersions-are-stable-under-base-change-for-lurie-tensor-product} implies the claim.
\end{proof}

We now establish the following version of Andreychev's analytic descent theorem for the completed solid category:

\begin{theorem}
  \label{sec:glob-stably-unif-1-analytic-and-etale-descent-for-hat-solid-version}
  Let $(A,A^+)$ be a classical stably uniform analytic Huber pair, and $X:=\Spa(A,A^+)$.
  \begin{lemenum}
  \item The functor $U\mapsto \D_{\hat\solid}(\mathcal{O}_X(U),\mathcal{O}_X^+(U))$ on rational open subsets satisfies descent for the analytic topology.
    \item If for any rational open $U\subseteq X$ each finite, \'etale $\mathcal{O}_X(U)$-algebra is stably uniform (e.g., $A$ is sousperfectoid, \cite[Proposition 6.3.3.]{scholze-berkeley-lectures}), then the functor $V\mapsto \D_{\hat\solid}(\mathcal{O}_V(V),\mathcal{O}_V^+(V))$ on the category of \'etale maps $V\to X$ with $V$ affinoid, satisfies descent for the \'etale topology. 
  \end{lemenum}
\end{theorem}
Here, a morphism $V\to X$ is \'etale if it is a composition of open immersions and finite \'etale maps (in the sense of \cite[Definition 7.1]{scholze-perfectoid-spaces}, generalized to our situation), and the \'etale topology is defined by jointly surjective \'etale morphisms.
\begin{proof}
By definition of the \'etale topology, it suffices to show (i) and descent along finite \'etale maps $(A,A^+)\to (B,B^+)$. But by definition of the latter, $(B,B^+)_\solid=(B,A^+)_\solid$, which reduces to descent along the map $A\to B$ of algebras in $\D_{\hat\solid}(A,A^+)$. Now, $A(\ast)\to B(\ast)$ is descendable as a map of classical rings (\cite[Corollary 3.33]{akhil-galois-group-of-stable-homotopy}). As it is moreover finite \'etale, this implies $B\cong A\otimes_{\underline{A(\ast)}}\underline{B(\ast)}$, and thus $A\to B$ is descendable in $\D_{\hat\solid}(A,A^+)$ because the functor $\D(A(\ast))\to \D_{\hat\solid}(A,A^+)$ is symmetric monoidal. Hence, (ii) is implied by (i). In order to show (i) we first note that if $X=\bigcup\limits_{i=1}^n U_i$ is a finite open cover with $j_i\colon U_i\to X$ rational open, then by \cref{sec:glob-stably-unif-1-rational-opens-still-define-localizations} the functors $\hat{j}^\ast_i\colon \D_{\hat{\solid}}(X)\to \D_{\hat{\solid}}(U_i)$ are open immersions in $\mathrm{Sym}$. Moreover, the collection $\hat{j}^\ast_i,i=1,\ldots, n,$ is jointly conservative. Indeed, this statement is equivalent to $\hat{j}_{1,!}(1_{U_1})\otimes\ldots\otimes \hat{j}_{n,!}(1_{U_n})=0$ and thus is implied by $j_1(1_{U_1})\otimes\ldots \otimes j_n(1_{U_n})=0$ (\cite[Proposition 4.12.(v)]{andreychev-condensed-huber-pairs}) using that $\alpha^\ast$ commutes with the $!$-functors (\cref{sec:glob-stably-unif-1-rational-opens-still-define-localizations}). Thus, the assertion follows from the general descent theorem \cite[Proposition 5.5]{condensed-complex-geometry} (or \cite[Proposition 4.13]{andreychev-condensed-huber-pairs} resp.\ \cite[Proposition 10.5]{condensed-mathematics}).
\end{proof}

\Cref{sec:glob-stably-unif-1-analytic-and-etale-descent-for-hat-solid-version} allows us to give a reasonable definition of completed solid sheaves on every stably uniform analytic adic space:

\begin{definition}
  \label{sec:glob-stably-unif-1-modified-quasi-coherent-sheaves-for-general-stably-uniform-spaces}
  Let $X$ be a (classical) stably uniform analytic adic space. Then we set $\D_{\hat\solid}(X)$ as the inverse limit of $\D_{\hat\solid}(A,A^+)$ over all affinoid opens $\Spa(A,A^+)\subseteq X$, with $(A,A^+)$ a stably uniform analytic Huber pair. 
\end{definition}

\subsection{Universal descent}
\label{sec:universal-descent}

In this section, we show that descent of $\D(-)$ along a \textit{steady} morphism $\mathcal A \to \mathcal B$ of analytic rings automatically implies descent after any base change. We discuss variants of this as well, e.g., for $\D_{\hat\solid}(-)$.

If $f\colon \mathcal A\to \mathcal B$ is a morphism of analytic rings, we will denote by $f^\ast(-):=\mathcal{B}\otimes_{\mathcal A}(-)\colon \D(\mathcal A)\to \D(\mathcal B)$ the associated base change functor, and by $f_\ast$ its right adjoint. Note that $f_\ast$ is always conservative.

We recall that $f$ is steady if for any morphism $g\colon \mathcal A \to \mathcal A^\prime$ of analytic rings with base change $f^\prime\colon \mathcal A^\prime\to \mathcal B^\prime:=\mathcal A^\prime\otimes_{\mathcal A}\mathcal B, g^\prime\colon \mathcal B\to \mathcal B^\prime$, the natural transformation
\[
  f^\ast g_\ast\to g^\prime_\ast f^{\prime, \ast}
\]
of functors $\D(\mathcal A^\prime)\to \D(\mathcal B)$ is an isomorphism (\cite[Proposition 12.14]{scholze-analytic-spaces}).

\begin{lemma}
  \label{sec:universal-descent-1-descent-implies-universal-descent-for-analytic-rings}
  Assume that $f\colon \mathcal A \to \mathcal B$ is a steady morphism of analytic rings and that modules descend along $f$, i.e., the natural functor $\D(A)\to \varprojlim\limits_{[n]\in \Delta}\D(\mathcal B^{n/\mathcal{A}})$ is an equivalence, where $\mathcal B^{\bullet/\mathcal A}$ denotes the \v{C}ech nerve for $f$. Then for any morphism $g\colon \mathcal A\to \mathcal A^\prime$ of analytic rings modules descend along $f^\prime\colon \mathcal A^\prime\to \mathcal{B}^\prime:=\mathcal A^\prime\otimes_{\mathcal A}\mathcal B$.  
\end{lemma}
\begin{proof}
  This follows from \cite[Corollary 5.2.2.37.]{lurie-higher-algebra}. We supplement the details using the notation of \cite[Corollary 5.2.2.37]{lurie-higher-algebra}. Thus, we set $\mathcal{C}=\Delta$, $\chi\colon \Delta^{\vartriangleleft}\to \mathcal{C}at_\infty, [n]\mapsto \D((\mathcal{B}^\prime)^{n/\mathcal{A}^\prime})$ with $(\mathcal{B}^\prime)^{-1/\mathcal{A}^\prime}:=\mathcal{A}^\prime$, and similarly, $\chi^\prime([n]):=\D(\mathcal{B}^{n/\mathcal{A}})$. To construct the natural transformation $\rho\colon \chi\to \chi^\prime$, given componentwise by $g_{n,\ast}$ for $g_n\colon \mathcal{B}^{n/\mathcal{A}}\to (\mathcal B^\prime)^{n/\mathcal{A}^\prime}$ is the natural map, one can argue as follows: clearly, the $g^\ast_n$ yield a natural transformation $\sigma\colon \chi^\prime\to \chi$ (by naturality of $\D(-)$ as a functor on the $\infty$-category of analytic rings). As for each morphism $[n]\to [m]$ in $\Delta^{\vartriangleleft}$ the induced morphism $\mathcal{B}^{n/\mathcal{A}}\to \mathcal{B}^{m/\mathcal A}$ is steady (using \cite[Proposition 12.15]{scholze-analytic-spaces}), we can conclude\footnote{Let $E\to \Delta^\vartriangleleft, E^\prime\to \Delta^\vartriangleleft$ be the cocartesian fibrations classified by $\chi, \chi^\prime$. The natural transformation $\sigma$ unstraightens to a functor $\kappa\colon E^\prime\to E$, which preserves cocartesian arrows. Using \cite[7.3.2.6.]{lurie-higher-algebra} the functor $\kappa$ admits an adjoint $\lambda\colon E\to E^\prime$ (relative to $\Delta^\vartriangleleft$) as this is true over each $[n]\in \Delta^\vartriangleleft$. Now, steadiness of $f$ implies that $\lambda$ preserves cocartesian arrows, and thus defines a natural transformation $\rho\colon \chi\to \chi^\prime$. By construction of $\lambda$, $\rho_{[n]}\colon \chi([n])\to \chi^\prime([n])$ is given by the pushforward $g_{n,\ast}$ as desired.} that the $g_{n,\ast}$ assemble to a natural transformation $\rho\colon \chi\to \chi^\prime$. Now condition 1) in \cite[5.2.2.37]{lurie-higher-algebra} is the assumption that modules descend along $f$. Condition 2) is automatic as each $g_{n,\ast}$ is conservative. Condition 3) is clear as $\D(\mathcal{A}^\prime)$ has all limits and $g_\ast$ preserves these. Condition 4) is again automatic because for each morphism $\alpha\colon [n]\to [m]$ in $\Delta^{\vartriangleleft}$ the functors $\chi([n])\to \chi([m])$ resp.\ $\chi^\prime([n])\to \chi^\prime([m])$ admit right adjoints given by $\ast$-pushforward, and clearly these right adjoints commute with $\rho$ (which in component $[n]$ is given by $g_{n,\ast}$). Thus, \cite[5.2.2.37]{lurie-higher-algebra} is applicable, and yields exactly that modules descend along $f^\prime$. 
\end{proof}

Next assume that $f\colon \mathcal{A}\to \mathcal{B}$ is a morphism of adic analytic rings. By \cref{sec:complete-descent-2-steady-maps-of-adic-analytic-rings-are-exactly-the-adic-ones} we know that $f$ is steady if and only if $f$ is adic. We get the following analog of \cref{sec:universal-descent-1-descent-implies-universal-descent-for-analytic-rings} for $\D(\hat{\mathcal{A}})$.

\begin{lemma}
  \label{sec:universal-descent-1-descent-implies-universal-descent-in-adic-case}
  Let $f\colon \mathcal A\to \mathcal B$ be an adic morphism of adic analytic rings, such that the natural functor $\D(\hat{\mathcal{A}})\to \varprojlim\limits_{[n]\in \Delta}\D(\hat{\mathcal{B}^{n/\mathcal{A}}})$ is an equivalence. Then the same holds true after any base change $f^\prime\colon \mathcal{A}^\prime\to \mathcal{A}^\prime\otimes_{\mathcal{A}}\mathcal{B}$ of $f$ along an adic morphism $\mathcal{A}\to \mathcal{A}^\prime$ of adic analytic rings.
\end{lemma}
Note that adicness of $\mathcal A\to \mathcal A^\prime$ implies by \cref{rslt:pushout-of-adic-analytic-rings} that all terms in the \v{C}ech nerves are adic analytic rings (and thus the assertion is well-defined). 
\begin{proof}
 Replacing $\D((-))$ by $\D(\hat{-})$ we can follow the steps in \cref{sec:universal-descent-1-descent-implies-universal-descent-for-analytic-rings}. Using \cref{rslt:steadiness-for-modified-modules} the existence of $\rho$ follows in the same way. The only statement to check is that $\hat{f}_\ast\colon \D(\hat{\mathcal{B}})\to \D(\hat{\mathcal{A}})$ is conservative for any morphism of adic analytic rings $f\colon \mathcal A\to \mathcal B$. As $\alpha_\ast \hat{f}_\ast\cong f_\ast \alpha_\ast$ (by \cref{rslt:functoriality-of-modified-modules} and passage to right adjoints) this follows from conservativity of $f_\ast$ and $\alpha_\ast$ (\cref{rslt:basic-properties-of-modified-modules}). 
\end{proof}

\begin{remark}
  \label{sec:universal-descent-1-universal-descent-for-stably-uniform-adic-spaces}
  Let $f\colon Y\to X$ be any morphism of stably uniform analytic adic spaces. Assume that all terms $Y^{n/X}$ of the \v{C}ech nerve of $f$ are again stably uniform (this is a serious assumption, cf.\ \cref{sec:universal-descent-2-product-of-q-p-cycl-with-itself-over-q-p}), and that $\D_{\hat\solid}$ satisfies descent for $f$. Then for any morphism $X^\prime\to X$ of stably uniform analytic adic spaces (such that the terms in the \v{C}ech nerve of $f^\prime\colon Y\times_X X^\prime \to X^\prime$ are again stably uniform) the modified quasi-coherent sheaves $\D_{\hat\solid}$ satisfy descent for $f^\prime$ (by reducing to the affinoid case and then using the same argument as in \cref{sec:universal-descent-1-descent-implies-universal-descent-for-analytic-rings}, \cref{sec:universal-descent-1-descent-implies-universal-descent-in-adic-case}). Note that the steadiness assumption is automatically satisfied in this case because $X$ is locally Tate.
\end{remark}

\begin{example}
  \label{sec:universal-descent-2-product-of-q-p-cycl-with-itself-over-q-p}
  Let $\Q_p^\cycl$ be the completion of $\Q_p(\mu_{p^\infty})$. We mention here the well-known example that the Huber ring $A:=\Q_p^\cycl\otimes_{\Q_p} \Q_p^\cycl$ (with $\otimes$ refering to the solid or equivalently the Banach space tensor product) is not uniform even though each of the factors is. In fact, the uniform completion of $A$ is the space $B:=C(\Gamma, \Q_p^\cycl)$ of continuous functions on $\Gamma=\Z_p^\cprod=\mathrm{Gal}(\Q_p^\cycl/\Q_p)$. The ring $A_0:=\Z_p^\cycl\otimes_{\Z_p}\Z_p^\cycl$ is a ring of definition of $A$ (by definition of the Banach space tensor product), and the ring $B_0:=C(\Gamma,\Z_p^\cycl)$ is a ring of definition of $B$. Moreover, $B_0=B^0$ because $\Z_p^\cycl$ is the ring of power bounded elements in $\Q_p^\cycl$. In particular, $B$ is uniform. There exists the natural map
  \[
    \Phi\colon A_0\to B_0, a\otimes b\mapsto (\gamma\mapsto a\gamma(b)).
 \]
 From the relation
 \[
   \sum\limits_{\zeta\in \mu_{p^n}} \zeta^x=
   \begin{cases}
     p^n,\ & x\in p^n\Z_p \\
     0,\ & \textrm{ otherwise }
   \end{cases}
 \]
 for $x\in \Z_p$ one can conclude that for $a\in \Z_p^\cprod$ the continuous function
 \[\Z_p^\cprod\to \Z_p^\cycl,\ x\mapsto \frac{1}{p^h}\sum\limits_{\zeta\in \mu_{p^n}} \zeta^{x-a}
 \]
 is the characteristic function $\chi_{a,n}$ of the clopen subset $a+p^n\Z_p$ of $\Z_p^\cprod=\Gamma$. We can conclude that $\Phi$ is injective with image exactly the closure of the subspace generated by the $p^n\chi_{a,n}$ for $a\in \Z_p^\cprod$ and $n\geq 0$ (cf.\ \cite[Proposition I.3.4]{colmez2010fonctions}).
 In particular, $\Phi$ extends to an injection $A\to B$ with dense image and this identifies $B$ as the uniform completion of $A$. Concretely, for $a\in \Z_p^\cprod$ and $n\geq 0$ the elements
 \[
   \frac{1}{p^n}\sum\limits_{\zeta\in \mu_{p^n}}\zeta^{-a}\otimes \zeta\in A
 \]
 generate $A^0\subseteq A$ as an $A_0$-module.
\end{example}

\section{Sheaves on diamonds} \label{sec:nuclear-sheaves}

Throughout this section we fix two primes $p$ and $\ell$, where we explicitly allow the case $\ell = p$. In the following we construct and analyze different categories of $\ell$-adic sheaves on diamonds and small v-stacks over $\Z_p$. Most importantly, we discuss a version of the nuclear $\Z_\ell$-sheaves defined in \cite{mann-nuclear-sheaves} in the case $\ell = p$. In fact, the basic definitions and results all work in the same way (similar to the fact that the basic results on étale $\Fld_p$-sheaves work the same as for $\Fld_\ell$), but the theory is lacking the required base-change result (for solid sheaves) to produce a good 6-functor formalism.\footnote{More precisely, \cite[Proposition VII.2.1]{fargues-scholze-geometrization} needs the assumption $\ell\neq p$.} Nevertheless, nuclear $\Z_p$-sheaves will play an important role in the proof of descent for affinoid perfectoid spaces, whose tilt admits a morphism of finite $\dimtrg$ to a totally disconnected perfectoid space, in \cref{sec:introduction-1-main-theorem-introduction}.

\subsection{Étale sheaves}

We fix a ring $\Lambda$ such that $\ell^n \Lambda = 0$ for some $n \ge 0$. Recall that in \cite[Definition 14.13]{etale-cohomology-of-diamonds} (and \cite[Lemma 17.1]{etale-cohomology-of-diamonds}) a $v$-sheaf of $\infty$-categories $Y\mapsto \D_\et(Y,\Lambda)$ on small $v$-stacks is defined for any ring $\Lambda$ such that $\D_\et(Y,\Lambda)\cong \D(Y_\et,\Lambda)$ if $Y$ is a strictly totally disconnected space (or more generally if $Y$ is a locally spatial diamond, which has finite $\ell$-cohomological dimension). By descent from the strictly totally disconnected case the category $D_\et(Y,\Lambda)$ is equipped with a natural $t$-structure. More precisely, an object $M\in \D_\et(Y,\Lambda)$ lies in $\D_{\geq 0}$ (or $\D_{\leq 0}$) if and only this is true after pullback to all strictly totally disconnected spaces over $Y$. If $Y$ is a locally spatial diamond such that $\D_\et(Y,\Lambda)\cong \D(Y_\et,\Lambda)$, then this recovers the usual $t$-structure on $\D(Y_\et,\Lambda)$.
For a morphism $f\colon Y'\to Y$ of small $v$-stacks the pullback functor $f^\ast\colon \D_\et(Y,\Lambda)\to \D_\et(Y',\Lambda)$ has a right adjoint $f_\ast\colon \D_\et(Y',\Lambda)\to \D_\et(Y,\Lambda)$. In general, this right adjoint $f_\ast$ is not easy to control. However, if $Y$ is a diamond, then by descent from the strictly totally disconnected case the functor $\D_\et(Y,\Lambda)\to \D(Y_\qproet, \Lambda)$ is fully faithful (\cite[Proposition 14.10]{etale-cohomology-of-diamonds}), and if $f\colon Y'\to Y$ is a qcqs morphism of diamonds, then $f_{\qproet,\ast}\colon \D(Y'_\qproet,\Lambda)\to \D(Y_\qproet,\Lambda)$ restricts to a functor $\D^+_\et(Y',\Lambda)\to \D^+_\et(Y,\Lambda)$ (by the argument of the proof of \cite[Corollary 16.8.(ii)]{etale-cohomology-of-diamonds} and \ref{sec:bound-cond-new-quasi-pro-etale-base-change} below). Thus, $f_{\qproet,\ast}=f_\ast$ on $\D^+_\et(Y',\Lambda)$.

In \cite[\S16]{etale-cohomology-of-diamonds} many base-change properties for étale sheaves on diamonds are discussed. In the following we need a slight generalization of \cite[Corollary~16.10]{etale-cohomology-of-diamonds} to the case of diamonds that are not necessarily locally spatial:

\begin{proposition} \label{sec:bound-cond-new-quasi-pro-etale-base-change}
  Let
  \[\begin{tikzcd}
    {X'} & {Y'} \\
    X & Y
    \arrow["{g'}", from=1-1, to=1-2]
    \arrow["f", from=1-2, to=2-2]
    \arrow["g", from=2-1, to=2-2]
    \arrow["{f'}"', from=1-1, to=2-1]
  \end{tikzcd}\]
  be a cartesian diagram of diamonds. Assume that $f$ is qcqs and that $g$ is quasi-pro-\'etale. Then for $M\in \D^+_\et(Y',\Lambda)$ the natural morphism
  \[
    g^\ast f_\ast M\to f'_\ast g'^\ast M
  \]
  is an isomorphism.
\end{proposition}
\begin{proof}
  If $Y,X,Y'$ (and hence $X'$) are locally spatial diamond, then $\D^+_\et(-,\Lambda)\cong \D^+((-)_\et,\Lambda)$ by \cite[Proposition 14.15]{etale-cohomology-of-diamonds} and the assertion follows from \cite[Corollary 16.10]{etale-cohomology-of-diamonds}.
  Using quasi-pro-\'etale descent of $\D_\et^+$ we may formally reduce to the case that $Y,X$ are spatial diamonds (even strictly totally disconnected perfectoid spaces). Let $h\colon Z\to Y'$ be a quasi-pro-\'etale surjection from a perfectoid space $Z$. Then $h$ is qcqs and all terms in the \v{C}ech nerve for $h$ are spatial diamonds (using \cite[Corollary 11.29]{etale-cohomology-of-diamonds}). This formally reduces the assertion to the case that $Y'$ is as well spatial (more details on this reduction can be found in \cite[Proposition 3.3.7]{mann-mod-p-6-functors}). This finishes the proof.
\end{proof}

\begin{remark}
\Cref{sec:bound-cond-new-quasi-pro-etale-base-change} holds for \emph{any} ring $\Lambda$. The assumption $\ell^n \Lambda = 0$ is not necessary.
\end{remark}

The following is an analog of \cite[Theorem~19.2]{etale-cohomology-of-diamonds} that includes the case $\ell = p$. We make the additional assumption that $Y'$ and $Y$ are diamonds to circumvent the failure of strong base change results mod $p$. In the following recall the definition of $j_!$ from \cite[Definition~19.1]{etale-cohomology-of-diamonds}.

\begin{proposition} \label{sec:bound-cond-1-proper-base-change-mod-p}
  Let $f\colon Y^\prime\to Y$ be a proper morphism of diamonds, and let $j\colon U \injto Y$ be an open immersion with pullback $g\colon U':=Y'\cprod_Y U \to U$, $j'\colon U' \injto Y'$. Then for any $A\in \D^+_\et(U^\prime,\Lambda)$ the natural morphism $j_! g_\ast A \isoto f_\ast j'_! A$ is an isomorphism.
\end{proposition}
\begin{proof}
  The proof of \cite[Theorem 19.2]{etale-cohomology-of-diamonds} goes through with a few changes: To avoid the use of \cite[Proposition 17.6]{etale-cohomology-of-diamonds}, take a quasi-pro-\'etale cover $X \to Y$ with $X$ strictly totally disconnected (this is possible as $Y$ is a diamond) and apply \cref{sec:bound-cond-new-quasi-pro-etale-base-change} to reduce to the case $Y$ strictly totally disconnected. We note that the reduction in \cite[Theorem 19.2]{etale-cohomology-of-diamonds} to the case that $Y'=\overline{X^\prime}^{/Y}$ for some strictly totally disconnected space $X'$ over $Y=X$ does not need \cite[Corollary 16.8.(ii)]{etale-cohomology-of-diamonds} (or something similar) if the $v$-hypercover is constructed using quasi-pro-\'etale surjections by strictly totally disconnected spaces (which is possible as $Y'$ is assumed to be a diamond): the quasi-pro-\'etale site of $Y^\prime$ is replete, and $R\Gamma_\et(Y',B)\cong R\Gamma(Y'_{\qproet}, B)$ for $B\in \D^+_\et(Y',\Lambda)$ (this follows by quasi-pro-\'etale descent from \cite[Proposition 14.10]{etale-cohomology-of-diamonds}). This implies the desired cohomological descent. The remaining steps in the proof, e.g., \cite[Lemma 19.4]{etale-cohomology-of-diamonds}, go through without change.
\end{proof}

The category of $\ell$-adic étale sheaves on a diamond $X$ behaves particularly well if $X$ is $\ell$-cohomologically bounded and the topology on $X$ is not too wild. Regarding the second point, \cite{etale-cohomology-of-diamonds} usually works with (locally) \emph{spatial} diamonds. For technical reasons we extend this definition slightly and work with \emph{prespatial} diamonds instead. The reader is invited to ignore these subtleties and assume all prespatial diamonds to be spatial in the following.

Let us recall that a diamond $Y$ is prespatial if there exists a spatial subdiamond $Y_0\subseteq Y$ such that for all perfectoid fields $K$ of characteristic $p$, the map $Y_0(K,\ri_K)\to Y(K,\ri_K)$ is a bijection (\cite[Definition 3.1]{mod-ell-stacky-6-functors}). We refer to \cite[Section 3.1]{mod-ell-stacky-6-functors} for general properties of prespatial diamonds, e.g., stability of prespatial diamonds under fiber products. With this definition at hand, we now make the following definition, similar to \cite[Definition 2.1]{mann-nuclear-sheaves}:

\begin{definition} \label{def:ell-bounded-prespatial-diamond}
A prespatial diamond $Y$ is called \emph{$\ell$-bounded} if there exists some $d \ge 0$ such that for every $M \in \D_\et(Y,\F{\ell})$ which is concentrated in degree $0$, the object
\begin{align*}
  \Gamma(Y, M) := \Hom_{\D_\et(Y,\F{\ell})}(\F{\ell}, M)\in \D(\F{\ell})
\end{align*}
is bounded to the right by $d$, i.e. lies in $\D_{\geq -d} = \D^{\leq d}$. 
\end{definition}

\begin{remark}
If $Y$ is a spatial diamond then $\D^+_\et(Y, \F\ell) = \D^+(Y_\et, \F\ell)$ and thus $\Gamma(Y, M)$ just denotes the usual étale cohomology.
\end{remark}

The following result helps detect $\ell$-bounded prespatial diamonds in some occations. In the next subsection we prove much more powerful $\ell$-boundedness results.

\begin{lemma} \label{sec:bound-cond-1-quasi-pro-etale-over-ell-bounded-implies-ell-bounded}
  Let $f\colon Y'\to Y$ be a qcqs quasi-pro-\'etale morphism of diamonds.
  \begin{lemenum}
    \item If $f$ is separated, then $f_\ast$ is $t$-exact.
    \item If $Y$ is an $\ell$-bounded prespatial diamond, then $Y'$ is $\ell$-bounded and prespatial. 
  \end{lemenum}
\end{lemma}
\begin{proof}
  Part (i) follows from \cref{sec:bound-cond-new-quasi-pro-etale-base-change} and \cite[Remark 21.14]{etale-cohomology-of-diamonds}. To prove (ii), note that it follows easily from \cite[Corollary~11.29]{etale-cohomology-of-diamonds} that $Y'$ is prespatial. To prove $\ell$-boundedness of $Y'$, we observe that $Y'$ is qcqs and $f$ is locally separated by definition of quasi-pro-étale maps, hence we can assume that $f$ is separated. But then the $t$-exactness of $f_\ast$ reduces $\ell$-boundedness of $Y'$ to $\ell$-boundedness of $Y$. 
\end{proof}

\subsection{Cohomologically bounded maps} \label{sec:bound-cond}

Fix a ring $\Lambda$ such that $\ell^n \Lambda = 0$ for some $n \ge 0$. The goal of this subsection is to provide a relative version of the definition of $\ell$-bounded prespatial diamonds from \cref{def:ell-bounded-prespatial-diamond}. This is more subtle than one might think at first, but we have found a definition that works well. In fact, the $p$-cohomological dimension of points in a diamond can increase after base change, and hence in the definition of a $p$-bounded morphism one needs to ask for bounds of these.

We briefly recall the $\ell$-adic cohomological dimension of points in a diamond. Let $Y$ be a quasi-separated diamond, e.g., a prespatial diamond (see the previous subsection). As is proven in \cite[Proposition 21.9]{etale-cohomology-of-diamonds} for each maximal point $y\in Y$ the ``residual'' subdiamond $Y_y\subseteq Y$ with $|Y_y|=\{y\}$ is of the form $Y_y=\Spd(C_y,\ri_{C_y})/G_y$ with $C_y$ an algebraically closed perfectoid field of characteristic $p$, and $G_y$ a profinite group acting continuously and faithfully on $C_y$. The pair $(C_y,G_y)$ is unique up to (non-unique) isomorphism, and $Y_y$ agrees with the stack quotient $[\Spd(C_y,\ri_{C_y})/G_y]$, i.e., the morphism $\Spd(C_y,\ri_{C_y})\times \underline{G_y}\to \Spd(C_y,\ri_{C_y})\times_{\Spd(\F{p}} \Spd(C_y,\ri_{C_y})$ is a monomorphism. The number $\mathrm{cd}_\ell(y):=\mathrm{cd}_\ell(G_y)$, with $\mathrm{cd}_\ell(G_y)$ the $\ell$-cohomological dimension of $G_y$, is called the \emph{$\ell$-adic cohomological dimension of $y$}. With this definition at hand, we can now come to the definition of $\ell$-bounded maps:

\begin{definition} \label{def:ell-bounded-map}
A morphism $f\colon Y'\to Y$ of small $v$-stacks is called $\ell$-bounded if
\begin{enumerate}[(i)]
  \item $f$ is locally separated and representable in prespatial diamonds (\cite[Definition 3.3]{mod-ell-stacky-6-functors}), in particular $f$ is qcqs,

  \item after pullback to every prespatial diamond, $f$ has finite $\dimtrg$ (cf. \cite[Definition~21.7]{etale-cohomology-of-diamonds}), and

  \item for all $X\to Y$ with $X$ a strictly totally disconnected space there exists some $d_X\geq 0$ such that for all maximal points $y'\in Y'\times_Y X$ the $\ell$-cohomological dimension $\mathrm{cd}_\ell(y')$ is bounded by $d_X$ (\cite[Definition 21.10]{etale-cohomology-of-diamonds}).
\end{enumerate}
\end{definition}

\begin{remarks}
  \label{sec:bound-cond-new-for-l-neq-p-cohom-dimension-not-necessary}
  \begin{remarksenum}
  \item If $\ell\neq p$, then the first two conditions in the definition of an $\ell$-bounded morphism imply the third by \cite[Proposition 21.16]{etale-cohomology-of-diamonds}.
  \item It is not known if the condition of having finite $\dimtrg$ can be checked $v$-locally on the source (cf.\ \cite[Section 21)]{etale-cohomology-of-diamonds}). However, it is stable under pullback by \cite[Remark 21.8]{etale-cohomology-of-diamonds}.
    \item Assume that $f\colon Y'\to Y$ is a morphism between spatial diamonds. By \cite[Lemma 21.6]{etale-cohomology-of-diamonds} the dimension of $f$ as a morphism of spectral spaces, i.e., the supremum of the Krull dimensions of the fibers of $f$ (\cite[Definition 21.1]{etale-cohomology-of-diamonds}), is bounded by $\dimtrg$. Indeed, this reduces to the case that $Y=\Spa(C,C^+)$, and then to the case that $Y'$ is a perfectoid space (as quasi-pro-\'etale surjections don't change the Krull dimension of spatial diamonds).
  \end{remarksenum}
\end{remarks}

It is a priori not clear that $\ell$-bounded maps are stable under composition. This is true and will be shown in \cref{sec:bound-cond-new-stability-properties-of-ell-bounded-maps} below, after having established the necessary prerequisites. Let us start our discussion of $\ell$-bounded maps with the following observation, telling us that $\ell$-bounded maps have finite $\ell$-cohomological dimension in a suitable sense. This can be seen as a version of \cite[Theorem 22.5]{etale-cohomology-of-diamonds} and \cite[Corollary 3.9]{mod-ell-stacky-6-functors}.

\begin{proposition} \label{sec:bound-cond-new-ell-bounded-map-has-locally-bounded-cohom-dimension}
  Let $f\colon Y'\to Y$ be an $\ell$-bounded morphism of small $v$-stacks. Then pushforward along $f$ has locally bounded cohomological dimension.

  More precisely: After pullback along any $X\to Y$ with $X$ a diamond, the pushforward $f'_\ast\colon \D_\et^+(X',\Lambda)\to \D_\et^+(X,\Lambda)$ along the base change $f'\colon X':=Y'\times_Y X\to X$ has finite cohomological dimension (with bound potentially depending on $X$).
\end{proposition}
If $\ell\neq p$, then by \cite[Corollary 3.9]{mod-ell-stacky-6-functors} the pushforward $f_\ast$ has cohomological dimension bounded by $3\cdot \dimtrg(f)$. Note that in this case $f_\ast$ commutes with base change by \cite[Proposition 17.6]{etale-cohomology-of-diamonds} as $f$ is qcqs.
\begin{proof}
  We may assume that $Y$ is a diamond, and then by \cref{sec:bound-cond-new-quasi-pro-etale-base-change} that $Y$ is a strictly totally disconnected perfectoid space. We change notation and assume that $f\colon X\to Y$ is an $\ell$-bounded morphism with $Y$ a strictly totally disconnected perfectoid space. As $f$ is locally separated and quasi-compact, we may assume that $X$ is separated. Let $d\geq 1$ be such that $\mathrm{cd}_\ell(x)\leq d$ for all maximal points $x\in X$. We claim that $f_\ast\colon \D^+_\et(X,\Lambda)\to \D^+_\et(Y,\Lambda)=\D^+(Y_\et,\Lambda)$ has a cohomological amplitude bounded in terms of $d$ and $\dimtrg(f)$ (in fact $d+2\cdot \dimtrg(f)$). This statement can be checked on stalks and hence we may assume that $Y=\Spa(C,C^+)$ for some algebraically closed, perfectoid field of characteristic $p$ and $C^+\subseteq C$ an open and bounded valuation subring. Using that $X\to \overline{X}^{/Y}$ is quasi-pro-\'etale and separated, we may by \cref{sec:bound-cond-1-quasi-pro-etale-over-ell-bounded-implies-ell-bounded} replace $X$ by $\overline{X}^{/Y}$ and assume that $f$ is proper (note that these reductions did not change $\dimtrg(f)$ or the cohomological bound $d$ on the cohomological dimension of all maximal fields).
  As in \cite[Theorem 22.5.]{etale-cohomology-of-diamonds}, or \cite[Proposition 3.7]{mod-ell-stacky-6-functors}, we may assume that $A\in \D^+_\et(X,\Lambda)$ is concentrated in degree $0$ and $j^\ast A=0$, where $j\colon f^{-1}(U)\to X$ is the base change of the open immersion $U\to Y$ with complement the unique closed point of $Y$. More precisely, the use of \cite[Theorem 19.2]{etale-cohomology-of-diamonds} is replaced by \cref{sec:bound-cond-1-proper-base-change-mod-p}. If $X$ is a spatial diamond, then as in \cite[Theorem 22.5.]{etale-cohomology-of-diamonds} we may apply \cite[Proposition 21.11]{etale-cohomology-of-diamonds} and conclude that $R^if_\ast(A)=0$ for $i>d+\dimtrg(f)$.

  It remains to handle the more general case that $X$ is only \emph{pre}spatial. But note that the above argument proves the theorem in the general case that $f$ is proper and representable in spatial diamonds (with bound $d_X+\dimtrg(f)$ for $d_X$ a cohomological bound of the cohomological dimension of all maximal points in the base change to the strictly totally disconnected space $X$). Thus the argument of \cite[Proposition 3.7]{mod-ell-stacky-6-functors} applies without change (and yields the bound $d_X+2\cdot \dimtrg(f)$), proving the prespatial case as well.
\end{proof}

\begin{remark}
In the setting of \cref{sec:bound-cond-new-ell-bounded-map-has-locally-bounded-cohom-dimension}, if $\ell \ne p$ then $f$ has $\ell$-cohomological dimension $\le 3\dimtrg(f)$ by \cite[Proposition~3.7]{mod-ell-stacky-6-functors} and this constant is independent of base-change. However, in the case $\ell = p$ it is unclear to us if there is a $p$-cohomological bound on $f$ that works globally (i.e. after every base-change).

If $f\colon Y'=\Spd(C',\ri_{C'})/G\to Y=\Spd(C,\ri_C)$ is a morphism of prespatial diamonds with $C\subseteq C'$ algebraically closed perfectoid fields and $\mathrm{cd}_p G, \mathrm{dimtrg}(f)$ finite, then for $X=\Spd(D,\ri_D)\to Y$ with $D$ another algebraically closed perfectoid field, it is not clear to us if the cohomological dimension of all maximal points of $X\times_Y Y'$ can be bounded by $\mathrm{cd}_p G$. It can nevertheless be bounded by $\mathrm{cd}_pG+1$ (because fields in characteristic $p$ have $p$-cohomological dimension $\leq 1$ by \cite[Proposition 6.5.10]{cohomology-of-number-fields}). 
\end{remark}

The next result relates the notion of $\ell$-bounded maps to the notion of $\ell$-bounded prespatial diamonds from \cref{def:ell-bounded-prespatial-diamond}, providing the expected compatibility:

\begin{proposition} \label{sec:bound-cond-new-absolute-and-relative-notion-of-ell-boundedness}
  Let $f\colon Y'\to Y$ be a locally separated morphism of prespatial diamonds such that $\dimtrg(f) < \infty$ and $Y$ is $\ell$-bounded. Then $f$ is $\ell$-bounded if and only if $Y'$ is $\ell$-bounded.
\end{proposition}
\begin{proof}
  Assume that $Y'$ is $\ell$-bounded. By \cite[Lemma 3.4.(v)]{mod-ell-stacky-6-functors} $f$ is representable in prespatial diamonds. We have to check the third condition for an $\ell$-bounded map in \cref{def:ell-bounded-map}. Let $g\colon X\to Y$ be a morphism from a strictly totally disconnected space $X$ and let $X':=Y'\times_Y X$ with projections $f'\colon X'\to X$ and $g'\colon X'\to Y'$. Let $d\geq 0$ such that $H^i_\et(Y^\prime,\mathcal{F})=0$ for $i>d$ and $\mathcal{F}\in \D^+_\et(Y',\Lambda)$, which are concentrated in degree $0$. Let $x'\in X'$ be a maximal point, and $y':=g'(x')\in Y'$. Then $y'\in Y'$ is a maximal point, and hence $\{y'\}\subseteq |Y'|$ is the underlying space of a subdiamond $j\colon Y'_{y'}\to Y'$. As $Y'$ is quasi-separated, $Y'_{y'}$ is quasi-separated, and hence by \cite[Propsition 21.9]{etale-cohomology-of-diamonds} of the form $Y'_{y'}=\Spd(C_{y'},\ri_{C_{y'}})/G_{y'}$ for an algebraically closed perfectoid field $C_{y'}$ of characteristic $p$ and $G_{y'}$ a profinite group acting continuously and faithfully on $C_{y'}$. We note that $j$ is quasi-compact, separated and quasi-pro-\'etale (\cite[Corollary 10.6]{etale-cohomology-of-diamonds}), and hence $j_\ast\colon \D^+_\et(Y'_{y'},\Lambda)\to \D^+_\et(Y',\Lambda)$ has cohomological amplitude $0$, i.e., is $t$-exact, by \cite[Remark 21.14]{etale-cohomology-of-diamonds} (resp.\ \cref{sec:bound-cond-1-quasi-pro-etale-over-ell-bounded-implies-ell-bounded}). This implies that $\mathrm{cd}_\ell(G_{y'})\leq d$ because \'etale cohomology of $Y'_{y'}$ identifies with continuous cohomology of $G_{y'}$. Set $Z:=X\times_Y Y'_{y'}$ and note that $X'_{x'}\subseteq Z\subseteq X'$. Applying the same argument to $X'_{x'}\subseteq Z$ it suffices to see that $Z$ is $\ell$-bounded (with a bound depending only on $d$ and $\dimtrg$). Now, the affinoid perfectoid space $Z':=Z\times_{Y'_{y'}} \Spd(C_{y'},\ri_{y'})\cong X\times_{Y} \Spd(C_{y'},\ri_{y'})$ is a $G_{y'}$-torsor over $Z$, and hence it suffices to show that $Z'$ is $\ell$-bounded (with a bound depending only on $\dimtrg(f)$). Note that for $h\colon Z'\to X$ we have $\dimtrg(h)\leq \dimtrg(f)$. Hence, if $\ell\neq p$, we can apply \cite[Corollary~3.9]{mod-ell-stacky-6-functors} and the fact that $X$ is strictly totally disconnected to conclude that $Z'$ is $\ell$-bounded (with bound $3\dimtrg(h)\leq 3\dimtrg(f)$). If $\ell=p$, then we can apply \cref{sec:bound-cond-1-totally-disconnected-space-is-p-bounded} below and conclude that $Z'$ is $\ell$-bounded (with bound $3+\dimtrg(h)\leq 3+\dimtrg(f)$). This finishes the proof that $f$ is $\ell$-bounded.

  Assume conversely that $f$ is $\ell$-bounded. As $Y$ is $\ell$-bounded and $Y'$ prespatial, it suffices to show that $f_\ast\colon \D^+_\et(Y^\prime,\Lambda)\to \D^+_\et(Y,\Lambda)$ has finite cohomological dimension. This is implied by \cref{sec:bound-cond-new-quasi-pro-etale-base-change,sec:bound-cond-new-ell-bounded-map-has-locally-bounded-cohom-dimension}.
\end{proof}

The previous result made use of the following computation in the case $\ell = p$, which is of independent interest to us and is therefore extracted here:

\begin{lemma} \label{sec:bound-cond-1-totally-disconnected-space-is-p-bounded}
  Let $X$ be an affinoid perfectoid space which admits a map $f\colon X^\flat\to Z$ to a totally disconnected space $Z$ with $\dimtrg(f)<\infty$. Then $X$ is $p$-bounded and the constant $d$ in \cref{def:ell-bounded-prespatial-diamond} can be chosen to be $d = \dimtrg(f) + 3$.
\end{lemma}
\begin{proof}
  We may replace $X$ by $X^\flat$ and hence assume that $X$ is of characteristic $p$. We first deal with the case that $X=Z$ is totally disconnected. Let $\pi\colon X_\et\to \pi_0(X)_\et$ be the natural morphism of sites. As $\pi_0(X)_\et$ has cohomological dimension $0$ (see \cite[Lemma 0A3F]{stacks-project}), it suffices to show that $R^i\pi_\ast(\mathcal{F})=0$ for $i>2$ and any static $\mathbb{F}_p$-sheaf on $X_\et$. This can be checked on stalks, and hence we may assume that $X=\Spa(K,K^+)$ is a perfectoid affinoid field in characteristic $p$, i.e., $K$ is a perfectoid field of characteristic $p$ and $K^+\subseteq K$ an open and bounded valuation subring. We may write $(K,K^+)$ as a (completed) filtered colimit of perfectoid affinoid fields $(L,L^+)$ and therefore reduce (via \cite[Proposition 14.9]{etale-cohomology-of-diamonds}) to the case that $K^+$ has finite Krull dimension. Let $j\colon U\to X$ be the complement of the closed point. Then $U=\Spa(K,\widetilde{K}^+)$ for some valuation ring $K^+\subseteq \widetilde{K}^+$. By \cite[Lemma 21.13]{etale-cohomology-of-diamonds} the fiber $\mathcal{G}:=[\mathcal{F}\to Rj_\ast j^\ast\mathcal{F}]$ lies in $\D^{[0,1]}(X_\et,\mathbb{F}_p)$. By \cite[Proposition 21.15]{etale-cohomology-of-diamonds} we can conclude that $H^i(X_\et,\mathcal{G})=0$ for $i\geq 3$ because $\Spa(K,\mathcal{O}_K)$ has $p$-cohomological dimension $\leq 1$ (using $\Spa(K,\mathcal{O}_K)_\et\cong \Spec(K)_\et$ and \cite[Proposition 6.5.10]{cohomology-of-number-fields}). This implies that $H^i(X_\et,\mathcal{F})\cong H^i(U_\et,\mathcal{F}_{|U})$ for $i\geq 3$. By induction this reduces to the case $X=\Spa(K,\mathcal{O}_K)$, in which case $H^i(X_\et,\mathcal{F})=0$ for $i\geq 2$.

  Now assume that $X=\Spa(A,A^+)$ is a general affinoid perfectoid space of characteristic $p$ and that $f\colon X\to Z$ is a morphism to a totally disconnected space $Z=\Spa(R,R^+)$ with $\dimtrg(f)<\infty$. As we have proven that $Z$ is $p$-bounded, it suffices to show that $f_\ast\colon \D^+(X_\et,\F{p})\to \D^+(Z_\et,\F{p})$ has cohomological dimension bounded by $1+\dimtrg(f)$. Using again that fields of characteristic $p$ have cohomological dimension $\leq 1$ (\cite[Proposition 6.5.10]{cohomology-of-number-fields}) it is clear from the definition that $f$ is a $p$-bounded morphism (as is any morphism between \textit{affinoid} perfectoid spaces with finite $\dimtrg$). Thus, the assertion follows from \cref{sec:bound-cond-new-ell-bounded-map-has-locally-bounded-cohom-dimension} and going through the proof yields the bound $1+\dimtrg(f)$. This finishes the proof.  
\end{proof}

With the previous results at hand we can finally check that $\ell$-bounded maps are stable under composition:

\begin{lemma} \label{sec:bound-cond-new-stability-properties-of-ell-bounded-maps}
$\ell$-bounded morphisms of small $v$-stacks are stable under composition and base change.
\end{lemma}
\begin{proof}
  Stability under base change is clear by definition. Stability under composition follows from \cref{sec:bound-cond-new-absolute-and-relative-notion-of-ell-boundedness}. More precisely, assume that $f\colon Y\to X, g\colon Z\to Y$ are $\ell$-bounded morphisms. We can assume that $X$, $Y$ and $Z$ are prespatial diamonds. Then $\dimtrg(f\circ g)\leq \dimtrg(f)+\dimtrg(g)$ by \cite[Lemma 21.3]{etale-cohomology-of-diamonds} and clearly $f\circ g$ is representable in prespatial diamonds. To check the third condition for an $\ell$-bounded morphism in \cref{def:ell-bounded-map} we may assume that $X$ is strictly totally disconnected, and hence $\ell$-bounded. Then by \cref{sec:bound-cond-new-absolute-and-relative-notion-of-ell-boundedness} $Y$ is $\ell$-bounded, which then implies by \cref{sec:bound-cond-new-absolute-and-relative-notion-of-ell-boundedness} that $Z$ is $\ell$-bounded. Applying \cref{sec:bound-cond-new-absolute-and-relative-notion-of-ell-boundedness} again shows that $f\circ g$ is $\ell$-bounded as desired.
\end{proof}

\subsection{\texorpdfstring{$\omega_1$}{ω\_1}-solid sheaves}
\label{sec:omeg-solid-sheav}

In the following we bring the main definitions and results from \cite{mann-nuclear-sheaves} to our situation of interest, allowing explicitly the case $\ell = p$. Recall the definition of $\ell$-bounded (pre)spatial diamonds from \cref{def:ell-bounded-prespatial-diamond}.

From now on we fix a static ring $\Lambda$, which is an adic and profinite $\Z_\ell$-algebra. Then $\Lambda\cong \varprojlim_{n}\Lambda/I^n$ for finite rings $\Lambda/I^n$ and a finitely generated ideal of definition $I\subseteq \Lambda$ containing $\ell$. We will abuse notation and write again $\Lambda$ for the pro-\'etale sheaf $\varprojlim_{n} \Lambda/I^n$ with inverse limit formed on the quasi-pro-\'etale site of a locally spatial diamond.

Before we can define nuclear $\Lambda$-sheaves on $\ell$-bounded spatial diamonds, we need to study solid and $\omega_1$-solid $\Lambda$-sheaves as considered in \cite[Section VII.1]{fargues-scholze-geometrization} and \cite[Section 2]{mann-nuclear-sheaves}. In the following, a quasi-pro-\'etale map $U\to X$ of spatial diamonds is called basic if it can be written as a cofiltered inverse limit of \'etale, quasi-compact and separated maps $U_i\to X$ (\cite[Definition 2.2]{mann-nuclear-sheaves}). We note that by \cite[Definition 10.1.(i)]{etale-cohomology-of-diamonds} quasi-pro-\'etale maps of diamonds are locally separated by definition, and hence the basic quasi-pro-\'etale maps form a basic for $X_\qproet$.

\begin{definition}
  \label{sec:omeg-solid-sheav-1-definition-solid-sheaves}
Let $X$ be a spatial diamond.
\begin{defenum}
	\item For every basic quasi-pro-étale map $U = \varprojlim_i U_i \to X$ (see \cite[Definition 2.2]{mann-nuclear-sheaves}) we denote $\Lambda_\solid[U] := \varprojlim_i \Lambda[U_i]\in \D(X_{\qproet},\Lambda)$.\footnote{The quasi-pro-\'etale site is defined in \cite[Definition 14.1.(ii)]{etale-cohomology-of-diamonds}.}

	\item A quasi-pro-étale sheaf $\mathcal M \in \D(X_\qproet, \Lambda)$ is called \emph{solid} if for all basic quasi-pro-étale $U \to X$ the natural map
	\begin{align*}
		\IHom(\Lambda_\solid[U], \mathcal M) \isoto \IHom(\Lambda[U], \mathcal M)
	\end{align*}
	is an isomorphism. We denote by $\D_\solid(X, \Lambda) \subset \D(X_\qproet, \Lambda)$ the full subcategory spanned by the solid sheaves.

	\item Assume $X$ is $\ell$-bounded. A solid sheaf $\mathcal M \in \D_\solid(X, \Lambda)$ is called \emph{$\omega_1$-solid} if for every $\omega_1$-filtered colimit $U = \varprojlim_i U_i$ of basic objects in $X_\qproet$ the natural map
	\begin{align*}
		\varinjlim_i \Gamma(U_i, \mathcal M) \isoto \Gamma(U, \mathcal M)
	\end{align*}
	is an isomorphism. We denote by $\D_\solid(X, \Lambda)_{\omega_1}$ the full subcategory spanned by the $\omega_1$-solid objects.
        \item An object $M\in \D(X_\qproet,\Lambda)$ is called complete if it is $I$-adically complete for some finitely generated ideal of definition $I\subseteq \Lambda$.
\end{defenum}
\end{definition}

\begin{proposition}
  \label{sec:omeg-solid-sheav-1-proposition-properties-of-solid-sheaves}
  Let $X$ be a spatial diamond.
  \begin{propenum}
  \item The category $\D_\solid(X, \Lambda)$ is stable under limits and colimits in $\D(X_\qproet, \Lambda)$ and contains all étale sheaves.
  \item $\D_\solid(X,\Lambda)$ is compactly generated and a collection of compact generators is given by $\Lambda_\solid[U]$ for w-contractible basic $U \in X_\qproet$. Moreover, for each basic quasi-pro-\'etale $U\to X$ the object $\Lambda_\solid[U]$ is static.
  \item The $t$-structure on $\D(X_\qproet, \Lambda)$ restricts to a $t$-structure on $\D_\solid(X, \Lambda)$.
    \item The composition $\D_\solid(X,\Lambda)\to \D(X_\qproet,\Lambda)\to \D(X_v,\Lambda)$ is fully faithful, and $X\mapsto \D_\solid(X,\Lambda)$ is a $v$-sheaf of $\infty$-categories.
  \item The inclusion $\D_\solid(X,\Lambda)\to \D(X_\qproet,\Lambda)$ admits a left adjoint $(-)_\solid$ such that $(\Lambda[U])_\solid=\Lambda_\solid[U]$ for $U\to X$ basic quasi-pro-\'etale.
  \item There exists a unique closed symmetric monoidal structure $-\otimes^\solid_\Lambda-$ on $\D_\solid(X,\Lambda)$ such that $(-)_\solid$ is symmetric monoidal. We have $\Lambda_\solid[U]\otimes^\solid_\Lambda \Lambda_\solid[U^\prime]\cong \Lambda_\solid[U\cprod_X U^\prime]$ for $U,U^\prime\to X$ basic quasi-pro-\'etale.
    \item If $M,N\in \D_\solid^-(X,\Lambda)$ are complete, then $M\otimes_\Lambda^\solid N$ is complete.
    \item For $M\in \D(X_\qproet,\Lambda)$ and $N\in \D_\solid(X,\Lambda)$ we have $\IHom_{\D(X_\qproet,\Lambda)}(M,N)\in \D_\solid(X,\Lambda)$.
  \item The natural functor $\D_\solid(X,\Lambda)\to \mathrm{Mod}_\Lambda(\D_\solid(X,\Z_\ell))$ is an equivalence.
  \end{propenum}
\end{proposition}
In \cite[Defintion VII.1.17]{fargues-scholze-geometrization}, $\D_\solid(X,\Lambda)$ is defined as $\Lambda$-modules in $\D_\solid(X,\Z_\ell)$ for any solid $\Z_\ell$-algebra $\Lambda$. If $\Lambda$ is an adic and profinite $\Z_\ell$-algebra, then \cref{sec:omeg-solid-sheav-1-proposition-properties-of-solid-sheaves} shows that the neat characterization \cref{sec:omeg-solid-sheav-1-definition-solid-sheaves} holds.
\begin{proof}
  We first assume that $\Lambda=\mathbb{Z}_\ell$. We first note that our (a priori different) definition \cref{sec:omeg-solid-sheav-1-definition-solid-sheaves} agrees with \cite[Definition VII.1.10]{fargues-scholze-geometrization}. Indeed, if $M\in \D(X_\qproet,\Z_\ell)$ has solid cohomology objects in the sense of \cite[Definition VII.1.1]{fargues-scholze-geometrization} then by \cite[VII.1.12]{fargues-scholze-geometrization} $M$ is solid in the sense of \cref{sec:omeg-solid-sheav-1-definition-solid-sheaves}. Moreover, the full subcategory $\mathcal{C}\subseteq \D(X_\qproet,\Z_\ell)$ defined in \cite[Definition VII.1.10]{fargues-scholze-geometrization} is stable under all colimits and limits in $\D(X_\qproet,\Z_\ell)$ (by \cite[Theorem VII.1.3]{fargues-scholze-geometrization}) and contains $\Z_\ell[U]$ for any basic quasi-pro-\'etale $U\to X$. Let $N\in \D_\solid(X,\Z_\ell)$. Then $N=\varinjlim_i\Lambda[U_i]$ is a colimit in $\D(X,\Z_\ell)$ for basic quasi-pro-\'etale morphisms $U_i\to X$ with $U_i$ $w$-contractible, as these form a basis of $X_\qproet$. This implies that $N$ is a retract of $\varinjlim_i \Lambda_{\solid}[U_i]$, and thus contained in $\mathcal{C}$. This finishes the claim that $\mathcal{C}=\D_\solid(X,\Z_\ell)$.
  Now, the assertions follow from \cite[Theorem VII.1.3]{fargues-scholze-geometrization}, \cite[Proposition VII.1.13]{fargues-scholze-geometrization} and \cite[Proposition VII.1.14]{fargues-scholze-geometrization}, \cite[Proposition VII.1.8]{fargues-scholze-geometrization}, \cite[Proposition VII.1.11]{fargues-scholze-geometrization} and \cite[Proposition 2.8]{mann-nuclear-sheaves} (which does not use the assumption $\ell\neq p$). This finishes the case that $\Lambda=\Z_\ell$.
  
  Let now $\Lambda$ be a general adic and profinite $\Z_\ell$-algebra. We claim that the natural map
  \[
    \Z_{\ell,\solid}[U]\otimes^\solid\Lambda\to \Lambda_\solid[U]
  \]
  is an isomorphism for any basic quasi-pro-\'etale $U\to X$. Now, $\Lambda=U\to X$ is represented (as a sheaf on $X_\qproet$) by a basic quasi-pro-\'etale morphism over $X$ because $\Lambda$ is adic and profinite. In particular, we see that $\Lambda$ (being solid) is a retract of $\Z_{\ell,\solid}[\Lambda]$. We know that
  \[
    \Z_{\ell,\solid}[U]\otimes^\solid \Z_{\ell,\solid}[\Lambda]\cong \Z_{\ell,\solid}[U\cprod_X \Lambda],
  \]
  which implies the claim by the passage to a retract (and it implies that $\Lambda_\solid[U]$ is static).
  From here it is formal that an object $C\in \D(X_\qproet, \Lambda)$ is solid if and only if its underlying object $C_{|\Z_\ell}\in \D(X_\qproet,\Z_\ell)$ is solid. Indeed, if $C$ is solid, then by the same argument as before $C$ is a retract of a colimit of $\Lambda_\solid[U]$'s with $U\to X$ basic quasi-pro-\'etale, and each of the $\Lambda_\solid[U]$ is solid. Conversely, if $C_{|\Z_\ell}$ is solid, then
  \[
    \IHom_\Lambda(\Lambda_\solid[U],C)\cong \IHom_\Lambda(\Lambda\otimes^\solid \Z_{\ell,\solid}[U],C)\cong \IHom_{\Z_\ell}(\Z_{\ell,\solid}[U],C_{|\Z_\ell})\cong \Gamma(U,C)
  \]
  as desired.

  We can conclude that $\D_\solid(X,\Lambda)\subseteq \D(X_\qproet,\Lambda)$ is stable under all colimits and limits, and that it contains all (static) sheaves of $\Lambda$-modules, which are pulled back from $X_\et$. Similarly, the claim on the $t$-structure on $\D_\solid(X,\Lambda)$ follows as well as the existence of the desired left adjoint $(-)_\solid$. To get the symmetric monoidal structure $-\otimes^\solid_\Lambda-$ it suffices to show that if $C,D\in \D_\solid(X,\Lambda)$ then $\IHom_{\D(X_\qproet,\Lambda)}(C,D)$ is contained in $\D_\solid(X,\Lambda)$. By stability of $\D_\solid(X,\Lambda)$ under limits it suffices to handle the case that $C=\Lambda\otimes^\solid_{\Z_\ell}C^\prime$ for some $C^\prime\in \D_\solid(X,\mathbb{Z}_\ell)$. Then it follows from the same assertion for $\Lambda=\Z_\ell$ (implicitly proven in \cite[Proposition VII.1.14]{fargues-scholze-geometrization}).

  Assume that $M,N\in \D^-_\solid(X,\Lambda)$ are complete. From the case $\Lambda=\Z_\ell$ we can conclude that for each $n\geq 0$ the tensor product $M\otimes_{\Z_\ell}^\solid \Lambda^{\otimes n}\otimes_{\Z_\ell}N$ is $\ell$-adically complete. This implies that $M\otimes^\solid_\Lambda N$ is $\ell$-adically complete as the $\ell$-adic completion commutes with uniformly right-bounded geometric realizations. Using a similar argument for preservation of completedness under geometric realizations, we can reduce first to the case that $\Lambda=\mathbb{F}_\ell[[x_1,\ldots, x_r]]$ for some $r\geq 0$, and then to the case that $r=1$. For $\Lambda=\mathbb{F}_\ell[[x]]$ the same argument as in \cite[Proposition 2.8]{mann-nuclear-sheaves} applies. 

  The final claim $\D_\solid(X,\Lambda)\cong\mathrm{Mod}_\Lambda(\D_\solid(X,\Z_\ell))$ follows from the observation that both embed fully faithfully into $\D(X_\qproet, \Lambda)$ with generators $\Z_{\ell,\solid}[U]\otimes^\solid\Lambda\cong \Lambda_\solid[U]$.
  Now, we can conclude that $X\mapsto \D_\solid(X,\Lambda)$ is a $v$-sheaf of $\infty$-categories on spatial diamonds, e.g., using \cite[Corollary 5.2.2.37]{lurie-higher-algebra}.
\end{proof}

We now turn to the better behaved subcategory of $\omega_1$-solid sheaves, which was introduced (for $\ell\neq p$) in \cite{mann-nuclear-sheaves}. We note that the compact generators of $\D_\solid(X,\Lambda)_{\omega_1}$ are much simpler than the ones of $\D_\solid(X,\Lambda)$. In particular, $\Lambda\in \D_\solid(X,\Lambda)_{\omega_1}$ is compact, while in general $\Lambda\in \D_\solid(X,\Lambda)$ is not.

\begin{proposition}
  \label{sec:omeg-solid-sheav-2-properties-of-omega-1-solid-sheaves}
  Let $X$ be an $\ell$-bounded spatial diamond.
  \begin{propenum}
  \item The category $\D_\solid(X, \Lambda)_{\omega_1}$ is stable under colimits and countable limits in $\D_\solid(X, \Lambda)$ and contains all étale sheaves.
  \item $\D_\solid(X,\Lambda)_{\omega_1}$ is compactly generated and a collection of compact generators is given by $\Lambda_\solid[U]$ for sequential limits $U = \varprojlim_n U_n$ with all $U_n \to X$ being étale, quasicompact and separated.
  \item The $t$-structure on $\D_\solid(X, \Lambda)$ restricts to a $t$-structure on $\D_\solid(X, \Lambda)_{\omega_1}$.
  \item The tensor product $-\otimes^\solid_\Lambda-$ in $\D_\solid(X,\Lambda)$ restricts to a closed symmetric monoidal structure on $\D_\solid(X,\Lambda)_{\omega_1}$.
  \item If $M\in \D_\solid(X,\Lambda)_{\omega_1}$ is $\omega_1$-compact, and $N\in \D_{\solid}(X,\Lambda)_{\omega_1}$ compact, then $\IHom_{\D_\solid(X,\Lambda)_{\omega_1}}(M,N)=\IHom_{\D_\solid(X,\Lambda)}(M,N)$.
  \item $\D_\solid(X,\Lambda)_{\omega_1}\cong \mathrm{Mod}_\Lambda(\D_\solid(X,\Z_\ell)_{\omega_1})$.
  \end{propenum}
\end{proposition}
\begin{proof}
  If $\Lambda = \Z_\ell$ then this - except for the claim on internal Hom's - is proven in \cite[Proposition 2.5]{mann-nuclear-sheaves} (note that the proof of that result uses nowhere that $\ell \ne p$). Thus, let $M\in \D_\solid(X,\Lambda)_{\omega_1}$ be $\omega_1$-compact, and let $N\in \D_{\solid}(X,\Lambda)_{\omega_1}$ be compact. Then $M$ is a countable colimit of compact objects, and by stability of $\D_\solid(X,\Lambda)_{\omega_1}$ under countable limits we may therefore assume that $M=\Lambda_\solid[U]$ is compact. As $N$ is compact, we may reduce to the case that it is a countable inverse limit of \'etale sheaves, and thus we may assume that $N$ is an \'etale sheaf. In this case, one checks directly that $\IHom_{\D_\solid(X,\Lambda)}(M,N)$ is $\omega_1$-solid. This finishes the case $\Lambda=\Z_\ell$.

  Back in the case that $\Lambda$ is a static adic and profinite $\Z_\ell$-algebra we observe that $\Lambda \in \D_\solid(X, \Z_\ell)_{\omega_1}$ and that an object $C\in \D_\solid(X,\Lambda)$ is $\omega_1$-solid if and only if $C_{|\mathbb{Z}_\ell}\in \D_\solid(X,\Z_\ell)$ is $\omega_1$-solid. Moreover, $\Lambda_\solid[U] = \Z_{\ell,\solid}[U] \tensor_{\Z_{\ell,\solid}} \Lambda$ by the proof of \cref{sec:omeg-solid-sheav-1-proposition-properties-of-solid-sheaves}. It follows that $\D_\solid(X, \Lambda)_{\omega_1}$ is the same as the category of $\Lambda$-modules in $\D_\solid(X, \Z_\ell)_{\omega_1}$ by \cref{sec:omeg-solid-sheav-1-proposition-properties-of-solid-sheaves}. All assertions now follow from the case $\Lambda=\Z_\ell$.
\end{proof}

We note the following stability of $(\omega_1-)$solid sheaves under pullback and pushforward along maps of spatial diamonds.

\begin{lemma}
  \label{sec:omeg-solid-sheav-2-pullback-and-pushforward-of-solid-sheaves}
  Let $f\colon X'\to X$ be a morphism spatial diamonds, with associated morphisms $f_\qproet\colon X'_\qproet\to X_\qproet, f_v\colon X'_v\to X_v$ on sites.
  \begin{lemenum}
  \item $f^\ast_\qproet\colon \D(X_\qproet,\Lambda)\to \D(X'_\qproet,\Lambda), f^\ast_v\colon \D(X_v,\Lambda)\to \D(X_v,\Lambda)$ restrict to the same $t$-exact, symmetric monoidal functor
    \[
      f^\ast\colon \D_\solid(X,\Lambda)\to \D_\solid(X',\Lambda), 
    \]
    which preserves $\omega_1$-solid objects (if $X',X$ are $\ell$-bounded).
  \item If $f$ is quasi-pro-\'etale, then $f_{\qproet,\ast}\colon \D(X'_\qproet,\Lambda)\to \D(X,\Lambda)$ restricts to a colimit-preserving functor
    \[
      f_\ast\colon \D_\solid(X',\Lambda)\to \D_\solid(X,\Lambda),
    \]
    which preserves $\omega_1$-solid sheaves (if $X',X$ are $\ell$-bounded).
  \item If $\ell\neq p$, then $f_{\qproet,\ast}\colon \D(X'_\qproet,\Lambda)\to \D(X_\qproet,\Lambda), f_{v,\ast}\colon \D(X'_v,\Lambda)\to \D(X_v,\Lambda)$ restrict to the same colimit-preserving functor
    \[
      f_{v,\ast}\colon \D_\solid(X',\Lambda)\to \D_\solid(X,\Lambda),
    \]
    which preserves $\omega_1$-solid sheaves (if $X',X$ are $\ell$-bounded).
  \end{lemenum}
\end{lemma}
\begin{proof}
The first assertion follows from \cite[Proposition VII.1.8]{fargues-scholze-geometrization} (and the proof of \cite[Proposition 2.6]{mann-nuclear-sheaves} for the preservation of $\omega_1$-solid sheaves). The third assertion (for $f_{v,\ast}$) is \cite[Proposition VII.2.1]{fargues-scholze-geometrization}. Using \cite[Corollary 16.8]{etale-cohomology-of-diamonds}, \cite[Corollary 16.9]{etale-cohomology-of-diamonds} and \cite[Corollary 16.10]{etale-cohomology-of-diamonds} the same argument implies the second assertion and the case $f_{\qproet,\ast}$ in the third assertion. The preservation of $\omega_1$-solid sheaves follows as in \cite[Proposition 2.6]{mann-nuclear-sheaves} (replacing the use of \cite[Corollary 16.8.(ii)]{etale-cohomology-of-diamonds} by \cite[Corollary 16.8.(i)]{etale-cohomology-of-diamonds} for the second part). 
\end{proof}

\begin{corollary}
  \label{sec:omeg-solid-sheav-2-solid-sheaves-are-hyper-complete-quasi-pro-etale-sheaf}
  \begin{corenum}
  \item The functors $X\mapsto \D_\solid(X,\Lambda)_{\omega_1}$, $X\mapsto \D_\solid(X,\Lambda)$ are hypercomplete quasi-pro-\'etale sheaves on the site of $\ell$-bounded spatial diamonds.\footnote{By \cref{sec:bound-cond-1-quasi-pro-etale-over-ell-bounded-implies-ell-bounded} being $\ell$-bounded ascents along under quasi-compact, separated quasi-pro-\'etale maps} 
  \item If $\ell\neq p$, then $X\mapsto \D_\solid(X,\Lambda)_{\omega_1}$, $X\mapsto \D_\solid(X,\Lambda)$ are hypercomplete $v$-sheaves on the site of $\ell$-bounded spatial diamonds.
  \end{corenum}
\end{corollary}
\begin{proof}
  The case for $\D_\solid(X,\Lambda)$ is clear by \cite[Proposition VII.1.8]{fargues-scholze-geometrization}.
  Thanks to \cref{sec:omeg-solid-sheav-2-pullback-and-pushforward-of-solid-sheaves} the proof of \cite[Corollary 2.7]{mann-nuclear-sheaves} implies both assertions for $\omega_1$-solid sheaves.
\end{proof}

\subsection{Overconvergent solid sheaves}
\label{sec:nucl-objects-overc}

As in \cref{sec:omeg-solid-sheav} we fix a prime $\ell$ and a static adic and profinite $\Z_\ell$-algebra $\Lambda$.

We now discuss the definition of overconvergent objects in $\D_\solid(X,\Lambda)$ following \cite[Section 6]{mann-nuclear-sheaves}, although with some differences. Most notably, we don't require that overconvergent sheaves are nuclear (in the sense of \cite[Definition 3.1]{mann-nuclear-sheaves}). Note that if $T$ is a profinite set, then we can consider its quasi-pro-\'etale site $T_\qproet$\footnote{For example, defined as the slice $\ast_\proet/T$ with $\ast_\proet$ the pro-\'etale site of a point, \cite[Section 4.3]{proetale-topology}.}, and the full subcategories $\D_\solid(T,\Lambda)_{\omega_1}\subseteq \D_\solid(T,\Lambda)\subseteq \D(T_\qproet,\Lambda)$ similarly defined as in \cref{sec:omeg-solid-sheav-1-definition-solid-sheaves}. These categories satisfy the same properties as in \cref{sec:omeg-solid-sheav-1-proposition-properties-of-solid-sheaves} and \cref{sec:omeg-solid-sheav-2-properties-of-omega-1-solid-sheaves} (this can be checked, e.g., by taking the product of $T$ with $\Spa(C,\mathcal{O}_C)$ for $C$ a perfectoid algebraically closed field in characteristic $p$). In particular, we can call a morphism $S\to T$ of profinite sets ``basic quasi-pro-\'etale'' if $S=\varprojlim_i S_i\to T$ with $S_i\to T$ a local isomorphism for $S_i$ a profinite set.

We recall that if $X$ is a spatial diamond, then there exists a natural morphism $\pi=\pi_X\colon X_\qproet\to \pi_0(X)_\qproet$ of sites such that $\pi^{-1}(T)=T\cprod_{\pi_0(X)} X$ (\cite[Definition 5.5]{mann-werner-simpson}).

\begin{lemma}
  \label{sec:nucl-objects-overc-2-morphism-to-pi-0}
  Let $X$ be a spatial diamond with morphism of sites $\pi\colon X_\qproet\to \pi_0(X)_\qproet$.
  \begin{lemenum}
  \item The functor $\pi^{-1}\colon \widetilde{\pi_0(X)_\qproet}\to \widetilde{X_\qproet}$ commutes with limits, and $\pi_\ast\circ \pi^{-1}=\mathrm{Id}_{\widetilde{\pi_0(X)_\qproet}}$. In fact, if $V\in X_{\qproet}$ is qcqs, then $\pi^{-1}(\mathcal{F})(V)=\mathcal{F}(\pi_0(V))$ for any sheaf of sets $\mathcal{F}$ on $\pi_0(X)$.
    \item The functor $\pi^{-1}\colon \widetilde{\pi_0(X)_\qproet}\to \widetilde{X_\qproet}$ admits a left adjoint $\pi_\natural\colon \widetilde{X_\qproet}\to \widetilde{\pi_0(X)_\qproet}$ such that $\pi_\natural(V\to X)=(\pi_0(V)\to \pi_0(X))$ if $V\in X_\qproet$ is qcqs.
  \end{lemenum}
\end{lemma}
\begin{proof}
  The first assertion is \cite[Lemma 5.7]{mann-werner-simpson}, and the second is implied by \cite[Lemma 5.4]{mann-werner-simpson}. 
\end{proof}

In the strictly totally disconnected case stronger properties hold true (generalizing the results in \cite[Lemma~6.7]{mann-nuclear-sheaves}).

\begin{lemma}
  \label{sec:nucl-objects-overc-1-morphism-to-pi-0-in-std-case}
  Let $X$ be a strictly totally disconnected perfectoid space.
  \begin{lemenum}
  \item If $X$ is strictly totally disconnected, then $\pi^\ast\colon \D(\pi_0(X)_\qproet,\Lambda)\to \D(X_\qproet,\Lambda)$ is fully faithful, and preserves ($\omega_1$-)solid sheaves and all colimits and limits. Its restriction, $\D_\solid(\pi_0(X)_\qproet, \Lambda)\to \D_\solid(X_\qproet,\Lambda)$ is symmetric monoidal.
  \item The functor $\pi_\ast\colon \D(X_\qproet,\Lambda)\to \D(\pi_0(X)_\qproet,\Lambda)$ is $t$-exact and preserves ($\omega_1$-)solid sheaves.
  \item The functor $\pi^\ast\colon \D_\solid(\pi_0(X),\Lambda)\to \D_\solid(X,\Lambda)$ admits a left adjoint $\pi_\natural$, which is symmetric monoidal and preserves $\omega_1$-solid sheaves. Moreover, if $U\to X$ is basic quasi-pro-\'etale, then $\pi_\natural(\Lambda_\solid[U])\cong \Lambda_\solid[\pi_0(U)]$.
  \item For $M\in \D_\solid(X,\Lambda)$ and $N\in \D_\solid(\pi_0(X),\Lambda)$ the natural map
    \[
      \IHom_{\D_\solid(X,\Lambda)}(M,\pi^\ast N)\to \pi^\ast \IHom_{\D_\solid(\pi_0(X),\Lambda)}(\pi_\natural M,N)
    \]
    is an isomorphism. Similarly, with $\D_\solid(-,\Lambda)$ replaced by $\D_\solid(-,\Lambda)_{\omega_1}$.
  \end{lemenum}
\end{lemma}
\begin{proof}
  Note that \cref{sec:nucl-objects-overc-2-morphism-to-pi-0} implies that $\Lambda$ is pulled back from $\pi_0(X)_\qproet$, and thus $\pi^\ast=\pi^{-1}$ (on static sheaves of $\Lambda$-modules). Hence, for fully faithfulness in the first assertion it suffices (by \cref{sec:nucl-objects-overc-2-morphism-to-pi-0}) to see that $\pi_\ast$ is exact on static sheaves of $\Lambda$-modules. This can be checked on extremally disconnected profinite sets $T$ over $\pi_0(X)$. For such $T$ the pullback $Z:=\pi^{-1}(T)=T\cprod_{\pi_0(X)} X$ is $w$-contractible in $X_\qproet$, i.e., each pro-\'etale cover of $Z$ splits. Preservation of ($\omega_1$-)solid sheaves by $\pi^\ast$ follows from commutation of $\pi^{-1}$ with limits and colimits by checking on the generators of $\D_\solid(\pi_0(X),\Lambda)$ resp.\ $\D_\solid(\pi_0(X),\Lambda)_{\omega_1}$ provided by \cref{sec:omeg-solid-sheav-1-proposition-properties-of-solid-sheaves} resp.\ \cref{sec:omeg-solid-sheav-2-properties-of-omega-1-solid-sheaves} (translated to $\pi_0(X)$).

  That $\pi^\ast$ is symmetric monoidal can be checked on generators where it follows by \cref{sec:omeg-solid-sheav-1-proposition-properties-of-solid-sheaves} from the fact that
  \[
    (T\to \pi_0(X))\mapsto T\cprod_{\pi_0(X)}X
  \]
  commutes with fiber products. Here, $T$ is assumed to be profinite.

  For (ii) we have to check that $\pi_\ast$ preserves ($\omega_1$-)solid sheaves. Let $N\in \D_\solid(X,\Lambda)$ and $T\to \pi_0(X)$ profinite, written as $T=\varprojlim_i T_i$ with $T_i$ profinite and $T_i\to \pi_0(X)$ a local isomorphism. Set $\Lambda_\solid[T]:=\varprojlim_i \Lambda[T_i]$ on $\pi_0(X)_\qproet$. Note that $\pi^\ast(\Lambda_\solid[T])=\Lambda_\solid[\pi^{-1}(T)]$ as $\pi^{-1}$ preserves limits. Then we can conclude
    \begin{align*}
      & \IHom(\Lambda_\solid[T],\pi_\ast N) \\
      = & \pi_\ast(\IHom(\Lambda_\solid[\pi^{-1}T]),N)) \\
      = & \pi_\ast(\IHom(\Lambda[\pi^{-1}T],N)) \\
      = & \pi_\ast(\IHom(\pi^{-1}(\Lambda[T]),N)) \\
      = & \IHom(\Lambda[T],\pi_\ast(N))
    \end{align*}
  using that $N$ is solid and the formal identity $\pi^{-1}(\Lambda[T])\cong \Lambda[\pi^{-1}T]$. For for any cofiltered limit of maps $T_j\to \pi_0(X)$ with $T_j$ profinite we have
  \[
    \Gamma(\varprojlim_j T_j,\pi_\ast(N))\cong \Gamma(\pi^{-1}(\varprojlim_j T_j),N)\cong \Gamma(\varprojlim_j \pi^{-1} T_j,N) 
  \]
  as $\pi^{-1}$ preserves limits. From here it follows that $\pi_\ast$ preserves $\omega_1$-solid objects.

  The existence of $\pi_\natural$ follows as $\pi^{-1}$ preserves limits (or from \cite[Lemma 5.4]{mann-werner-simpson}). It follows formally from the formula $\pi^{-1}\mathcal{F}(V)=\mathcal{F}(\pi_0(V))$ that $\pi_\natural(\Lambda_\solid[U])=\Lambda_\solid[\pi_0(U)]$ for $U\to X$ basic quasi-pro-\'etale. Using \cite[Lemma 5.8]{mann-nuclear-sheaves} we can conclude that $\pi_0(U\cprod_X V)=\pi_0(U)\cprod_{\pi_0(X)}\pi_0(V)$ for $U,V$ basic quasi-pro-\'etale over $X$. This implies that $\pi_\natural$ is symmetric monoidal by \cref{sec:omeg-solid-sheav-1-proposition-properties-of-solid-sheaves}. We need to see that $\pi_\natural$ preserves $\omega_1$-solid sheaves. The explicit formula for $\pi_\natural\Lambda_\solid[U]$ and the fact that $\pi_0(\varprojlim_{i}U_i)\cong \varprojlim_{i}\pi_0(U_i)$ for \'etale, separated maps $U_i\to X$ in a cofiltered system $\{U_i\}_{i}$ shows that it suffices to show that $\pi_\natural$ maps $\Lambda_\solid[U]$ to an $\omega_1$-solid sheaf of $U\to X$ is \'etale quasi-compact and separated. This proven in \cref{sec:overc-solid-sheav-explicit-description-of-pi-0-for-etale-morphism}.
  The assertion (iv) follows formally from symmetric monoidality of $\pi_\natural$. In fact, the left hand side is right adjoint as a functor in $N$ to $\pi_\natural(M\otimes^\solid_\Lambda (-))$ while the right hand side is right adjoint to $\pi_\natural(M)\otimes^\solid_\Lambda\pi_\natural(-)$.
\end{proof}

The proof of \cref{sec:nucl-objects-overc-1-morphism-to-pi-0-in-std-case} made use of following result about $\omega_1$-solidness of $\pi_\natural\Lambda[U]$. Be aware that even if $U \injto X$ is an open immersion, the map $\pi_0(U) \to \pi_0(X)$ is in general quite far from being an open immersion itself, which makes the $\omega_1$-solidness of $\pi_\natural\Lambda[U] = \Lambda_\solid[\pi_0(U)]$ more subtle than one might think at first.

\begin{lemma}
  \label{sec:overc-solid-sheav-explicit-description-of-pi-0-for-etale-morphism}
  Let $j\colon U\to X$ be a quasi-compact, separated, \'etale morphism to a strictly totally disconnected space $X$. Then $\pi_\natural\Lambda[U]=\Lambda_\solid[\pi_0(U)]\in \D_\solid(\pi_0(X),\Lambda)$ is $\omega_1$-solid.
\end{lemma}
\begin{proof}
  Write $X=\Spa(A,A^+)$. As $X$ is strictly totally disconnected, each point $u\in U$ has a quasi-compact open neighborhood $V_u$, such that the map $V_u\to X$ is an open immersion. Indeed, this follows from the definition of an \'etale morphism (\cite[Definition 6.2]{etale-cohomology-of-diamonds}) using that each finite \'etale map onto some quasi-compact open in $X$ will split. Hence, without loss of generality $j\colon U\to X$ can be assumed to be a quasi-compact open immersion (using that $\omega_1$-solid sheaves are stable under finite colimits). By \cite[Lemma 7.6]{etale-cohomology-of-diamonds} $U$ is the intersection of subsets $\{|f|\leq 1\}$ for varying $f\in A$. As $U\subseteq X$ is quasi-compact open, $X\setminus U$ is quasi-compact in the constructible topology. As any $\{|f|\leq 1\}$ for $f\in A$ is closed in the constructible topology, this implies that $U$ is the finite intersection of subsets $U_1=\{|f_1|\leq 1\},\ldots, U_r=\{|f_r|\leq 1\}$ with $f_1,\ldots, f_r\in A$. By \cite[Lemma 5.8]{mann-werner-simpson} $\pi_0(U)=\pi_0(U_1)\cprod_{\pi_0(X)}\ldots \cprod_{\pi_0(X)} \pi_0(U_r)$ and thus \cref{sec:omeg-solid-sheav-1-proposition-properties-of-solid-sheaves}, \cref{sec:omeg-solid-sheav-2-properties-of-omega-1-solid-sheaves} imply that we may assume that $U=U_1=\{|f|\leq 1\}$ for $f:=f_1$. Let $\pi\in A$ be a pseudo-uniformizer and set
  \[
    W_n:=\{|f|<|\pi^{-1/p^n}|\}
  \]
  for $n\in \mathbb{N}$. Each $W_n$ is a spectral space, being closed in the constructible topology (\cite[0902]{stacks-project}), and hence $\pi_0(W_n)$ is a profinite set. We claim:
  \begin{enumerate}
  \item[1)] $\pi_0(U), \pi_0(W_n), n\in \mathbb{N},$ are subspaces of $\pi_0(X)$,
    \item[2)] $\bigcap\limits_{n\in \mathbb{N}}W_n$ is the closure of $U$,
    \item[3)] $\pi_0(U)=\bigcap\limits_{n\in \mathbb{N}}\pi_0(W_n)$,
      \item[4)] for each $n\in \mathbb{N}$ the subspace $\pi_0(U)$ lies in the interior of $\pi_0(W_n)\subseteq \pi_0(X)$.
      \end{enumerate}
      The claims imply that $\Lambda_\solid[\pi_0(U)]$ is $\omega_1$-solid. Indeed, we may inductively construct quasi-compact open subsets $S_n\subseteq \pi_0(X)$ with $\pi_0(U)\subseteq S_n\subseteq \pi_0(W_n)$, and $S_{n+1}\to S_n$. Then $\Lambda_\solid[\pi_0(U)]=\varprojlim\limits_{n\in \mathbb{N}}\Lambda[S_n]$ is $\omega_1$-solid by \cref{sec:omeg-solid-sheav-2-properties-of-omega-1-solid-sheaves}.

      To show 1) it suffices (because the $\pi_0$ are profinite and the maps continuous) to show that $\pi_0(U)\to \pi_0(X), \pi_0(W_n)\to \pi_0(X)$ are injective. Now, $U$ is open, and thus stable under generalization while the $W_n$ are closed, hence stable under specializations. This implies that if $Z\subseteq X$ is a connected component, then $Z\cap U$, $Z\cap W_n$ are connected (if non-empty) because $Z$ is a linearly ordered chain of points. This in turn implies injectivity.

      As each $W_n$ is closed, and $U\subseteq W_n$ we see $\overline{U}\subseteq \bigcap\limits_{n\in \mathbb{N}}W_n$. The reverse implication may be checked after intersection with each connected component $Z$ of $X$ (because $\overline{U}$ is the set of specializations of $U$ by \cite[0903]{stacks-project}). If $Z\cap U\neq \empty$, then $\overline{U}\cap Z=Z$ and similarly for the intersection with the $W_n$'s. Let $z$ be the generic point of $Z$ and assume that $Z\cap \overline{U}=\emptyset$ or equivalently $Z\cap U=\emptyset$. Thus, $|f(z)|>1$ and as $z$ is a rank $1$ valuation this implies that $z\notin \{|f|\leq |\pi^{-1/p^{n}}|\}$ for some $n\in \mathbb{N}$. This implies that $Z\cap W_{n}=\emptyset$ as the set $\{|f|\leq |\pi^{-1/p^n}|\}$ is stable under generalizations. Thus $Z\cap \bigcap\limits_{n\in \mathbb{N}}W_n=\emptyset$ as well.

      Assertion 3) follows from 2) because $\pi_0(U)=\pi_0(\overline{U})=\bigcap\limits_{n\in \mathbb{N}}\pi_0(W_n)$.

      We are left with assertion 4). Fix a connected component $Z\subseteq X$ with $Z\cap U\neq \emptyset$, i.e., the point $z\in \pi_0(X)$ defined by $Z$ lies in $\pi_0(U)$. We can write $Z=\bigcap\limits_{i\in I} f^{-1}(T_i)$ with $T_i\subseteq \pi_0(X)$ running through the quasi-compact open neighborhoods of $z$ and $f\colon |X|\to \pi_0(X)$ the natural quotient map. Note that the $f^{-1}(T_i)$ are closed and open in $X$. Moreover, let us fix $n\in \mathbb{N}$. We know that $Z\subseteq W_n$ as $\overline{U}\subseteq W_n$ and the generic point of $Z$ lies in $U$. Furthermore, $X\setminus W_n=\{|\pi^{-1/p^n}|\leq |f|\}$ is a rational open subset of $X$. In particular, $X\setminus W_n$ is compact in the constructible topology. As $X\setminus W_n \cap \bigcap\limits_{i\in I}f^{-1}(T_i)=X\setminus W_n\cap Z=\emptyset$, we can conclude that there exists some $i_0$ such that $X\setminus W_n\cap f^{-1}(T_{i_0})=\emptyset$, i.e., $f^{-1}(T_{i_0})\subseteq W_n$. This implies that $z\in T_{i_0}\subseteq \pi_0(W_n)$. In particular, $z$ lies in the interior of $\pi_0(W_n)$. This finishes the proof.   
\end{proof}

With the previous results at hand, we can now introduce a well-behaved (see \cref{sec:nucl-objects-overc-1-overconvergent-sheaves-on-std-spaces-satisfy-quasi-pro-etale-descent}) notion of overconvergent sheaves on general spatial diamonds:

\begin{definition}
  \label{sec:nucl-objects-overc-1-overconvergent-sheaves}
  \begin{defenum}
  \item Let $X$ be a strictly totally disconnected space with morphism of sites $\pi\colon X_\qproet\to \pi_0(X)_\qproet$. An object $M\in \D_\solid(X,\Lambda)$ is called overconvergent if $M\cong \pi^\ast N$ for some $N\in \D_\solid(\pi_0(X)_\proet,\Lambda)$.
  \item Let $X$ be a spatial diamond. An object $M\in \D_\solid(X,\Lambda)$ is called overconvergent if $f^\ast M$ is overconvergent for any quasi-pro-\'etale morphism $f\colon X'\to X$ with $X'$ strictly totally disconnected.
  \end{defenum}
  We let $\D_\solid(X,\Lambda)^\oc\subseteq \D_\solid(X,\Lambda)$ (resp.\ $\D_\solid(X,\Lambda)^\oc_{\omega_1}\subseteq \D_\solid(X,\Lambda)_{\omega_1}$) denote the full subcategories spanned by the overconvergent objects (resp.\ the overconvergent and $\omega_1$-solid objects).
\end{definition}

\begin{lemma}
  \label{sec:nucl-objects-overc-1-overconvergent-sheaves-on-std-spaces-satisfy-quasi-pro-etale-descent}
  Let $f\colon X'\to X$ be a morphism of strictly totally disconnected spaces.
  \begin{lemenum}
  \item If $M\in \D_\solid(X,\Lambda)$ is overconvergent, then so is $f^\ast M$.
  \item Conversely, if $f$ is a quasi-pro-\'etale cover and $f^\ast M$ is overconvergent, then so is $M$.
  \end{lemenum}
\end{lemma}
\begin{proof}
  The first claim follows by naturality in $X$ of the morphism of sites $\pi\colon X_\qproet\to \pi_0(X)_\qproet$. For the converse we note that (using \cite[Lemma 5.8]{mann-werner-simpson}) it suffices to see that $X\mapsto \D_\solid(\pi_0(X),\Lambda)$ is a hypercomplete quasi-pro-\'etale sheaf on strictly totally disconnected spaces. Now \cref{sec:omeg-solid-sheav-2-solid-sheaves-are-hyper-complete-quasi-pro-etale-sheaf} implies that $T\mapsto \D_\solid(T,\Lambda)$ is a quasi-pro-\'etale sheaf on profinite sets, in other words a condensed sheaf of $\infty$-categories on $\ast_\qproet$. As $X\mapsto \pi_0(X)$ sends quasi-pro-\'etle covers, even $v$-covers, to quasi-pro-\'etale covers the claim follows.
\end{proof}

The next result summarizes the very good properties of the full subcategory of overconvergent objects.

\begin{lemma}
  \label{sec:nucl-objects-overc-1-properties-of-overconvergent-objects}
  Let $X$ be an $\ell$-bounded spatial diamond.
  \begin{lemenum}
  \item Overconvergent objects satisfy quasi-pro-\'etale hyperdescent on the big site of $\ell$-bounded spatial diamonds.
  \item The inclusion $\D_\solid(X,\Lambda)^\oc\subseteq \D_\solid(X,\Lambda)$ admits a symmetric monoidal left adjoint $M\mapsto M_\oc$ and a right adjoint $M\mapsto M^\oc$. In particular, $\D_\solid(X,\Lambda)^\oc$ is stable under all colimits and limits in $\D_\solid(X,\Lambda)$.
  \item The inclusion $\D_\solid(X,\Lambda)^\oc_{\omega_1}\subseteq \D_\solid(X,\Lambda)_{\omega_1}$ admits a symmetric monoidal left adjoint $M\mapsto M_\oc$ and a right adjoint $M\mapsto M^\oc$. In particular, $\D_\solid(X,\Lambda)^\oc_{\omega_1}$ is stable under all colimits and limits in $\D_\solid(X,\Lambda)_{\omega_1}$.
  \item If $M, N\in \D_\solid(X,\Lambda)$ and $N$ is overconvergent, then $\IHom_{\D_\solid(X,\Lambda)}(M,N)$ is overconvergent and naturally isomorphic to $\IHom_{\D_\solid(X,\Lambda)}(M_\oc,N)$. Similarly, with $\D_\solid(X,\Lambda)$ replaced by $\D_\solid(X,\Lambda)_{\omega_1}$.
    \item If $U=\varprojlim\limits_{i\in I}U_i\to X$ with $U_i\to X$ \'etale, quasi-compact and separated, then the natural map $\Lambda_\solid[U]\to \Lambda_\solid[\overline{U}^{/X}]$ induces an isomorphism $\Lambda_\solid[U]_\oc\cong \Lambda_\solid[\overline{U}^{/X}]$. In particular, the functor $\D_\solid(X,\Lambda)\to \D_\solid(X,\Lambda),\ M\mapsto M_\oc$ preserves compact objects. 
    \item The functor $\D_\solid(X,\Lambda)\to \D_\solid(X,\Lambda),\ M\mapsto M^\oc$ preserves colimits and right bounded objects.
  \end{lemenum}
\end{lemma}
\begin{proof}
  Assertion (i) follows from \cref{sec:nucl-objects-overc-1-overconvergent-sheaves-on-std-spaces-satisfy-quasi-pro-etale-descent} and \cref{sec:omeg-solid-sheav-2-solid-sheaves-are-hyper-complete-quasi-pro-etale-sheaf}.
  For assertion (ii) it suffices by the adjoint functor theorem to show that $\D_\solid(X,\Lambda)^\oc$ is stable under all colimits and limits in $\D_\solid(X,\Lambda)$. By \cref{sec:omeg-solid-sheav-1-proposition-properties-of-solid-sheaves} and (i) we may then assume that $X$ is a strictly totally disconnected perfectoid space. Then the assertion follows from \cref{sec:nucl-objects-overc-1-morphism-to-pi-0-in-std-case}. We note that in this case we have $M_\oc=\pi^\ast\pi_\natural(M)$ and $M^\oc=\pi^\ast\pi_\ast(M)$ for $M\in \D_\solid(X,\Lambda)$. As in \cref{sec:nucl-objects-overc-1-morphism-to-pi-0-in-std-case} symmetric monoidality of $(-)_\oc$ implies assertion (iv).

  For assertion (iii) we may again reduce to the case that $X$ is strictly totally disconnected. Indeed, for the existence of $(-)^\oc$ this follows by stability of overconvergent sheaves under colimits (which can be checked on a strictly totally disconnected cover), and the existence of $(-)_\oc$ can be descended from the strictly totally disconnected case as there $(-)_\oc$ commutes with pullback as will be proved in \cref{sec:nucl-objects-overc-1}.
  Then the assertion follows again from \cref{sec:nucl-objects-overc-1-morphism-to-pi-0-in-std-case} as $\pi_\natural, \pi^\ast, \pi_\ast$ all preserve $\omega_1$-solid sheaves.

  Let us note that if $U\to X$ is quasi-pro-\'etale, then $\overline{U}^{/X}\to X=\overline{X}^{/X}$ is quasi-pro-\'etale by \cite[Corollary 18.8.(vii)]{etale-cohomology-of-diamonds}. If $X$ is strictly totally disconnected and $U$ qcqs, then the natural map $\overline{U}^{/X}\to \pi_0(U)\cprod_{\pi_0(X)}X$ is an isomorphism as it is qcqs and bijective if $X$ is connected. To check (v) we may (by \cref{sec:nucl-objects-overc-1} below) reduce to the case that $X$ is strictly totally disconnected. Indeed, in this case $\Lambda_\solid[\overline{U}^{/X}]\cong \Lambda_\solid[\pi_0(U)\cprod_{\pi_0(X)}X]$ is overconvergent, which implies the same for a general $\ell$-bounded spatial diamond. This produces the natural map $\Lambda_\solid[U]_\oc\to \Lambda[\overline{U}^{/X}]$, which then can be checked to be an isomorphism again under the assumption that $X$ is strictly totally disconnected. Here, it follows from \cref{sec:nucl-objects-overc-1-morphism-to-pi-0-in-std-case} as $(-)_\oc\cong \pi^\ast \pi_\natural(-)$.
  It follows from the explicit description that the functor $(-)_\oc$ maps compact objects in $\D_\solid(X,\Lambda)$ (resp.\ $\D_\solid(X,\Lambda)_{\omega_1}$) to compact objects in $\D_\solid(X,\Lambda)$ (resp.\ $\D_\solid(X,\Lambda)_{\omega_1}$). This implies that the inclusion of overconvergent into solid sheaves preserves compact objects, which implies that the right adjoint $(-)^\oc$ (for solid or $\omega_1$-solid objects) preserves colimits. Now pick a right bounded object $M\in \D_\solid(X,\Lambda)_{\omega_1}$. As $X$ is $\ell$-bounded can conclude that there exists some $a\in \Z$ such that for each compact object $N\in \D_\solid(X,\Lambda)_{\omega_1}$, which is supported to the right of $a$ and has terms given by $\Lambda_\solid[U]$ for $U=\varprojlim\limits_{n} U_n\to X$ with $U_n$ \'etale quasi-compact and separated, we have $\Hom(N,M)=0$.
  This implies
  \[
    \Hom(N_\oc,M)=\Hom(N_\oc,M^\oc)=0
  \]
  for any such $N$ by the explict description of $(-)_\oc$ in (v). This shows that $M^\oc$ is again right bounded. 
\end{proof}

The left adjoint $(-)_\oc$ commutes with base change (while the right adjoint $(-)^\oc$ does not, even on strictly totally disconnected spaces).

\begin{lemma}
  \label{sec:nucl-objects-overc-1}
  Let $f\colon X'\to X$ be a morphism of $\ell$-bounded spatial diamonds.
  \begin{lemenum}
    \item For $M\in \D_\solid(X,\Lambda)$ the natural map
  \[
    (f^\ast M)_\oc\to f^\ast(M_\oc)
  \]
  is an isomorphism.
  \item If $f$ is quasi-pro-\'etale, then $f_\ast(\D_\solid(X',\Lambda)^\oc)\subseteq \D_\solid(X,\Lambda)^\oc$ and $f_\ast(\D_\solid(X',\Lambda)^\oc_{\omega_1})\subseteq \D_\solid(X,\Lambda)^\oc_{\omega_1}$
  \end{lemenum}
\end{lemma}
\begin{proof}
  By quasi-pro-\'etale hyperdescent it suffices to treat the case that $X,X'$ are strictly totally disconnected. Let $\pi_{X}\colon X_\qproet\to \pi_0(X)_\qproet, \pi_{X'}\colon X'_\qproet\to \pi_0(X')_\qproet$ be the morphisms of sites. It suffices to check that the natural map $\pi_{X',\natural}f^\ast M\to \pi_0(f)^\ast \pi_{X,\natural}(M)$ is an isomorphism for $M\in \D_\solid(X,\Lambda)$. Passing to right adjoints, it suffices to see that for $N\in \D_\solid(\pi_0(X'),\Lambda)$
  \[
    f_\ast\pi^\ast_{X'}N\cong \pi^\ast_X \pi_0(f)_\ast(N)
  \]
  via the natural map. This follows from \cite[Corollary 5.9]{mann-werner-simpson}.
  This formula (together with \cref{sec:omeg-solid-sheav-2-pullback-and-pushforward-of-solid-sheaves}) imply the second assertion as well if $X$ is strictly totally disconnected. The case for general $X$ follows by quasi-pro-\'etale descent (for which $f_\ast$ satisfies base change by \cref{sec:omeg-solid-sheav-2-pullback-and-pushforward-of-solid-sheaves}).
\end{proof}

\subsection{Nuclear sheaves}
\label{sec:abstr-nucl-sheav}

In this section we will study nuclear objects in $\D_\solid(X,\Lambda)_{\omega_1}$ and show that they are overconvergent. Critically, we will moreover show in \cref{sec:nucl-objects-overc-1-d-nuc-generated-by-complete-objects} that the full subcategory of nuclear objects is generated under colimits by right bounded complete objects. In order to achieve this, we need a better understanding of the compact objects of $\D_\solid(X,\Lambda)_{\omega_1}^\oc$.

Let us recall from \cite[Lecture VIII]{condensed-complex-geometry} that a morphism $f\colon P\to Q$ in a closed symmetric monoidal $\infty$-category $\mathcal{C}$ with compact unit, is called of trace class if it lies in the image of the natural morphism
\[
  \pi_0((\IHom(P,1)\otimes Q)(\ast))\to \pi_0(\Hom(P,Q)).
\]
Here, $\otimes, \IHom, 1$ refer to the tensor product, internal Hom, unit in $\mathcal{C}$, while $(-)(\ast):=\Hom(1,-)$.
We will use critically the following property of trace class maps:
If $P_0\to P_1\to \ldots$ is a sequential diagram with all morphisms $P_n\to P_{n+1}$ of trace class, then for any $Q\in \mathcal{C}$ the natural map
\[
  \varprojlim\limits_{n} (\IHom(P_n,1)\otimes Q)(\ast))\to \varprojlim\limits_n \Hom(P_n,Q)
\]
is an isomorphism. Indeed, each witness in $(\IHom(P_n,1)\otimes P_{n+1})(\ast)$ for $P_n\to P_{n+1}$ being trace class yields a factorization
\[
  \Hom(P_{n+1},Q)\to (\IHom(P_{n},1)\otimes Q)(\ast)\to \Hom(P_n,Q).
\]

Following \cite[Definition 8.5]{condensed-complex-geometry} we make the following definition.

\begin{definition}
  \label{sec:overc-solid-sheav-basic-nuclear-and-conuclear}
  Let $X$ be an $\ell$-bounded spatial diamond and $M\in \D_\solid(X,\Lambda)_{\omega_1}$.
  \begin{defenum}
  \item $M$ is called basic nuclear if $M=\varinjlim (M_0\to M_1\to \ldots)$ with all $M_n\to M_{n+1}$ of trace class.
    \item $M$ is called nuclear if for all compact objects $P\in \D_\solid(X,\Lambda)_{\omega_1}$, the map
  \[
    (\IHom_{\D_\solid(X,\Lambda)_{\omega_1}}(P,\Lambda)\otimes^\solid_\Lambda M)(\ast)\to \Hom_{\D_\solid(X,\Lambda)_{\omega_1}}(P,M)
  \]
  is an equivalence.
  \end{defenum}
   We let $\D_\nuc(X,\Lambda)\subseteq \D_\solid(X,\Lambda)_{\omega_1}$ be the full subcategory of nuclear objects.
\end{definition}

\begin{remark}
  \label{sec:omeg-solid-sheav-1}
  We warn the reader that the abstract notion of nuclearity from \cite[Lecture VIII]{condensed-complex-geometry} that we employ here is different from the more subtle ``geometric nuclearity'' in \cite[Definition 3.1]{mann-nuclear-sheaves} as noted in \cite[Warning 3.3]{mann-nuclear-sheaves}. In \cref{sec:nucl-objects-overc-3-overconvergent-and-nuclear-objects} below we will prove that the abstract notion of nuclearity singles out the \textit{overconvergent} nuclear sheaves from \cite[Definition 6.8]{mann-nuclear-sheaves} (thereby providing details to \cite[Warning 3.3]{mann-nuclear-sheaves}).
\end{remark}

The basic observation for the overconvergence of the nuclear objects in $\D_{\solid}(X,\Lambda)_{\omega_1}$ is the overconvergence of $\IHom_{\D_\solid(X,\Lambda)_{\omega_1}}(\Lambda_\solid[U],\Lambda)$.

\begin{lemma}
  \label{sec:nucl-objects-overc-2-nuclear-objects-are-overconvergent}
  Let $X$ be an $\ell$-bounded spatial diamond.
  \begin{lemenum}
  \item Each trace class map $P\to Q$ in $\D_\solid(X,\Lambda)_{\omega_1}$ factors as $P\to P_\oc\to Q$ with $P_\oc\to Q$ of trace class.
  \item Each nuclear object in $\D_\solid(X,\Lambda)_{\omega_1}$ is overconvergent.
  \end{lemenum}
\end{lemma}
In particular, $\D_\nuc(X,\Lambda)_{\omega_1}$ is equivalent to the category $\mathrm{Nuc}(\D_\solid(X,\Lambda)_{\omega_1}^\oc)$ of nuclear objects in $\D_\solid(X,\Lambda)_{\omega_1}^\oc$. 
\begin{proof}
  By \cite[Theorem 8.6]{condensed-complex-geometry} each nuclear object is a colimit of basic nuclear object. By \cref{sec:nucl-objects-overc-1-properties-of-overconvergent-objects} the subcategory of overconvergent objects is stable under limits and colimits. Hence, it suffices to prove the first assertion to show overconvergence. Let $P\to Q$ be a trace class morphism in $\D_\solid(X,\Lambda)_{\omega_1}$. As $\Lambda$ is overconvergent, we can conclude from \cref{sec:nucl-objects-overc-1-properties-of-overconvergent-objects} that
  \[
    \IHom(P,\Lambda)\cong \IHom(P_\oc,\Lambda)
  \]
  as $(-)_\oc$ is symmetric monoidal. This implies the desired factorization.

\end{proof}

In the strictly totally disconnected case, we can be more explicit and identify nuclear sheaves with nuclear modules over the ring of continuous functions:

\begin{lemma}
  \label{sec:nucl-objects-overc-1-nuclear-objects-are-pulled-back-from-pi-0}
  Let $X$ be a strictly totally disconnected perfectoid space and let $T$ be a profinite set.
  \begin{lemenum}
  \item The functors $\pi_\ast, \pi^\ast$ induce a symmetric monoidal equivalence $\D_\nuc(X,\Lambda)\cong \D_\nuc(\pi_0(X),\Lambda)$.
  \item Let $\alpha\colon T_\qproet\to \ast_\qproet$ be the natural morphism of sites. Then $\alpha_\ast$ induces an equivalence
    \[
      \D_\nuc(T_\qproet,\Lambda)\cong \D_\nuc(C(T,\Lambda)),
    \]
    where $C(T,\Lambda)\cong \alpha_\ast(\Lambda)$ is the nuclear $\Lambda$-algebra of continuous functions $T\to \Lambda$.
  \end{lemenum}
  In particular, $\D_\nuc(X,\Lambda)\cong \D_\nuc(C(X,\Lambda))$.
\end{lemma}
We recall that here ``nuclear'' refers to \cref{sec:overc-solid-sheav-basic-nuclear-and-conuclear} (and not \cite[Definition 3.1]{mann-nuclear-sheaves}). However, for $\pi_0(X)$ both notions agree (by \cref{sec:nucl-objects-overc-1-nuclear-objects-are-pulled-back-from-pi-0} and \cite[Lemma 6.7.(iv)]{mann-nuclear-sheaves} or by \cite[Remark 3.10]{mann-nuclear-sheaves} if $\Lambda=\Z_\ell$).
\begin{proof}
  As $\pi^\ast$ is symmetric monoidal, $\pi^\ast$ preserves nuclear objects. Moreover, the unit $M\to \pi_\ast\pi^\ast M$ is an isomorphism for any $M\in \D_\nuc(\pi_0(X),\Lambda)$ by \cref{sec:nucl-objects-overc-2-morphism-to-pi-0}. Let $N\in \D_\nuc(X,\Lambda)$ be nuclear. We need to see that the counit $\pi^\ast \pi_\ast N\to N$ is an isomorphism. This follows from \cref{sec:nucl-objects-overc-2-nuclear-objects-are-overconvergent}.

    Let us proof the second assertion. Set $A:=C(T,\Lambda)$. As $A$ is a nuclear $\Lambda$-algebra, e.g., by \cref{sec:defin-d_hats-2-complete-and-discrete-mod-i-implies-nuclear}, $\D_\nuc(A)$ is equivalent to $\mathrm{Mod}_A(\D_\nuc(\Lambda))$, cf.\ \cite[Corollary 8.20]{condensed-complex-geometry}.
    The existence of a left adjoint $\alpha^\ast\colon \D_\solid(\ast_\qproet,\Lambda)_{\omega_1}\to \D_\solid(T,\Lambda)_{\omega_1}$ to $\alpha_\ast$ follows as in the proof of \cite[Lemma 6.7.(iv)]{mann-nuclear-sheaves}. Similarly, it follows (using \cref{sec:defin-d_hats-2-complete-and-discrete-mod-i-implies-nuclear}) that $\alpha^\ast$ preserves nuclearity, and that $\alpha_\ast \alpha^\ast M\to M$ is an isomorphism for $M\in \D_\nuc(A)$. It remains to see that $\alpha^\ast$ is essentially surjective. Here, we a priori have to offer a different proof as in \cite[Lemma 6.7.(iv)]{mann-nuclear-sheaves} as we use a different notion of nuclearity. However, \cite[Remark 3.10]{mann-nuclear-sheaves} shows that both notions agree in this case as $T$ is profinite. Thus, the assertion follows from \cite[Lemma 6.7]{mann-nuclear-sheaves}.
\end{proof}

We note that $\alpha_\ast\colon \D_\solid(T,\Lambda)\to \mathrm{Mod}_{C(T,\Lambda)}\D_\solid(\Lambda)$ is in general not an equivalence (if $T$ is not finite), even if $\Lambda=\mathbb{F}_\ell$. Indeed, for each $t\in T$ the morphism $\{t\}\to T$ is pro-\'etale, and thus $\Lambda[\{t\}]\in \D_\solid(T,\Lambda)$ is compact. But if $\Lambda$ is discrete, then $\alpha_\ast(\Lambda[\{t\}])\cong \Lambda$ viewed as a $C(T,\Lambda)$-module via evaluation at $t$. If $t\in T$ is not an open point, then this implies that $\alpha_\ast(\Lambda[\{t\}])$ is not a (discrete) finitely presented $C(T,\Lambda)$-module, and hence not compact in $\mathrm{Mod}_{C(T,\Lambda)}\D_\solid(\Lambda)$. Nuclear objects satisfy quasi-pro-\'etale descent:

\begin{lemma}
  \label{sec:nucl-objects-overc-3-overconvergent-and-nuclear-objects}
  The functor $X\mapsto \D_\nuc(X,\Lambda)$ is a hypercomplete quasi-pro-\'etale sheaf on the big quasi-pro-\'etale site of $\ell$-bounded spatial diamonds.
\end{lemma}
\begin{proof}
  This is formally implied by \cref{sec:omeg-solid-sheav-2-solid-sheaves-are-hyper-complete-quasi-pro-etale-sheaf}: Pullback of nuclear objects are nuclear (by symmetric monoidality of pullbacks) and if $f_\bullet\colon X_\bullet\to X$ is a quasi-pro-\'etale hypercover and $M\in \D_\solid(X,\Lambda)_{\omega_1}$ with $f^\ast_n M$ nuclear for each $n\in \Delta$, then for $P\in \D_\solid(X,\Lambda)_{\omega_1}$ compact we have
  \begin{align*}
    \Hom_{\D_\solid(X,\Lambda)_{\omega_1}}(P,M) = & \varprojlim\limits_{n\in \Delta} \Hom_{\D_\solid(X_n,\Lambda)_{\omega_1}}(f^\ast_nP,f^\ast_nM) \\
    = & \varprojlim\limits_{n\in \Delta} \Hom_{\D_\solid(X_n,\Lambda)_{\omega_1}}(\Lambda,\IHom_{X'_n}(f^\ast_nP,\Lambda)\otimes^\solid_\Lambda f^\ast_n M)\\
    = & \varprojlim\limits_{n\in \Delta} \Hom_{\D_\solid(X_n,\Lambda)_{\omega_1}}(\Lambda,f^\ast_n(\IHom_{X}(P,\Lambda)\otimes^\solid_\Lambda M))\\
    = & \Hom_{\D_\solid(X,\Lambda)_{\omega_1}}(\Lambda, \IHom(P,\Lambda)\otimes^\solid_\Lambda M)
  \end{align*}
  using (in this order) the description of Hom-spaces in limits of categories, nuclearity of all the $f_n^\ast M$, the fact that $f^\ast_n$ is the restriction for some slice (plus \cref{sec:omeg-solid-sheav-2-properties-of-omega-1-solid-sheaves} to control the $\IHom$ using that $f^\ast_nP$ is compact), and again the description of Hom-spaces in limits of categories.
\end{proof}

The next result analyzes the behavior of completions in $\D_\nuc(X,\Lambda)$. More precisely, fix an ideal of definition $I\subseteq \Lambda$. In the following, we adhere to the conventions in \cref{sec:notat-conv} and in particular implicitly fix a finite set of generators of $I$; morever, we use the terminology ``complete'' instead of ``$I$-adically complete'' whenever convenient. Note that a priori completeness in $\D_\nuc(X,\Lambda)$ and $\D_\solid(X,\Lambda)_{\omega_1}$ may be different, as $\D_\nuc(X,\Lambda)$ is not stable under limits in $\D_\solid(X,\Lambda)_{\omega_1}$. However, the next result shows that in fact, the two notions of completeness coincide:

\begin{lemma}
  \label{sec:nucl-objects-overc-1-completions-of-nuclear-objects-are-nuclear}
  Let $X$ be an $\ell$-bounded spatial diamond and $M\in \D_\nuc(X,\Lambda)$.
  \begin{lemenum}
    \item The $I$-adic completion $\widehat{M}$ of $M$ in $\D_\solid(X,\Lambda)_{\omega_1}$ is automatically nuclear.
    \item $M$ is $I$-adically complete in $\D_\nuc(X,\Lambda)$ if and only if it is so in $\D_\solid(X,\Lambda)_{\omega_1}$.
  \end{lemenum}
\end{lemma}
\begin{proof}
  We start with (i).
  By \cref{sec:nucl-objects-overc-3-overconvergent-and-nuclear-objects} and \cref{sec:omeg-solid-sheav-2-pullback-and-pushforward-of-solid-sheaves}, specifically the commutation of pullbacks with limits, the claim is quasi-pro-\'etale local on $X$ and hence we may assume that $X$ is strictly totally disconnected. By \cref{sec:nucl-objects-overc-2-morphism-to-pi-0} and \cref{sec:nucl-objects-overc-1-nuclear-objects-are-pulled-back-from-pi-0} the claim reduces to the case that $X$ is a profinite set (realized as a qcqs perfectoid space over some $\Spa(C,\mathcal{O}_C)$ with $C$ algebraically closed and perfectoid). More precisely, we use that $\pi^\ast\colon \D_{\solid}(\pi_0(X),\Lambda)_{\omega_1}\to \D_\solid(X,\Lambda)_{\omega_1}$ commutes with limits to see that $\widehat{M}=\pi^\ast\pi_\ast(\widehat{M})$ is the pullback along $\pi^\ast$ of the completion of $\pi_\ast(M)$. Using that $\D_\nuc(X,\Lambda)=\D_\nuc(C(X,\Lambda))$ by \cref{sec:nucl-objects-overc-2-morphism-to-pi-0}, it suffices by \cref{sec:defin-d_hats-2-complete-and-discrete-mod-i-implies-nuclear} to see that $\widehat{M}$ is discrete modulo $I$. But $M/I\cong \widehat{M}/I$, and the claim follows.

  Now, (ii) is a consequence of (i): Namely, if $\widehat{M}$ denotes the $I$-adic completion of $M$ in $\D_\solid(X,\Lambda)_{\omega_1}$ then it is also the $I$-adic completion of $M$ in $\D_\nuc(X,\Lambda)$ because $\widehat{M}$ is nuclear by (i) and the nuclearization $\D_\solid(X,\Lambda)_{\omega_1}\to \D_\nuc(X,\Lambda)$, i.e. the right adjoint to the inclusion, preserves limits.
\end{proof}

We now move to a proof of the assertion that the category $\D_\nuc(X,\Lambda)$ is generated by complete objects for any $\ell$-bounded diamond $X$. For this assertion we will use some general properties of dualizable and rigid categories. Let $\Pr^L_\Sp$ be the $\infty$-category of presentable, stable $\infty$-categories (with colimit preserving functors). We recall that a stable $\infty$-category $\mathcal{C}\in \Pr^L_{\Sp}$ is dualizable if and only if $\mathcal{C}$ is a retract (in $\Pr^L_\Sp$) of a compactly generated stable $\infty$-category $\mathcal{D}\in \Pr^L_\Sp$ (see \cite{clausen2023_efimov_k-theory}), and this happens if and only if $\mathcal{C}$ is generated under colimits by objects of the form $X=\varinjlim_n(X_0\to X_1\to \ldots)$ with each morphism $X_n\to X_{n+1}$ compact. Here, a morphism $X\to Y$ in $\mathcal{C}$ is called compact if for any filtered diagram $Z_j,j\in J,$ in $\mathcal{C}$ with $\varinjlim_j Z_j\cong 0$, the morphism $\varinjlim_j \pi_0(\Hom_{\mathcal{C}}(Y,Z_j))\to \varinjlim_j \pi_0(\Hom_{\mathcal{C}}(X,Z_j))$ is zero. 
A presentably symmetric monoidal stable $\infty$-category $\mathcal{C}\in \mathrm{CAlg}(\Pr^L_\Sp)$ is rigid if and only if $\mathcal{C}$ is dualizable and the compact morphisms are exactly the trace-class morphisms (see \cite{clausen2023_efimov_k-theory}). We note the following abstract property guaranteeing rigidity. Essentially the argument is taken from \cite[Satz 3.17]{andreychev2023k}.\footnote{We thank Gregory Andreychev and Peter Scholze for related discussions.}

\begin{lemma}
  \label{sec:nuclear-sheaves-1-general-assertion-when-nuclear-rigid}
  Let $\mathcal{C}$ be a compactly generated closed symmetric monoidal $\infty$-category. Assume that $1\in \mathcal{C}$ is compact. Then:
  \begin{lemenum}
  \item Each trace-class morphism in $\mathcal{C}$ is compact.
  \item Assume that for each compact object $P\in \mathcal{C}$ the internal dual $P^\vee=\IHom(P,1)$ is nuclear. Then $\mathrm{Nuc}(\mathcal{C})$ is rigid. Moreover, the nuclearization functor $(-)^\nuc\colon \mathcal{C}\to \mathrm{Nuc}(\mathcal{C})$, right adjoint to the inclusion, commutes with colimits, and agrees with the trace functor $X\mapsto X^\tr:=\varinjlim_{1\to P\otimes X,\ P \textrm{ compact}} P^\vee$.
  \end{lemenum}
\end{lemma}
\begin{proof}
  Let $X\to Y$ be a trace-class morphism in $\mathcal{C}$ with trace-class witness $f\in \Hom_{\mathcal{C}}(1,X^\vee\otimes Y)$. Writing $Y=\varinjlim_i P_i$ with $P_i$ compact and using that $1\in \mathcal{C}$ is compact, there exists some $P_i\in \mathcal{C}$ and a factorization $X\to P_i\to Y$ with $X\to P_i$ is trace-class. In particular, the morphism $X\to Y$ is compact.

  For the second assertion, we first discuss the trace functor $X\mapsto X^{\tr}$. 
  As before \cite[Satz 3.16]{andreychev2023k}, one can construct a functor $\mathcal C \to \mathcal C$, $X\mapsto X^\tr$, which is uniquely determined by requiring that (naturally) $\Hom(P,X^\tr)\cong (P^\vee\otimes X)(\ast)$ for $P\in \mathcal{C}$ compact. Now, the argument of \cite[Satz 3.16]{andreychev2023k} shows that $X^\tr\cong \varinjlim_{1\to P\otimes X,\ P \textrm{ compact}} P^\vee$ for any $X\in \mathcal{C}$. Indeed, by definition of $X^\tr$ for any compact object $P\in \mathcal{C}$ we have a natural isomorphism $\Hom(P,X^\tr)\cong (P^\vee\otimes X)(\ast)$. Writing $P^\vee$ as a colimit of compact objects $Q$ with a morphism $Q\to P^\vee$, we can conclude that $\Hom(P,X^\tr)\cong \varinjlim_{Q \textrm{ compact}, Q\to P^\vee} (Q\otimes X)(\ast)$. Thus,
  \begin{align*}
    X^\tr \cong & \varinjlim\limits_{P \textrm{ compact}, P\to X^\tr} P \\
    \cong & \varinjlim\limits_{P,Q \textrm{ compact}, 1\ \to Q\otimes X,\ Q\to P^\vee} P \\
    \cong & \varinjlim\limits_{P,Q \textrm{ compact}, 1\ \to Q\otimes X,\ P\to Q^\vee} P.
  \end{align*}
  Taking first the colimit over $P$ this identifies with $\varinjlim_{Q \textrm{ compact}, 1\to Q\otimes X} Q^\vee$ as claimed.

  Having checked the formula for the trace functor, the remaining assertions follow. Namely, if $X\in \mathcal{C}$ is nuclear, then by definition the natural morphism $X^\tr\to X$ is an isomorphism. Moreover for any $X\in \mathcal{C}$ the object $X^\tr$ is nuclear by assumption and the above description of $X^\tr$ (we recall that nuclear objects are stable under all colimits). Clearly, the functor $X\mapsto X^\tr$ commutes with colimits. Thus the trace functor realizes $\mathrm{Nuc}(\mathcal{C})$ as a retract in $\Pr^L_{\Sp}$ of the compactly generated category $\mathcal{C}$. Hence, $\mathcal{C}$ is dualizable. We claim that $(-)^\nuc\cong (-)^\tr$. For this it suffices to see that for any basic nuclear object $X=\varinjlim_n X_n\in \mathcal{C}$ with each $X_n\to X_{n+1}$ a trace class morphism between compact objects and any $Y\in \mathcal{C}$, the natural map
  \[
    \Hom(X,Y^\tr)\to \Hom(X,Y)
  \]
  is an isomorphism. This is true as
  \[
    \Hom(X,Y^\tr)=\varprojlim\limits_{n\in \mathbb{N}} \Hom(X_n,Y^\tr)\cong \varprojlim\limits_{n\in \mathbb{N}} (X_n^\vee\otimes Y)(\ast)\cong \varprojlim\limits_{n\in \mathbb{N}} \Hom(X_n,Y)\cong \Hom(X,Y)
  \]
  using the definition of $Y^\tr$ and that each $X_n\to X_{n+1}$ is a trace-class morphism between compact objects.
  Knowing that $\mathrm{Nuc}(\mathrm{C})$ is dualizable and that $(-)^\nuc$ commutes with colimits (as it identifies with $(-)^\tr$) formally imply that $\mathrm{Nuc}(\mathcal{C})$ is rigid. Namely, by \cite{clausen2023_efimov_k-theory} (and (i)) it suffices to see that each compact morphism $f\colon X\to Y$ in $\mathrm{Nuc}(\mathcal{C})$ is trace class. As $(-)^\nuc$ commutes with colimits, $f$ is compact in $\mathcal{C}$, which implies that there exists a compact object $P\in \mathcal{C}$ and a factorization $X\to P\to Y$ of $f$. As $Y$ is nuclear, the morphism $P\to Y$ is automatically of trace class in $\mathcal{C}$. Let $1\to P^\vee\otimes Y$ be a witness for $P\to Y$, and let $\IHom^\nuc(X,1)$ the dual of $X$ in $\mathrm{Nuc}(\mathcal{C})$. As $P^\vee$ is nuclear the natural morphism $P^\vee\to \IHom(X,1)$ factors over $\IHom^\nuc(X,1)$. This implies that the morphism $f\colon X\to Y$ is trace class in $\mathrm{Nuc}(\mathcal{C})$. This finishes the proof.  
\end{proof}

Rigid categories have favourable properties for their module categories.

\begin{lemma}
  \label{sec:nuclear-sheaves-1-good-properties-of-modules-over-rigid-categories}
  Let $\mathcal{C}\in \CAlg(\Pr^L_\Sp)$ be a presentably stable symmetric monoidal $\infty$-category. Assume that $\mathcal{C}$ is rigid.
  \begin{lemenum}
  \item A $\mathcal{C}$-module $\mathcal{M}$ in $\Pr^L_\Sp$, i.e., an object of $\Pr^L_{\mathcal{C}}$, is dualizable in the symmetric monoidal category $(\Pr^L_{\mathcal{C}},\otimes_{\mathcal{C}})$ if and only it is dualizable in $(\Pr^L_{\Sp},\otimes)$.
  \item Let $F\colon \mathcal{M}\to \mathcal{N}$ be a morphism in $\Pr^L_{\mathcal{C}}$, and assume that the right adjoint $G\colon \mathcal{N}\to \mathcal{M}$ commutes with colimits. Then for $A\in \mathcal{C}$ and $N\in \mathcal{N}$ the natural morphism\footnote{This morphism is adjoint to $F(A\otimes G(N))\cong A\otimes F\circ G(N)\to A\otimes N$, with the first isomorphism using $\mathcal{C}$-linearity of $F$.}
    \[
      A\otimes G(N)\to G(A\otimes N),
    \]
    is an isomorphism. In other words, $G$ is $\mathcal{C}$-linear.
  \item Any morphism $F\colon \mathcal{C}\to \mathcal{R}$ in $\CAlg(\Pr^L_\Sp)$, where the symmetric monoidal unit $1_{\mathcal{R}}\in \mathcal{R}$ is compact, automatically has a colimit-preserving right adjoint $G$, i.e., is strongly continuous.
    \item If $A\in \CAlg(\mathcal{C})$, then $\mathrm{Mod}_A(\mathcal{C})$ is rigid.
  \end{lemenum}
\end{lemma}
\begin{proof}
  Statement (i) is proven in \cite[Proposition 9.4.4]{gaitsgory-rozenblyum-vol1} (for an equivalent characterization of rigid categories). Statement (ii) is \cite[Lemma 9.3.6]{gaitsgory-rozenblyum-vol1}. Let us prove (iii). Let $Z_i\in \mathcal{R},i\in I$, be a collection of objects. We have to see that for any $X\in \mathcal{C}$ the natural morphism
  \[
    \Hom_{\mathcal{C}}(X,\bigoplus\limits_{i\in I} G(Z_i))\to \Hom_{\mathcal{C}}(X,G(\bigoplus\limits_{i\in I}Z_i))
  \]
  is an isomorphism. We may assume that $X=\varinjlim_n X_n$ is basic nuclear with $X_n\to X_{n+1}$ of trace class. Then
  \[
    \Hom_{\mathcal{C}}(X,\bigoplus\limits_{i\in I} G(Z_i))\cong \varprojlim\limits_{n\in \mathbb{N}} (X_n^\vee\otimes \bigoplus\limits_{i\in I} G(Z_i))(\ast)\cong \varprojlim\limits_{n\in \N} \bigoplus\limits_{i\in I}(X^\vee_n\otimes G(Z_i))(\ast)\cong \varprojlim\limits_{n\in \N} \bigoplus\limits_{i\in I} \Hom(X_n,G(Z_i))
  \]
  using that $X_n\to X_{n+1}$ is of trace class and $1\in \mathcal{C}$ compact (as $\mathrm{Id}_1$ is of trace class, hence compact). On the other hand,
  \[
    \Hom(X,G(\bigoplus\limits_{i\in I}Z_i))\cong \Hom(F(X),\bigoplus\limits_{i\in I}Z_i)\cong \varprojlim\limits_{n\in \mathbb{N}} (F(X_n)^\vee\otimes \bigoplus\limits_{i\in I}Z_i) \cong \varprojlim\limits_{n\in \N} \bigoplus\limits_{i\in I} \Hom(F(X_n),Z_i)
  \]
  using that $F(X_n)\to F(X_{n+1})$ is of trace class and $1\in \mathcal{R}$ compact. This implies the assertion as $F$ is left adjoint to $G$.

  Assume that $A\in \CAlg(\mathcal{C})$. As the right adjoint of the functor $A\otimes (-)\colon \mathcal{C}\to \mathrm{Mod}_A(\mathcal{C})$, i.e., the forgetful functor, preserves colimits, the functor $A\otimes(-)$ preserves compact morphisms. As $\mathcal{C}$ is dualizable, this implies that $\mathrm{Mod}_A(\mathcal{C})$ is generated by objects, which are colimits along compact morphisms. This implies that $\mathcal{C}$ is dualizable. In fact, $\mathcal{C}$ is generated by colimits along trace-class maps, which by symmetric monoidality of the functor $A\otimes (-)$ implies the same for $\mathrm{Mod}_A(\mathcal{C})$, i.e., rigidity of $\mathrm{Mod}_A(\mathcal{C})$. 
\end{proof}

\begin{example}
  \label{sec:nuclear-sheaves-1-example-for-rigid-categories}
  \cref{sec:nuclear-sheaves-1-general-assertion-when-nuclear-rigid} and \cref{sec:defin-d_hats-2-complete-and-discrete-mod-i-implies-nuclear} imply that for any adic analytic ring $(A,A^+)$ in the sense of \cref{sec:adic-rings-1-definition-adic-analytic-ring} the category $\D_\nuc(A,A^+)$ is rigid. We note that $\D_\nuc(A,A^+)$ does not depend on $A^+$. 
\end{example}

We now apply these general assertions on rigid categories to nuclear sheaves on $\ell$-bounded spatial diamonds in order to show that they are generated under colimits by complete objects. In fact, we have:

\begin{proposition}
  \label{sec:nucl-objects-overc-1-d-nuc-generated-by-complete-objects}
  Let $X$ be an $\ell$-bounded spatial diamond. Then $\D_\nuc(X,\Lambda)$ is rigid and generated by complete objects which are right-bounded in $\D_\solid(X,\Lambda)_{\omega_1}$. Moreover, the nuclearization functor
  \begin{align*}
    (-)^\nuc\colon \D_\solid(X,\Lambda)_{\omega_1} \to \D_\nuc(X,\Lambda),
  \end{align*}
  right adjoint to the inclusion, preserves all small colimits.
\end{proposition}
\begin{proof}
  By \cref{sec:nuclear-sheaves-1-general-assertion-when-nuclear-rigid} it suffices to see that for each compact object $P\in \D_\solid(X,\Lambda)_{\omega_1}$ the object $P^\vee$ is right-bounded and nuclear. In fact, we may assume that $P=\Lambda_\solid[U]$ for some basic quasi-pro-\'etale $U\to X$. We note that $P^\vee$ is right-bounded, complete, overconvergent and discrete modulo some ideal of definition $I\subseteq \Lambda$. We claim that these properties imply nuclearity. By \cref{sec:nucl-objects-overc-1-completions-of-nuclear-objects-are-nuclear} and \cref{sec:nucl-objects-overc-3-overconvergent-and-nuclear-objects} we can reduce this claim to the case that $X$ is a strictly totally disconnected perfectoid space (because the pullback will again be right-bounded, complete, overconvergent and discrete mod $I$). By \cref{sec:nucl-objects-overc-1-nuclear-objects-are-pulled-back-from-pi-0} the assertion then reduces to \cref{sec:defin-d_hats-1-examples-for-nuclear-modules}.
\end{proof}

\begin{remark}
A convenient source for trace-class maps in $\D_\et(X,\Lambda)$ is the following. Let $X$ be an $\ell$-bounded spatial diamond and $f\colon U\to V, g\colon V\to X$ \'etale, quasi-compact and separated maps. Assume that $\Lambda$ is discrete (but potentially derived). If $f\colon U\to V$ extends to a map $h\colon \overline{U}^{/X}\to V$, then $\Lambda[U]\to \Lambda[V]$ is trace class in $\D_\solid(X,\Lambda)_{\omega_1}$.
\end{remark}

\subsection{\texorpdfstring{$\mathcal{C}$}{C}-valued nuclear sheaves}
\label{sec:mathc-valu-nucl}

In order to prove strong descent results for $\ri^+$-modules on perfectoid spaces, we need a slightly exotic version of nuclear $\Lambda$-modules, where we introduce an additional ``algebraic'' solid structure. In fact, the following definitions and results work more generally for certain $\D_\nuc(\Lambda)$-linear categories $\mathcal{C}$.

We continue with the notation from \cref{sec:nucl-objects-overc}, in particular $\Lambda$ is an adic and profinite $\Z_\ell$-algebra.

\begin{remark}
  \label{sec:mathc-valu-geom-remark-rigidity-of-d-nuc-r}
  The category $\Pr^L_{\D_\nuc(\Lambda)}$ of presentable $\D_\nuc(\Lambda)$-linear categories has some favourable properties due to the rigidity of $\D_\nuc(\Lambda)$ as a presentably, symmetric monoidal $\infty$-category (see \cref{sec:nuclear-sheaves-1-good-properties-of-modules-over-rigid-categories,sec:nuclear-sheaves-1-example-for-rigid-categories}).
\end{remark}

The $\D_\nuc(\Lambda)$-modules that we will interested in have particular properties. In the following we use the short hand notation ``complete'' instead of ``$I$-adically complete for some ideal of definition $I\subseteq \Lambda$''.

\begin{definition} \label{def:nicely-generated-nuc-linear-category}
We say that a $\D_\nuc(\Lambda)$-linear presentable category $\mathcal C$ is \emph{nicely generated} if there is a small family of objects $(P_i)_{i \in I}$ in $\mathcal C$ satisfying the following properties:
\begin{defenum}
	\item For every bounded complete $M \in \D_\nuc(\Lambda)$ and all $i \in I$ the object $M \tensor P_i \in \mathcal C$ is complete.

	\item For every $i \in I$ the functor $\Hom^{\D_\nuc(\Lambda)}(P_i, -)\colon \mathcal C \to \D_\nuc(\Lambda)$ preserves all small colimits.

	\item The family of functors $\Hom^{\D_\nuc(\Lambda)}(P_i, -)$ from (ii) is conservative.
\end{defenum}
\end{definition}

Here, $\Hom^{\D_\nuc(\Lambda)}(X,-)$ for $X\in \mathcal{C}$ is the functor right adjoint to the functor $(-)\otimes_{\Lambda}X\colon \D_\nuc(\Lambda)\to \mathcal{C}, M\mapsto M\otimes_{\Lambda} X$. In particular,
\begin{equation}
  \label{eq:1:absolute-vs-internal-hom}
  \Hom_{\D_\nuc(\Lambda)}(\Lambda, \Hom^{\D_\nuc(\Lambda)}(X,Y))\cong \Hom_{\mathcal{C}}(X,Y)  
\end{equation}
for $X,Y\in \mathcal{C}$.

\begin{lemma} \label{rslt:properties-of-nicely-generated-nuc-linear-categories}
  Let $\mathcal C$ be a nicely generated $\D_\nuc(\Lambda)$-linear presentable category.
  \begin{lemenum}
  \item The category $\mathcal C$ is compactly generated, and dualizable as an object in $\Pr^L_{\D_\nuc(\Lambda)}$.
    \item For every compact $P \in \mathcal C$ the functor $\Hom^{\D_\nuc(\Lambda)}(P, -)\colon \mathcal C \to \D_\nuc(\Lambda)$ is $\D_\nuc(\Lambda)$-linear.
  \end{lemenum}
  
\end{lemma}
\begin{proof}  
  As $\Lambda\in \D_\nuc(\Lambda)$ is compact, each $P_i$ from \cref{def:nicely-generated-nuc-linear-category} is compact, hence $\mathcal C$ is compactly generated. By \cite[Proposition D.7.2.3]{lurie-spectral-algebraic-geometry} the category $\mathcal{C}$ is therefore dualizable in $\Pr^L_\Sp$, and hence in $\Pr^L_{\D_\nuc(\Lambda)}$ by \cref{sec:mathc-valu-geom-remark-rigidity-of-d-nuc-r}. If $P\in \mathcal{C}$ is compact, then the $\D_\nuc(\Lambda)$-linear functor $\D_\nuc(\Lambda)\to \mathcal{C},\ A\mapsto A\otimes P$ has the colimit-preserving right adjoint $\Hom^{\D_\nuc(\Lambda)}(P,-)\colon \mathcal{C}\to \D_\nuc(\Lambda)$, which is therefore $\D_\nuc(\Lambda)$-linear by \cref{sec:mathc-valu-geom-remark-rigidity-of-d-nuc-r}.
\end{proof}

Recall that in \cref{sec:overc-solid-sheav-basic-nuclear-and-conuclear} we have introduced the $\D_\nuc(\Lambda)$-linear category $\D_\nuc(X,\Lambda)$ for an $\ell$-bounded spatial diamond $X$. Using this category, we now give the following definition.

\begin{definition} \label{def:C-valued-oc-nuc-sheaves}
Let $\mathcal C$ be a nicely generated $\D_\nuc(\Lambda)$-linear presentable category. Then for every $\ell$-bounded spatial diamond $X$ we denote
\begin{align*}
	\D_\nuc(X, \mathcal C):= \D_\nuc(X, \Lambda) \tensor_{\D_\nuc(\Lambda)} \mathcal C
\end{align*}
and call it the \emph{category of $\mathcal C$-valued nuclear (or nuclear overconvergent) sheaves on $X$}.
\end{definition}

\begin{proposition} \label{rslt:properties-of-C-valued-oc-nuc-sheaves}
Let $\mathcal C$ be a nicely generated $\D_\nuc(\Lambda)$-linear presentable category. Then:
\begin{propenum}
	\item The assignment $X \mapsto \D_\nuc(X, \mathcal C)$ defines a hypercomplete sheaf on the big quasi-pro-étale site of $\ell$-bounded spatial diamonds.
	\item For every $\ell$-bounded diamond $X$ the category $\D_\nuc(X, \mathcal C)$ is generated under colimits by $\pi$-adically complete objects.

	\item If $X$ is a strictly totally disconnected space then $\D_\nuc(X, \mathcal C)$ naturally identifies with the category of $C(\pi_0(X),\Lambda)$-modules in $\mathcal C$.
\end{propenum}
\end{proposition}
\begin{proof}
  Using \cite[Theorem 4.8.4.6]{lurie-higher-algebra} claim (iii) follows from the fact that
  \[
    \D_\nuc(X, \Lambda) \cong  \D_\nuc(C(\pi_0(X),\Lambda))\cong \mathrm{Mod}_{C(\pi_0(X),\Lambda)}(\D_\nuc(\Lambda))
  \]
  if $X$ is strictly totally disconnected (see \cref{sec:nucl-objects-overc-1-nuclear-objects-are-pulled-back-from-pi-0} and its proof).

  Let us prove claim (i) and take a quasi-pro-\'etale hypercover $X_\bullet\to X$ of $\ell$-bounded spatial diamonds. We need to see that the natural functor
  \[
    	F\colon \D_\nuc(X, \mathcal C) \to \varprojlim_{n\in\Delta} \D_\nuc(X_n, \mathcal C)
      \]
      is an equivalence. By \cref{sec:nucl-objects-overc-3-overconvergent-and-nuclear-objects} this holds if $\mathcal{C}=\D_\nuc(\Lambda)$, and then in $\Pr^L_{\D_\nuc(\Lambda)}$, i.e., all transition functors are $\D_\nuc(\Lambda)$-linear (as they are induced by pullback functors for $\D_\nuc(-,\Lambda)$). Therefore the claim follows by dualizability of $\mathcal{C}$ in $\Pr^L_{\D_\nuc(\Lambda)}$ as established in \cref{rslt:properties-of-nicely-generated-nuc-linear-categories}. Indeed, dualizability implies that $\mathcal{C}\otimes_{\D_\nuc(\Lambda)}-$ commutes with limits in $\Pr^L_{\D_\nuc(\Lambda)}$.

      It remains to prove (ii), so let $X$ be a given $\ell$-bounded diamond. Then $\D_\nuc(X, \mathcal C)$ is generated under colimits by objects $\mathcal M \boxtimes P$, where $\mathcal M \in \D_\nuc(X, \Lambda)$ and $P \in \mathcal C$ is compact (and hence also complete by our assumption that $\mathcal{C}$ is nicely generated). By \cref{sec:nucl-objects-overc-1-d-nuc-generated-by-complete-objects} we may assume that $M$ is complete, and additionally bounded above.
      
      It is therefore enough to show that each of these $\mathcal M \boxtimes P$ is complete. By (i) we can check this after pullback to any quasi-pro-étale hypercover (using that limits of complete objects are again complete), so we can assume that $X$ is strictly totally disconnected. Then the claim follows from (iii) and the fact that complete objects in $\mathcal C$ are stable under the action of complete objects in $\D_\nuc(\Lambda)$ by assumption. More precisely, completeness of $M\boxtimes P$ can be checked after applying the forgetful functor to $\mathcal{C}$, and there $M\boxtimes P$ is given by $|M|\otimes P$ with $|M|\in \D_\nuc(\Lambda)$ the underlying object of $M\in \D_\nuc(X,\Lambda)\cong \D_\nuc(C(\pi_0(X),\Lambda))$ and $\otimes$ the given $\D_\nuc(\Lambda)$-action on $\mathcal{C}$.
\end{proof}

\section{\texorpdfstring{$\D_{\hat\solid}^a(\ri_X^+)$}{D\textasciicircum a\_sld(O\textasciicircum+\_X)} for perfectoid spaces}
\label{sec:d_hats-perf-spac}

In this section we will prove the first part of \cref{sec:introduction-1-main-theorem-introduction}, i.e., $v$-descent of almost $\ri^+$-sheaves on perfectoid spaces for our modified version of $+$-modules from \cref{def:modified-modules-over-adic-ring}.

\subsection{Definition of \texorpdfstring{$\D_{\hat\solid}^a(A^+)$}{D\textasciicircum a\_sld(A\textasciicircum+)}}
\label{sec:defin-d_hats-a+}

As in \cite[Definition 3.1.1]{mann-mod-p-6-functors} we denote by $\AffPerfd$ the category of affinoid perfectoid spaces over $\Spf\Z_p$. We note that each object $(A,A^+)\in \AffPerfd$ is classical and static and hence we will often not distinguish between $(A,A^+)$ and the classical Tate-Huber pair $(A(\ast),A^+(\ast))$.

\begin{lemma}
  \label{sec:defin-d_hats-1-idempotent-algebra-defining-almost-mathematics}
  Let $(A,A^+)\in \AffPerfd$, and let $A^{\circ\circ}\subseteq A^+$ be the ideal of topologically nilpotent elements. Then $A^+/A^{\circ\circ}$ is an idempotent algebra in the classical derived category $\D(A^+)=\D(A^+(\ast))$ of $A^+$-modules.
\end{lemma}
\begin{proof}
  This follows from the fact that $A^{\circ\circ}$ is a filtered colimit of principal ideals generated by non-zero divisors, and that $A^{\circ\circ}\cdot A^{\circ\circ}=A^{\circ\circ}$.
\end{proof}

\begin{definition}
  \label{sec:defin-d_hats-2-definition-of-almost-plus-category}
  Let $(A,A^+)\in \AffPerfd$. Then we define $\D_{\hat\solid}^a(A^+):=\D_{\hat\solid}((A^+,A^+))\otimes_{\D(A^+)} \D^a(A^+)$, where $\D(A^+)\to \D^a(A^+)$ is the open immersion defined by the idempotent algebra $A^+/A^{\circ\circ}$ from \cref{sec:defin-d_hats-1-idempotent-algebra-defining-almost-mathematics}.
\end{definition}

\begin{remark}
By \cref{sec:glob-stably-unif-2-closed-open-immersions-are-stable-under-base-change-for-lurie-tensor-product} we have $\D_{\hat\solid}^a(A^+)=\D_{\hat\solid}((A^+,A^+))/\mathrm{Mod}_{A^+/A^{\circ\circ}}(\D_{\hat\solid}((A^+,A^+)))$.
\end{remark}

\begin{remark}
  \label{sec:defin-d_hats-1-remark-on-role-of-a-+}
  We note that $A^+\cong A^\circ$ in $\D^a_{\hat\solid}(A^+)$ as $A^{\circ\circ}\subseteq A^+$. Thus, $\D_{\hat\solid}((A^\circ, A^+)), \D_{\hat\solid}((A^+,A^+)$ become isomorphic after $(-)\otimes_{\D(A)}\D^a(A)$. However, in general $\D^a_{\hat\solid}(A^+)$ and $\D^a_{\hat\solid}(A^\circ)$ are different. 
\end{remark}

\begin{remark}
  \label{sec:defin-d_hats-2-comparison-with-almost-definition-of-lucas-thesis}
  Assume that $\pi\in A\cap A^+$ is a pseudo-uniformizer. Then $\mathrm{Mod}_{A^+/\pi}(\D^a_{\hat\solid}(A^+))\cong \D^a_{\solid}(A^+/\pi)$ is naturally equivalent to the category defined in \cite[Definition 3.1.2]{mann-mod-p-6-functors} because $\mathrm{Mod}_{A^+/\pi}(\D^a_{\hat\solid}(A^+))\cong \mathrm{Mod}_{A^+/\pi}(\D_{\solid}(A^+))\otimes_{\D(A^+)}\D^a(A^+)\cong \D^a_\solid(A^+/\pi)$ using \cref{rslt:A-mod-I-modified-modules-equals-usual-modules} in the first isomorphism.
\end{remark}

The following special property is the basic reason why (almost) $v$-descent results for $+$-modules on perfectoid spaces are possible while they fail drastically for non-perfectoid spaces.

\begin{lemma} \label{rslt:base-change-for-affinoid-perfectoid}
Let
\begin{center}\begin{tikzcd}
  Y' \arrow[r] \arrow[d] & Y \arrow[d]\\
  X' \arrow[r] & X
\end{tikzcd}\end{center}
be a cartesian diagram in $\AffPerfd$. Write $X = \Spa(A, A^+)$, $X' = \Spa(A', A'^+)$, $Y = \Spa(B, B^+)$ and $Y' = \Spa(B', B'^+)$. Then the natural morphism of analytic rings
\begin{align*}
  (A'^+)_\solid^a \tensor_{(A^+)_\solid^a} (B^+)_\solid^a \isoto (B'^+)_\solid^a.
\end{align*}
is an isomorphism.
\end{lemma}
Here, we view $(A^+)^a_\solid$ etc.\ as analytic rings over the almost setup $(A^+,A^{\circ\circ})$ following \cite[Section 2.3]{mann-mod-p-6-functors}. More concretely, the assertion means that the natural functor $\D((B')^+_\solid)\to \D((A'^+)_\solid\otimes_{A^+_\solid} B^+_\solid)$ becomes an equivalence after $(-)\otimes_{\D(A^+)}\D^a(A^+)$.
\begin{proof}
Since $B'^+$ is the (completed) integral closure of the image of $\pi_0(A'^+ \tensor_{A^+} B^+)$ in $\pi_0(A' \tensor_A B)$, by \cref{rslt:pushout-of-adic-analytic-rings} it is enough to show that the map
\begin{align*}
  A'^{+a} \hat\tensor_{A^{+a}} B^{+a} \isoto B'^{+a}
\end{align*}
is an isomorphism of almost rings. But both sides of this claimed isomorphism are $\pi$-adically complete for any pseudouniformizer $\pi$ of $A$, hence the claim can be checked modulo $\pi$, where it follows from \cite[Lemma 3.1.6]{mann-mod-p-6-functors}.
\end{proof}

As a corollary from the proof we see that the derived (completed) tensor product $A'^{+a}\otimes_{A^{+a}} B^{+a}$ is in fact static, i.e., concentrated in degre $0$. This statement also holds before passing to the almost category by reduction to the case of perfect rings.

\subsection{Descent on totally disconnected spaces}
\label{sec:desc-totally-disc}

In this section, we prove the first part of \cref{sec:introduction-1-main-theorem-introduction}, i.e., the existence of the $v$-sheaf $X\mapsto \D^a_{\hat\solid}(\mathcal{O}_X)$ on perfectoid spaces. More precisely, in this section we prove the following theorem.

\begin{theorem} \label{rslt:v-hyperdescent-for-O+-modules}
There is a unique hypercomplete v-sheaf of categories
\begin{align*}
  (\Perfd)^\opp \to \catcat, \qquad X \mapsto \D^a_{\hat\solid}(\ri^+_X)
\end{align*}
such that for every $X = \Spa(A, A^+)$ which admits a quasi-pro-étale map to a totally disconnected space we have
\begin{align*}
  \D^a_{\hat\solid}(\ri^+_X) = \D^a_{\hat\solid}(A^+).
\end{align*}
\end{theorem}

We recall that a qcqs perfectoid space $X$ is totally disconnected if each connected component of $X$ has the form $\Spa(K,K^+)$ for some perfectoid field $K$ with open and bounded valuation subring $K^+\subseteq K$ (\cite[Lemma 7.3]{etale-cohomology-of-diamonds}), in which case $X=\Spa(R,R^+)$ is necessarily affinoid, and for any pseudo-uniformizer $\pi\in R$ we have $(|X|,\mathcal{O}^+_X/\pi)\cong \Spec(R^+/\pi)$ as locally ringed spaces (\cite[Lemma 3.6.1]{mann-mod-p-6-functors}).

Totally disconnected perfectoid spaces enjoy the following surprising flatness property (\cite[Lemma 3.1.7]{mann-mod-p-6-functors}).

\begin{lemma}
  \label{sec:desc-totally-disc-1-finite-type-over-totally-disconnected-implies-flat-finitely-presented-mod-pi}
  Let $X=\Spa(A,A^+)$ be a totally disconnected perfectoid space, and $\pi\in A$ a uniformizer. Let $A^\prime$ be a $\pi$-torsion free $A^+$-algebra of finite type. Then for every connected component $x=\Spec((A^+/\pi)_x)$ of $\Spec(A^+/\pi)$ the map $(A^+/\pi)_x\to A^\prime\otimes_{A^+} (A^+/\pi)_x$ is flat and finitely presented.
\end{lemma}
\begin{proof}
  This is \cite[Lemma 3.1.7]{mann-mod-p-6-functors}.
\end{proof}

This result strengthens \cite[Proposition 7.23]{etale-cohomology-of-diamonds} and is critical for \cref{sec:introduction-1-mod-p-theorem-of-lucas} as it implies descent for $\mathcal{O}^+/\pi$-modules for a $v$-cover $f\colon Y\to X$ with $X$ totally disconnected (\cite[Lemma 3.1.8]{mann-mod-p-6-functors}). The proof of loc.\ cit.\ uses critically that the (pre)sheaf $\Spa(R,R^+)\mapsto R^+/\pi$ on affinoid perfectoid spaces sends cofiltered inverse limits to filtered colimits, roughly to reduce to a single connected component of $X$ where \cref{sec:desc-totally-disc-1-finite-type-over-totally-disconnected-implies-flat-finitely-presented-mod-pi} applies. This argument does not work for the sheaf $\mathcal{O}^+$. Instead, we roughly spread out the flatness mod $\pi$ on a connected component (provided by \cref{sec:desc-totally-disc-1-finite-type-over-totally-disconnected-implies-flat-finitely-presented-mod-pi}) to some neighborhood, at least up to replacing $Y$ by some disjoint union of some finite open cover. For the precise assertion below, we denote by $Y^\wl \in \AffPerfd$ the w-localization of any $Y \in \AffPerfd$, as defined in \cite[Proposition 7.12]{etale-cohomology-of-diamonds}.

\begin{proposition} \label{rslt:v-cover-of-td-spaces-is-weakly-descendable}
Let $Y=\Spa(B,B^+) \to X=\Spa(A,A^+)$ be a surjective map of totally disconnected spaces in $\AffPerfd$. Then the map $Y^\wl \to X$ is $+$-weakly adically descendable of index $\le 4$, i.e., if $Y^\wl=\Spa(B_\wl,B_\wl^+)$ the map $(A^+)_\solid\to (B^+_\wl)_\solid$ is weakly adically descendable of index $\leq 4$ in the sense of \cref{sec:complete-descent-1-definition-weakly-adically-descendable}. 
\end{proposition}
\begin{proof}
  We need to show that the map $A^+ \to B_\wl^+$ of adic rings is an $\omega_1$-filtered colimit of adically descendable maps of index $\le 4$. Let $I$ be the category of factorizations $Y^\wl \to Y_i \to Y$, where $Y_i$ is a finite disjoint union of qcqs open subsets of $Y$. Then by \cite[Lemma 7.13]{etale-cohomology-of-diamonds} $I$ is cofiltered and $Y^\wl = \varprojlim_{i\in I} Y_i$. We denote $Y_i = \Spa(B_i, B^+_i)$, so that $B_\wl^+$ is the completed filtered colimit of the $B^+_i$.

In the following discussion, all rings and modules will be discrete and static unless stated otherwise. For some fixed pseudouniformizer $\pi \in A$ we let $J_\pi$ be the following category: An object of $J_\pi$ is a factorization $A^+ \to A_j \to B_\wl^+$ of $A^+$-algebras such that the map $A^+ \to A_j$ is a finitely presented map of classical rings which is flat modulo $\pi$ (here we take the \emph{underived} quotient $A_j/\pi$!) and such that the map $A_j \to B_\wl^+$ factors over some $B^+_i$. A morphism in $J_\pi$ is a morphism $\alpha\colon A_j \to A_{j'}$ over $A^+$ such that there is some $i \in I$ with the property that both $A_j \to B^+_\wl$ and $A_{j'} \to B^+_\wl$ factor over $B^+_i$ and these factorizations are compatible with $\alpha$. We now claim that $J_\pi$ has the following crucial property:
\begin{itemize}
  \item[(a)] Let $A_j \in J_\pi$. Pick some $i$ such that $A_j \to B^+_\wl$ factors over $B^+_i$ and some $a \in A_j$ which is sent to $0$ via $A_j \to B^+_i$. Then there is some morphism $\alpha\colon A_j \to A_{j'}$ in $J_\pi$ such that $\alpha(a) = 0$.
\end{itemize}
We now prove this claim, so fix $A_j$, $i$ and $a$ as in the claim. Let $A'_j \subset B^+_i$ be the image of $A_j$ and fix some connected component $x \in \pi_0(X)$. Equivalently $x$ is a connected component of $\Spec A^+/\pi$, i.e. of the form $x = \Spec (A^+/\pi)_x$, as $X$ is totally disconnected. For any $A^+$-algebra $A'$ we denote $(A'/\pi)_x := A' \tensor_{A^+} (A^+/\pi)_x$. Then by \cref{sec:desc-totally-disc-1-finite-type-over-totally-disconnected-implies-flat-finitely-presented-mod-pi} the map $(A^+/\pi)_x \to (A'_j/\pi)_x$ is flat and finitely presented for $A^\prime_j$ as above and in particular the map $\beta_{x,\pi}\colon (A_j/\pi)_x \to (A'_j/\pi)_x$ is finitely presented (\cite[Tag 00F4]{stacks-project}). Pick some generators $\overline a_{1,x}, \dots, \overline a_{n,x}$ of $\ker(\beta_{x,\pi})$. They automatically extend over some (pullback of some) clopen neighbourhood of $x$ in $\Spec A^+/\pi$ and hence to all of $A_j/\pi$ by extending them by $0$ on the (pullback of the) clopen complement of that neighbourhood. Moreover, we can guarantee that these extensions $\overline a_1, \dots, \overline a_n \in A_j/\pi$ lie in the kernel of the map $A_j/\pi \to A'_j/\pi$ by shrinking the clopen neighbourhood of $x$ if necessary. Since the map $\beta\colon A_j \to A'_j$ is surjective, we can find lifts $a_1, \dots, a_n \in A_j$ of $\overline a_1, \dots, \overline a_n$ which lie in the kernel of $\beta$. We can assume that $a$ is one of the $a_k$'s as by assumption $a$ maps to $0$ in $A_j^\prime\subseteq B_i^+$. Now
denote
\begin{align*}
  A'_{j,x} := A_j/(a_1, \dots, a_n).
\end{align*}
Then $A'_{j,x}$ is a finitely presented $A^+$-algebra with a map to $A'_j$ such that $(A'_{j,x}/\pi)_x = (A'_j/\pi)_x$ and such that $a = 0$ in $A'_{j,x}$. By \cite[Tag 00RC]{stacks-project} the locus $U \subset \Spec A'_{j,x}/\pi$ which is flat over $A^+/\pi$ is an open subset and by construction we have $\Spec (A'_{j,x}/\pi)_x \subset U$. We can thus find elements $\overline f_1, \dots \overline f_m \in A'_{j,x}/\pi$ such that each localization $(A'_{j,x}/\pi)_{\overline f_k}$ is flat over $A^+/\pi$ and such that the associated schemes cover $\Spec (A'_{j,x}/\pi)_x$. By the proof of \cite[Lemma 3.6.1]{mann-mod-p-6-functors} the elements $\overline f_1, \dots, \overline f_m$, more precisely their images in $B_i^+/\pi$, induce qcqs open, even rational, subsets $V_{1,x}, \dots V_{m,x} \subset Y_i$ such that $\overline f_k$ is a unit in $\mathcal{O}^+_{Y_i}(V_{k,x})$ for $k=1, \ldots, m$ (here we use that $Y$ and hence $Y_i$ is totally disconnected in order to apply \cite[Lemma 3.6.1]{mann-mod-p-6-functors}). Pick any lifts $f_1, \dots, f_m \in A'_{j,x}$ and denote
\begin{align*}
  A_{j,x} := \prod_{k=1}^m (A'_{j,x})_{f_k}, \qquad Y_{i,x} := \bigdunion_{k=1}^m V_{k,x}.
\end{align*}
Writing $Y_{i,x} = \Spa(B_{i,x}, B^+_{i,x})$ we get a map $A_{j,x} \to B^+_{i,x}$ of $A^+$-algebras which is compatible with the natural maps $A_j \to A_{j,x}$ and $B^+_i\to B^+_{i,x}$. Moreover, by construction the subsets $Y_{k,x}$ cover the fiber of $Y_i$ over $x$ (because the distinguished opens $D(\overline f_1), \ldots, D(\overline f_m)$ cover $\Spec(A'_{j,x}/\pi)_x)$, and hence their pullbacks cover $\Spec(B^+_i/\pi)$). Since $Y_i$ is qcqs, we can find finitely many connected components $x_1, \dots, x_{\ell}$ of $X$ such that the map $Y_{i'} := Y_{i,x_1} \dunion \dots \dunion Y_{i,x_{\ell}} \surjto Y_i$ is a cover, so that $Y_{i'} \in I$. Let
\begin{align*}
  A_{j'} := \prod_{k=1}^\ell A_{j,x_k}.
\end{align*}
Then $A_{j'} \in J_\pi$ and the natural map $\alpha\colon A_j \to A_{j'}$ satisfies the claim (a). This proves (a).

We note that we did not use that $A_j/\pi$ is flat over $A^+/\pi$ in the proof, only flatness of $A^\prime_j/\pi$, which was supplied by \cref{sec:desc-totally-disc-1-finite-type-over-totally-disconnected-implies-flat-finitely-presented-mod-pi} as the $A^\prime_j$ was a finitely generated, $\pi$-torsion free $A^+$-algebra. Replacing in the argument $\pi$ by $\pi^n, n\geq 1,$ we see that we can even guarantee that $A_{j^\prime}/\pi^n$ is flat over $A^+/\pi^n$. 

From now on, we work with derived condensed rings again. With (a) at hand, we can now prove the following claim:
\begin{itemize}
  \item[(b)] The category $J_\pi$ is filtered and $B^+_\wl$ is the (derived) $\pi$-adically completed colimit of the system $(\hat A_j)_{j \in J_\pi}$, where $\hat A_j$ denotes the (derived) $\pi$-adic completion of $A_j$.
\end{itemize}
Let us first show that $J_\pi$ is filtered. Clearly, $J_\pi$ is non-empty as $A^+\in J_\pi$. Given $A_j, A_{j'} \in J_\pi$, pick any $i \in I$ such that both $A_j \to B^+_\wl$ and $A_{j'} \to B^+_\wl$ factor over $B^+_i$. Then the static, non-condensed tensor product $A_{j''} := A_j \tensor_{A^+} A_{j'}$ lies in $J_\pi$ and the map $A_{j''} \to B^+_\wl$ factors over $B^+_i$. In particular the natural maps $A_j \to A_{j''}$ and $A_{j'} \to A_{j''}$ lie in $J_\pi$. Now suppose we have two maps $\alpha_1, \alpha_2\colon A_j \to A_{j'}$ in $J_\pi$. By repeatedly applying (a) to the images of polynomial generators of $A_j$ under $\alpha_1 - \alpha_2$ we can construct some $A_{j''} \in J_\pi$ with a map $A_{j'} \to A_{j''}$ such that $\alpha_1$ and $\alpha_2$ become equal on $A_{j''}$. This proves that $J_\pi$ is indeed filtered.

Now let $B'$ be the (derived) $\pi$-adic completion of $\varinjlim_{j\in J_\pi} \hat A_j$. There is a natural map $B' \to B^+_\wl$ of $\pi$-adic rings which we need to show to be an isomorphism. This can be checked modulo $\pi$, so we need so show that the natural map
\begin{align*}
  \varinjlim_j \hat A_j/\pi = \varinjlim_j A_j/\pi \to B^+_\wl/\pi = \varinjlim_i B^+_i/\pi
\end{align*}
is an isomorphism of discrete rings (here the tensor product $A_j/\pi$ is a priori derived!). For this it is enough to show the stronger claim that the natural map $\varinjlim_j A_j \to \varinjlim_i B^+_i$ of classical, non-condensed rings is an isomorphism of static discrete rings. This map is clearly surjective, because every $b \in B^+_i$ is hit by some $A_j \to B^+_i$, e.g., for $A_j = A^+[x]$ and $x\mapsto b$. The map is injective by claim (a). This proves the desired isomorphism and thus also claim (b).

We have now shown that the map $A^+ \to B^+_\wl$ is the $\pi$-completed filtered colimit of $\pi$-complete finitely presented $A^+$-algebras whose underived reduction modulo $\pi$ (i.e. $\pi_0(A_j/\pi)$) is flat over $A^+/\pi$. Since we only talk about underived reductions modulo $\pi$, we cannot immediately conclude the desired descendability statement: Since we have no control over the $\pi$-torsion of $A_j \in J_\pi$, we cannot deduce from the flatness of $\pi_0(A_j/\pi)$ that also $\pi_0(A_j/\pi^n)$ is flat. We therefore need the following additional argument.

From now on we denote $J_n := J_{\pi^n}$ for some fixed pseudouniformizer $\pi \in A$. Then $J_{n+1} \subset J_n$ is a full subcategory for every $n \ge 1$ and as we noted after the proof of (a) this subcategory is even cofinal. Now let $J$ be the following category: The objects of $J$ are the functors $j\colon \N \to J_1$ such that $j(n) \in J_n$ for all $n \ge 1$. For $j, j' \in J$ we define
\begin{align*}
  \Hom_J(j, j') = \varinjlim_{n\in\N} \Hom(\restrict j{\N_{\ge n}}, \restrict{j'}{\N_{\ge n}}).
\end{align*}
To every $j \in J$ we associate the $A^+$-algebra $A_j := \varinjlim_n A_{j(n)}$. It follows from (a) that $J$ is $\omega_1$-filtered\footnote{Given any countable diagram $j_k, k\in K,$ in $J$, we can write the category $K$ as a filtered union of finite subgraphs $K_n, n\in \mathbb{N},$ and set $j^\prime(n)\in J_n$ as an object with a map from each $j_k(l), k\in K_n$ and $l\leq n$ such that these morphisms to $j^\prime(n)$ equalize all morphisms $j_k(l)\to j_{k^\prime}(l)$ induced by a morphism in $K_n$. Inductively, we can even construct compatible morphisms $j^\prime(n)\to j^\prime(n+1)$, and set $j_k\to j^\prime$ for $k\in K_n$ as the morphism, which for $m\geq n$ is the chosen morphism $j_k(m)\to j^\prime(m)$. } and it follows as in (b) that $B^+_\wl$ is the $\pi$-completed colimit of $(\hat A_j)_{j\in J}$. Thus, to finish the proof it is now enough to show the following claim:
\begin{itemize}
  \item[(c)] For every $j \in J$ the map $(A^+)_\solid \to (\hat A_j)_\solid$ is adically descendable of index $\le 4$.
\end{itemize}
To prove (c), fix $j \in J$ and $n \ge 1$. We need to show that the map $(A^+/\pi^n)_\solid \to (A_j/\pi^n)_\solid$ is descendable of index $\le 4$. This map is the filtered colimit of the maps $(A^+/\pi^n)_\solid \to (A_{j(k)}/\pi^n)_\solid$, so by \cite[Proposition 2.7.2]{mann-mod-p-6-functors} it is enough to show that for $k \gg 0$ the map $(A^+/\pi^n)_\solid \to (A_{j(k)}/\pi^n)_\solid$ is descendable of index $\le 2$. We claim that this works for $k \ge n$. Indeed, we have the factorization
\begin{align*}
  \alpha\colon (A^+/\pi^n)_\solid \to (A_{j(k)}/\pi^n)_\solid \to \pi_0(A_{j(k)}/\pi^n)_\solid,
\end{align*}
and by definition of $J_n$ the composed map $\alpha$ is flat and finitely presented. The map $\alpha$ is even surjective on spectra, and thus an fppf cover, because its composition with $\pi_0(A_{j(k)}/\pi^n)\to B^+_\wl/\pi^n$ is $A^+/\pi^n\to B^+_\wl/\pi^n$, which is surjective on spectra (as $Y^\wl\to X$ is a $v$-cover of totally disconnected spaces). Thus by \cite[Lemma 2.10.6]{mann-mod-p-6-functors} the map $\alpha$ is descendable of index $\le 2$, so by \cite[Lemma 2.6.10.(ii)]{mann-mod-p-6-functors} the same is true for the first map $(A^+/\pi^n)_\solid\to (A_{j(k)}/\pi^n)_\solid$, as desired.
\end{proof}

\begin{proof}[Proof of \cref{rslt:v-hyperdescent-for-O+-modules}]
  It is sufficient to construct $X\mapsto \D^a_{\hat\solid}(\mathcal{O}^+_X)$ for $X\in \AffPerfd$ as it will formally extend via analytic descent.
  We define the presheaf $X=\Spa(A,A^+)\mapsto F(X):=\D_{\hat\solid}^a(A^+)$ on $\Perfd^\aff$.
We first observe that for every v-cover $f\colon Y=\Spa(B,B^+) \to X=\Spa(A,A^+)$ of totally disconnected spaces, the presheaf $F$ satisfies descent along $f$. Indeed, by \cref{rslt:v-cover-of-td-spaces-is-weakly-descendable} and \cref{rslt:base-change-for-affinoid-perfectoid} the presheaf $F$ satisfies descent along the map $Y^\wl \to X$, even after any base change $X^\prime\to X$ with $X^\prime$ any affinoid, not necessarily totally disconnected, perfectoid space.\footnote{More precisely, we as well use that descent of $\D_{\hat\solid}(-)$ for $(A^+)_\solid\to (B^+)_\solid$ implies descent of $\D_{\hat\solid}^a(-)$ by embedding the almost category (via $j_!$) into $\D_{\hat\solid}(-)$ and using \cref{sec:glob-stably-unif-3-tensoring-preserves-fiber-sequences} to see that this embedding is compatible with pullback. } Hence, by \cite[Lemma 3.1.2.(3)]{enhanced-six-operations} the same is true for $f\colon Y\to X$. Using \cref{rslt:base-change-for-affinoid-perfectoid} we can argue as in \cite[Proposition 3.1.9]{mann-mod-p-6-functors} to show that there is a (necessarily unique) quasi-pro-étale sheaf $X \mapsto \D^a_{\hat\solid}(\ri^+_X)$ on $\AffPerfd$ which has the desired form on those spaces which are quasi-pro-étale over some totally disconnected space. We recall the argument for the convenience of the reader: first, the presheaf $F$ defines a quasi-pro-\'etale sheaf on the pro-\'etale site\footnote{We note that a qcqs perfectoid space $Y$ with a quasi-pro-\'etale map $Y\to X$ to a strictly totally disconnected perfectoid space $X$ is itself again strictly totally disconnected. This need not be true if $X=\Spa(K,K^+)$ is connected and totally disconnected as the valuation ring $K^+$ need not be henselian.} of \textit{strictly} totally disconnected perfectoid spaces, i.e., those totally disconnected spaces whose connected components are of the form $\Spa(C,C^+)$ with $C$ algebraically closed and $C^+\subseteq C$ an open and bounded valuation subring. Formally, this quasi-pro-\'etale sheaf extends to the quasi-pro-\'etale sheaf $X\mapsto \D^a_{\hat\solid}(\ri^+_X)$ on $\Perfd^\aff$. Assume that $X=\Spa(A,A^+)$ is affinoid perfectoid with a quasi-pro-\'etale map $X\to Z$ to a totally disconnected space $Z$. If $Y\to Z$ is a quasi-pro-\'etale cover with $Y$ strictly totally disconnected, then $Y\to Z$ satisfies universal descent for $F$ as we have seen above. In particular, $F$ satisfies descent for the cover $Y^\prime:=Y\cprod_Z X\to X$. As the terms of the \v{C}ech nerve for $Y^\prime\to X$ are strictly totally disconnected, we can therefore conclude that $F(X)=\D^a_{\hat\solid}(\mathcal{O}^+_X)$ as claimed.

It remains to show that the sheaf $X \mapsto \D^a_{\hat\solid}(\ri^+_X)$ satisfies v-hyperdescent. This can be done in the same way as in the proof of \cite[Theorem 3.1.27]{mann-mod-p-6-functors}: One first shows the analog of \cite[Lemma 3.1.23]{mann-mod-p-6-functors} for adic coefficients, where the two crucial inputs are \cite[Lemma 3.1.22]{mann-mod-p-6-functors} (which can be replaced by \cref{rslt:perfectoid-etale-descent-of-integral-sheaves}) and the finite Tor dimension (which can be lifted to the adic level by \cref{rslt:tor-dim-of-adic-maps}). Then the adic version of \cite[Corollary 3.1.24]{mann-mod-p-6-functors} is a formal corollary and consequently we deduce the adic version of \cite[Proposition 3.1.25]{mann-mod-p-6-functors}; here the only non-formal input is the fact that $+$-modules, i.e., $\D^+_{\hat\solid}(A^+)$ for $X=\Spa(A,A^+)$, descend along v-covers of totally disconnected spaces, as shown above. In other words, we have now shown that $X \mapsto \D^a_{\hat\solid}(\ri^+_X)$ is a v-sheaf. To get v-hyperdescent it remains to prove the adic version of \cite[Proposition 3.1.26]{mann-mod-p-6-functors}. But here the only non-formal input is the finite Tor dimension, which can again be deduced from the mod-$\pi$ version via \cref{rslt:tor-dim-of-adic-maps}.
\end{proof}

The following result was used in the proof of \cref{rslt:v-hyperdescent-for-O+-modules} above. It is an interesting statement on its own, so we extract it here:

\begin{lemma} \label{rslt:perfectoid-etale-descent-of-integral-sheaves}
Let $f\colon Y \to X$ be an étale cover in $\AffPerfd$. Then $+$-modules descend along $j$, i.e., the functor $X=\Spa(A,A^+)\mapsto \D^a_{\hat\solid}(A^+)$ is an \'etale sheaf on $\AffPerfd$.
\end{lemma}
\begin{proof}
  After inverting a pseudo-uniformizer $\pi\in A$, i.e., for the functor $X=\Spa(A,A^+)\mapsto \mathrm{Mod}_A(\D^a_{\hat\solid}(A^+))=\mathrm{Mod}_A(\D_{\hat\solid}(A^+))$ this follows from \cref{sec:glob-stably-unif-1-analytic-and-etale-descent-for-hat-solid-version}. Using the fiber sequence $A^+\to A\to A/A^+=\varinjlim\limits_{n} A^+/\pi^n$ reduces therefore to the descent for the functor $X=\Spa(A,A^+)\mapsto \mathrm{Mod}_{A^+/\pi^n}(\D^a_{\hat\solid}(A^+))=\mathrm{Mod}_{A^+/\pi^n}(\D^a_{\solid}(A^+))$. Moreover, using induction and the fiber sequence $A^+/\pi^{n-1}\to A^+/\pi^n\to A^+/\pi$ reduces to the case $n=1$. As $\pi$ was arbitrary, we may even assume that $p|\pi$, and then (by tilting) that $X=\Spa(A,A^+), Y=\Spa(B,B^+)$ are perfectoid spaces of characteristic $p$. As $f\colon Y\to X$ is \'etale, there exists by \cite[Proposition 6.4.(iv)]{etale-cohomology-of-diamonds} some \'etale morphism $f_0\colon Y_0=\Spa(B_0,B_0^+\to X_0=\Spa(A_0,A_0^+)$ of affinoid perfectoid spaces, which are weakly of perfectly finite type over $\Spa(\mathbb{F}_p((\pi^{1/p^\infty})))$ (in the sense of \cite[Definition 3.1.13.]{mann-mod-p-6-functors}), and morphisms $g\colon X\to X_0, Y\to Y_0$ such that $Y\cong Y_0\cprod_{X_0} X$. By \cite[Theorem 3.1.17]{mann-mod-p-6-functors} the natural functor
  \[
    \D^a_{\solid}(A_0^+/\pi)\to \varprojlim\limits_{[n]\in \Delta}\D^a_\solid((B_0^{n/A_0})^+/\pi)
  \]
  is an equivalence (here $(B_0^{n/A_0})^+$ denotes the $+$-ring for the $n$-th stage of the \v{C}ech nerve of $Y_0\to X_0$). Using the arguments from \cref{sec:universal-descent-1-descent-implies-universal-descent-for-analytic-rings}, we see that this implies that the functor
  \[
    \D^a_{\solid}(A^+/\pi)\to \varprojlim\limits_{[n]\in \Delta}\D^a_\solid((B^{n/A})^+/\pi)
  \]
  is an equivalence as well. More precisely, the critical statement to check is that for any morphism $g\colon X=\Spa(A,A^+)\to X_0=\Spa(A_0,A_0^+)$ of affinoid perfectoid spaces over $\Spa(\mathbb{F}_p((\pi^{1/p^\infty})))$ the pushforward $g_\ast^a\colon \D^a_\solid(A^+/\pi)\to \D^a_\solid(A^+_0/\pi)$ on almost categories is conservative. This however follows formally from conservativity of $g_\ast\colon \D_\solid(A^+/\pi)\to \D_\solid(A^+_0/\pi)$ by embedding the almost categories via $\ast$-pushforward, i.e., the functors of almost elements over $A^+/\pi$ resp.\ $A^+_0/\pi$, into $\D_\solid(A^+/\pi)$ resp.\ $\D_\solid(A^+_0/\pi)$. This finishes the proof.
\end{proof}

\begin{remark}
  \label{sec:desc-totally-disc-1-more-general-version-for-descent-on-totally-disconnected-spaces}
The same argument as in \cref{rslt:v-hyperdescent-for-O+-modules} can be used to show the following: Let $X \to S$ be a map in $\AffPerfd$. Then there is a unique hypercomplete v-sheaf
\begin{align*}
  (\Perfd_{/S})^\opp \to \mathcal{C}at_\infty, \qquad T \mapsto \D^a_{\hat\solid}(\ri^+_{X_T})
\end{align*}
such that if $T$ admits a quasi-pro-étale map to some totally disconnected space then $\D^a_{\hat\solid}(\ri^+_{X_T}) = \D^a_{\hat\solid}(B^+)$, where $X \cprod_S T = \Spa(B, B^+)$. Indeed, \cref{rslt:v-cover-of-td-spaces-is-weakly-descendable} implies descent after base change.
\end{remark}

We note that logically we did not use \cref{sec:introduction-1-mod-p-theorem-of-lucas} in the proof of \cref{rslt:v-hyperdescent-for-O+-modules}, only in its disguise through the key ingredient \cref{sec:desc-totally-disc-1-finite-type-over-totally-disconnected-implies-flat-finitely-presented-mod-pi} and the \'etale descent from \cite[Lemma 3.1.22]{mann-mod-p-6-functors}.

\begin{remark}
  \label{relation-to-lucas-mod-pi-version}
  Let $X=\Spa(A,A^+)\in \AffPerfd$ with pseudo-uniformizer $\pi\in A$. It follows from \cref{rslt:v-hyperdescent-for-O+-modules} and \cref{rslt:A-mod-I-modified-modules-equals-usual-modules} that
  \[
    \mathrm{Mod}_{A^+/\pi}(\D^a_{\hat\solid}(\mathcal{O}^+))\cong \D^a_{\solid}(\mathcal{O}^+_X/\pi),
  \]
  with the right hand side defined in \cite[Definition 3.1.3]{mann-mod-p-6-functors}. Indeed, both sides satisfy $v$-descent and if $X$ is totally disconnected we can apply \cref{rslt:A-mod-I-modified-modules-equals-usual-modules}.
\end{remark}

Even though the present paper focuses on understanding the category $\D^a_{\hat\solid}(\ri^+_X)$ for (affinoid) perfectoid spaces, it is sometimes useful to extend the definition to general small v-stacks. Let us therefore introduce the following notation:

\begin{definition} \label{def:D-ri-on-untilted-small-vstacks}
\begin{defenum}
  \item We denote by
  \begin{align*}
    \vStacks^\sharp := \vStacks_{/\Spd(\Z_p)}
  \end{align*}
  the category of \emph{untilted small v-stacks}, i.e. small v-stacks $X$ (in the sense of \cite[Definition~12.4]{etale-cohomology-of-diamonds}) together with a map $X \to \Spd(\Z_p)$. Note that if $X$ is representable, i.e. given by a perfectoid space in characteristic $p$, then a map $X \to \Spd(\Z_p)$ is exactly given by the choice of an untilt of $X$. In particular we can identify $\vStacks^\sharp$ with the category of small v-stacks on $\Perfd$.

  \item We denote by
  \begin{align*}
    \D^a_{\hat\solid}(\ri^+_{(-)})\colon (\vStacks^\sharp)^\opp \to \catcat, \qquad X \mapsto \D^a_{\hat\solid}(\ri^+_X)
  \end{align*}
  the unique hypercomplete v-sheaf of categories that restricts to the one from \cref{rslt:v-hyperdescent-for-O+-modules} on perfectoid spaces.

  \item For every map $f\colon Y \to X$ in $\vStacks^\sharp$ we denote
  \begin{align*}
    f^\ast\colon \D_{\hat\solid}^a(\mathcal{O}^+_X) \rightleftarrows \D^a_{\hat\solid}(\mathcal{O}^+_Y) \noloc f_\ast
  \end{align*}
  the adjunction of pullback and pushforward functors, where $f^*$ is defined as the restriction map of the sheaf $\D^a_{\hat\solid}(\ri^+_{(-)})$.\footnote{It follows by \cite[Proposition 5.5.3.13]{lurie-higher-topos-theory} that $\D^a_{\hat\solid}(\mathcal{O}^+_Z)$ is presentable for any $Z\in \Perfd$. By reduction to the totally disconnected case it follows that $f^\ast$ preserves colimits, and hence admits a right adjoint.}

  \item If $X=\Spa(A,A^+)$ is an affinoid perfectoid space, we denote by
  \begin{align*}
    \widetilde{(-)}\colon \D^a_{\hat\solid}(A^+)\to \D^a_{\hat\solid}(\mathcal{O}^+_X)
  \end{align*}
  the natural functor, and by $\Gamma(X,-)$ its right adjoint.
\end{defenum}
\end{definition}

The major goal of this paper is to find a large class of affinoid perfectoid spaces $X = \Spa(A, A^+)$ for which we have $\D^a_{\hat\solid}(\mathcal O^+_X) = \D^a_{\hat\solid}(A^+)$. The next two results provide useful abstract criteria for when this identification is valid. The first result allows one to argue via descent of modules from a totally disconnected covers, while the second result allows a reduction of the descent modulo $\pi$ if one has a good enought understanding of $\pi$-complete objects.

\begin{lemma}
  \label{sec:bound-cond-1-criterion-for-tilde-to-be-an-equivalence}
  Let $X=\Spa(A,A^+)\in \AffPerfd$. Assume that there exists a quasi-pro-\'etale cover $Y\to X$ with $Y=\Spa(B,B^+)$ totally disconnected such that almost modules (in the sense of \cref{sec:defin-d_hats-2-definition-of-almost-plus-category}) descent along $(A^+)_{\solid}\to (B^+)_\solid$. Then the natural functor
  \[
    \widetilde{(-)}\colon \D^a_{\hat\solid}(A^+)\to \D^a_{\hat\solid}(\mathcal{O}^+_X)
  \]
  is an equivalence.
\end{lemma}
\begin{proof}
  Let $Y_\bullet\to X$ be the \v{C}ech nerve of $Y\to X$ and write $Y_n=\Spa(B_n,B_n^+)$. Note that each $Y_n, n\geq 0$, is quasi-pro-\'etale over the totally disconnected space $Y=Y_0$. Then
  \[
    \D^a_{\hat\solid}(\mathcal{O}^+_X) \cong \varprojlim\limits_{[n]\in \Delta} \D^a_{\hat\solid}(\mathcal{O}^+_{Y_n})\cong \varprojlim\limits_{[n]\in \Delta} \D^a_{\hat\solid}(B_n^+)\cong \D^a_{\hat\solid}(A^+),
  \]
  using \cref{rslt:v-hyperdescent-for-O+-modules} in the second isomorphism and that almost modules descent along $(A^+)_\solid\to (B^+)_\solid$ in the third.
\end{proof}

\begin{lemma} \label{sec:bound-cond-1-criterion-for-sheafification-being-an-equivalence}
  Let $X=\Spa(A,A^+)\in \AffPerfd$ with pseudouniformizer $\pi \in A$. Then the functor
  $\widetilde{(-)}\colon \D^a_{\hat\solid}(A^+)\to \D^a_{\hat\solid}(\mathcal{O}^+_X)$ is an equivalence if and only if the following conditions are satisfied:
  \begin{enumerate}[(a)]
    \item The functor $\widetilde{(-)}\colon \D^a_\solid(A^+/\pi) \isoto \D^a_\solid(\mathcal O^+_X/\pi)$ is an equivalence.

    \item $\Gamma(X,-)\colon \D^a_{\hat\solid}(\mathcal{O}^+_X)\to \D^a_{\hat\solid}(A^+)$ preserves colimits.

    \item $\D^a_{\hat\solid}(\mathcal{O}^+_X)$ is generated under colimits by $\pi$-complete objects $\mathcal{N}$ with $\Gamma(X,\mathcal{N})$ bounded to the right.
  \end{enumerate}
\end{lemma}
\begin{proof}
  The necessity is clear (as $\D^a_{\hat\solid}(A^+)$ is generated by $\pi$-complete objects), so let us check the converse. We need to show that the adjoint functors
\begin{align*}
  \widetilde{(-)}\colon \D^a_{\hat\solid}(A^+) \rightleftarrows \D^a_{\hat\solid}(\ri^+_X) \noloc \Gamma(X, -)
\end{align*}
are inverse to each other. We first show that $\widetilde{(-)}$ is fully faithful, i.e. that for every $M \in \D^a_{\hat\solid}(A^+)$ we have $\Gamma(X, \widetilde M) = M$. Since $\Gamma(X, -)$ preserves colimits, we can assume that $M=P^a$ for a compact generator $P\in \D_{\hat\solid}(A^+)$, in particular $M$ is $\pi$-complete and bounded to the right. We claim that $\widetilde{M}\in \D^a_{\hat\solid}(\mathcal{O}^+_X)$ is $\pi$-complete. This can be checked on a $v$-cover $Y\to X$ with $Y$ totally disconnected. By naturality of $\widetilde{(-)}$ and \cref{rslt:basic-properties-of-adically-complete-modules}, this reduces to the case that $Y=X$, where $\widetilde{(-)}$ is an equivalence. Thus $\Gamma(X, \widetilde M)$ and $ M$ are $\pi$-complete and hence the identity $\Gamma(X, \widetilde M) = M$ can be checked modulo $\pi$, where it is true by (a).

It remains to show that $\widetilde{(-)}$ is essentially surjective. Since by assumption $\D^a_{\hat\solid}(\ri^+_X)$ is generated under colimits by $\pi$-complete objects, it is enough to show that every $\pi$-complete object lies in the essential image of $\widetilde{(-)}$. Given a $\pi$-complete $\mathcal N \in \D^a_{\hat\solid}(\ri^+_X)$ we need to see that the counit $\widetilde{\Gamma(X, \mathcal N)} \isoto \mathcal N$ is an isomorphism. By assumption we may even assume that $\Gamma(X,\mathcal{N})$ is bounded to the right. Then by the same argument as above both sides of the claimed isomorphism are $\pi$-complete, so we can check the claimed isomorphism modulo $\pi$. Then it reduces again to (a).
\end{proof}

\subsection{Families of Riemann--Zariski spaces}
\label{sec:suppl-discr-adic}

In this subsection we provide some general results on families of Riemann--Zariski spaces which will enter in the proof of the descent on relative compactifications of totally disconnceted spaces in the next subsection (see \cref{rslt:descendability-for-relative-compactification-of-td-spaces} below). Unless stated otherwise, each Huber pair is classical in this section. A related discussion can be found in \cite{temkin2011relative}, but we need to ensure that the relative Riemann--Zariski space is a cofiltered inverse limits of \textit{projective} schemes (and not merely of proper schemes) in order to ensure the uniform bound on descendability in \cref{rslt:descendability-of-open-cover-of-ZR-space-family}.

\begin{lemma} \label{rslt:Zariski-lemma-for-connected-components}
Let $A^+ \to A$ be a map of classical, discrete rings such that each connected component of $X := \Spa(A, A^+)$ is of the form $\Spa(K, A_x^+)$ for some field $K$ and some subring $A_x^+ \subset K$. Let $I$ be the category of factorizations $\Spec A \to Y_i \to \Spec A^+$ such that $Y_i$ is a reduced projective $A^+$-scheme and the map $\Spec A \to Y_i$ is dominant. Then:
\begin{lemenum}
  \item $I$ is cofiltered and for every $i \in I$ the map $\Spec A \to Y_i$  factors uniquely over a map $X \to Y_i$ over $\Spec(A^+)$. The induced map $\abs X \to \abs{Y_i}$ is closed.

  \item The map $\abs X \isoto \varprojlim_i \abs{Y_i}$ is a homeomorphism of topological spaces.
\end{lemenum}
\end{lemma}
We note that under the assumptions of this lemma the natural map $\Spa(A,A)\to \Spec(A)$ is an isomorphism of locally ringed spaces, which makes condition (i) well-defined.
\begin{proof}
Let us first show that $I$ is cofiltered. Given any $i, i' \in I$ there can be at most one map $Y_i \to Y_{i'}$ in $I$. Indeed, this follows immediately from the fact that the map $\Spec A \to Y_i$ is an epimorphism (in general dominant morphisms are epimorphisms in the category of reduced separated schemes, as one easily checks by a standard diagonal argument). Thus to prove that $I$ is cofiltered we only need to check that there is some $i'' \in I$ such that there are maps $Y_{i''} \to Y_i$ and $Y_{i''} \to Y_{i'}$. But note that we can take $Y_{i''}$ to be the scheme-theoretic image of the natural map $\Spec A \to Y_i \cprod_{\Spec A^+} Y_{i'}$.

Now fix some $i \in I$. Note that $X$ is the relative compactification of the map $\Spec A \to \Spec A^+$ (see \cite[Example 2.9.27]{mann-mod-p-6-functors}), so since $Y_i \to \Spec A^+$ is proper (and thus its own relative compactification) we deduce that the map $\Spec A \to Y_i$ extends canonically to a map $X \to Y_i$ over $\Spec(A^+)$. The uniqueness of the extension can be checked on connected components, where it follows from the valuative criterion for properness. We now check that the morphism $X\to Y_i$ is closed. As $A$ is discrete, the morphism $X\to Y_i$ is a quasi-compact morphism of spectral spaces. Because it is as well specializing (by the valuative criterion for properness) this implies that the morphism $X\to Y_i$ is closed. Indeed, by quasi-compactness of $X\to Y_i$ the image of a closed subsets is pro-constructible and stable under specialization, hence closed (\cite[Lemma 0903]{stacks-project}). This proves (i).

It remains to prove (ii). By (i) we know that the map $\abs X \to \varprojlim_i \abs{Y_i}$ is closed and continuous, so it is enough to show that this map is bijective. Let us first check that it is injective, so let $y, y' \in \abs X$ be two given points which get mapped to the same point in each $\abs{Y_i}$. Let $\widetilde{A^+}$ be the integral closure of the image of $A^+$ in $A$. Then $\pi_0(\Spec(\widetilde{A^+}))=\pi_0(\Spec(A))$ and writing $\widetilde{A^+}$ as a filtered colimit of finite, reduced $A^+$-algebras, we see that $y$ and $y'$ lie in the same connected component $x = \Spa(K, A^+_x)$ of $X$ and hence correspond to valuation rings $V_y, V_{y'} \subset K$ containing $A^+_x$. Assuming $y \ne y'$ we can w.l.o.g. assume that there is some $\overline a \in V_y \setminus V_{y'}$. Let $a \in A$ be a lift of $\overline a$, which we can assume to be a unit (note that $K$ is the filtered colimit of the ring of functions of clopen neighbourhoods of $\Spec K \subset \Spec A$). Then the pair $\{ 1, a \}$ determines a map $f\colon \Spec A \to \mathbb P^1_{A^+}$ and we let $Y_a$ be the scheme theoretic image of $f$. Then $Y_a \in I$ and we claim that the induced map $X \to Y_a$ separates $y$ and $y'$. To see this, we can assume that the map $A^+ \to A$ is injective, because otherwise the surjection from $A^+$ to its image in $A$ induces a closed immersion on schemes and in particular an injective map. Now $Y_a$ can be described explicitly: It is obtained by gluing $\Spec A^+[a]$ and $\Spec A^+[1/a]$ along $\Spec A^+[a, 1/a]$. Moreover, the map $X \to Y_a$ sends a valuation ring $V \subset K$ to the unique point on $Y_a$ such that $V$ dominates the associated local ring. Thus $V_{y'}$ is sent to a point in $\Spec A^+[1/a]$ which does not lie in $\Spec A^+[a, 1/a]$, whereas $V_y$ is sent to a point in $\Spec A^+[a]$.

We now prove surjectivity of the map $\abs X \to \varprojlim_i \abs{Y_i}$, so let any $(y_i)_i \in \varprojlim_i \abs{Y_i}$ be given. Replacing $A^+$ by $\widetilde{A^+}$ we may assume that $\pi_0(\Spec(A))=\pi_0(\Spec(A^+))$. Then all $y_i$ live over the same connected component of $X$, so by base-changing to this connected component we can assume that $X$ is connected, i.e. of the form $X = \Spa(K, A^+)$ for a field $K$. For each $i$ the dominant map $\Spec K \to Y_i$ induces an injection $\ri_{Y_i,y_i} \subset K$. Let $R \subset K$ be the (filtered) union of all $\ri_{Y_i,y_i}$. Then $R$ is a local ring and hence dominated by some valuation ring $V \subset K$. This valuation ring corresponds to a point $y \in \abs X$ which maps to $(y_i)_i$ (by uniqueness), as desired.
\end{proof}

The next lemma yields the crucial uniform descendability bound for \cref{rslt:descendability-for-relative-compactification-of-td-spaces}. It is the analog of \cite[Proposition 2.10.10]{mann-mod-p-6-functors} for families of Riemann--Zariski spaces.

\begin{lemma} \label{rslt:descendability-of-open-cover-of-ZR-space-family}
Let $A^+ \to A$ be a map of classical, discrete rings and $X = \Spa(A, A^+)$ the associated discrete adic space. Assume that:
\begin{enumerate}[(a)]
\item Every connected component of $\Spec A$ is of the form $\Spec K'$ for some field $K'$ and every connected component of $\Spec A^+$ is of the form $\Spec V$ for some valuation ring $V$.
\item There is an integer $d \ge 0$ such that for every map of connected components $\Spec K' \to \Spec V$ induced by $A^+ \to A$, the associated map $V \injto K'$ is injective and the transcendence degree of $K'$ over the fraction field of $V$ is $\le d$.
\end{enumerate}
Let $U = \bigdunion_{j=1}^n U_j \surjto X$ be an open covering by quasicompact open subsets $U_j \subset X$. Then $U \to X$ is descendable of index bounded by a constant only depending on $d$.
\end{lemma}
\begin{proof}
Let $I$ be as in \cref{rslt:Zariski-lemma-for-connected-components}. Then by \cref{rslt:Zariski-lemma-for-connected-components} we have $\abs X = \varprojlim_i \abs{Y_i}$ and all the maps $\abs X \to \abs{Y_i}$ are closed, which together implies that there is some $i \in I$ and an open covering $Y_i = \bigunion_{j=1}^m W_{ij}$ by qcqs open subsets $W_{ij} \subset Y_i$ such that $U \to X$ is the base-change of the map $W_i := \bigdunion_{j=1}^m W_{ij} \surjto Y_i$ along the map $X \to Y_i$. By \cite[Lemma 2.6.9]{mann-mod-p-6-functors} it is thus enough to show that the map $W_i \to Y_i$ is descendable of bounded index.

To simplify notation let us write $Y := Y_i$ and $W := W_i$. By \cite[Corollary 2.10.7]{mann-mod-p-6-functors} it is enough to show that $Y$ can be covered by $d+1$ affine open subschemes. The map $\pi_0(\Spec A) \to \pi_0(\Spec A^+)$ is a map of compact Hausdorff spaces and in particular has closed image. This image thus corresponds to an affine subscheme $\Spec A'^+ \subset \Spec A^+$ (the limit over all clopen neighbourhoods of the set-theoretic image inside $\Spec A^+$) by definition of $I$ we know that $Y$ is supported over $\Spec A'^+$. We can thus replace $A^+$ by $A'^+$ in order to assume that the map $\pi_0(\Spec A) \to \pi_0(\Spec A^+)$ is surjective. Now pick a connected component $x = \Spec V \in \pi_0(\Spec A^+)$ and let $Y_x$ be the fiber of $Y$ over $x$. Denote by $\eta$ and $s$ the open and closed point of $V$ and let $Y_{x,\eta}$ and $Y_{x,s}$ be the the generic and special fiber of $Y_x$, respectively. We claim:
\begin{itemize}
  \item[($*$)] $Y_{x,s}$ has dimension $\le d$. 
\end{itemize}
To see this, we first claim that there are finitely many connected components $y_1 = \Spec K_1'$, \dots, $y_n = \Spec K_n'$ over $x$ such that the map $y_1 \dunion \dots \dunion y_n \to Y_x$ is dominant. To see this, fix any irreducible component $Z \subset Y_{x,\eta}$. For every (qcqs) open subset $W \subset Z$, let $S_W \subset \Spec A$ be the preimage of $W$ under the map $\Spec A \to Y$. Then $S_W$ is a quasicompact open subset and since the topological space of $\Spec A$ is profinite it follows that $S_W$ is closed. We claim that the intersection of all $S_W$, for $W$ ranging through non-empty subsets of $Z$, is non-empty. Otherwise, by compactness there would be some $W_1, \dots, W_m \subset Z$ such that $S_{W_1} \isect \dots \isect S_{W_m} = \emptyset$. Since $Z$ is irreducible, $W := W_1 \isect \dots \isect W_m$ is still a non-empty open subset of $Z$ and we deduce $S_W = \emptyset$; but this contradicts the fact that $\Spec A \to Y$ is dominant. Pick any $y$ in the intersection of all $S_W$; then the map $y \to Z$ is dominant. Repeating this procedure for all irreducible components of $Y_{x,\eta}$, we find $y_1, \dots, y_n \subset \Spec A$ such that the map $y_* := y_1 \dunion \dots \dunion y_n \to Y_{x,\eta}$ is dominant. By the assumption that the induced maps $V \injto K'$ are injective (see (b)) we see that every point of $Y_x$ that is hit by some $y \in \Spec A$ necessarily lies in $Y_{x,\eta}$. Consequently, since the map $\Spec A \to Y$ is dominant, we deduce that $y_* \to Y_x$ must be dominant.

With $y_1, \dots, y_n$ as in the previous paragraph, let $Z_1, \dots, Z_n \subset Y_x$ be the scheme-theoretic images of these points under the map $\Spec A \to Y$. Then $Y_x = Z_1 \union \dots \union Z_n$, so to prove claim ($*$) it is enough to show that each $Z_{i,s}$ has dimension $\le d$. Since $Z_i$ is integral, by \cite[Lemma 0B2J]{stacks-project} it is enough to show that the generic fiber $Z_{i,\eta}$ has dimension $\le d$. But this follows immediately from the fact that $Z_i$ is a projective $K$-variety and there is a dominant map $\Spec K'_i \to Z_{i,\eta}$ such that $K'_i$ has transcendence degree $\le d$ over $K$. This finishes the proof of claim ($*$).

With claim ($*$) proved, we now proceed with the proof that $Y$ is covered by $d+1$ affine open subschemes. Let $\kappa = \kappa(s)$ be the residue field of $s$. By the proof of \cite[Lemma 2.10.9]{mann-mod-p-6-functors} we can find some $n \ge 0$ such that $Y \subset \mathbb P^n_{A^+}$ and homogeneous polynomials $\overline f_0, \dots, \overline f_d \in \kappa[t_1, \dots, t_n]$ such that the intersection of the vanishing loci of the $\overline f_j$'s in $\mathbb P^n_\kappa$ is disjoint from $Y_{x,s}$. We can assume that each $\overline f_j$ has a coefficient $1$ and we can then lift each $\overline f_j$ to a homogeneous polynomial $f_j \in A^+[t_1, \dots, t_n]$, each of which has one coefficient $1$ (first lift the coefficients to $V$ and then use that $V$ is the colimit of the functions on clopen neighbourhoods of $x \subset \Spec A^+$). Let $H_j \subset \mathbb P^n_{A^+}$ be the zero locus of $f_j$. As in the proof of \cite[Lemma 2.10.9]{mann-mod-p-6-functors} each $\mathbb P^n_{A^+} \setminus H_j$ is affine. Now let $Y_j := Y \isect (\mathbb P^n_{A^+} \setminus H_j)$. This is a closed subscheme of the affine scheme $\mathbb P^n_{A^+} \setminus H_j$ and hence affine. Let $Y' := \bigunion_{j=0}^d Y_j$. This is an open subscheme of $Y$ and by the argument in \cite[Lemma 2.10.9]{mann-mod-p-6-functors} it contains $Y_x$.

Since the map $Y \to \Spec A^+$ is closed, the image $Z' \subset \Spec A^+$ of $Y \setminus Y'$ under this map is closed. By construction we have $x \not\in Z'$. By the argument in \cite[Lemma 7.5]{etale-cohomology-of-diamonds} there exists a clopen neighborhood $U \subseteq \Spec A^+$ of $x$ contained in $\Spec( A^+\setminus Z'$. Hence $Y_U = Y \cprod_{\Spec A^+} U$ is covered by the affine opens $(Y_0)_U, \dots, (Y_d)_U$.

The above proof shows that for every $x \in \pi_0(X)$ there is a clopen neighbourhood $U_x \subset \Spec A^+$ of $x$ such that $Y_U$ can be covered by $d + 1$ affine open subsets. By passing to a finite disjoint cover of $\Spec A^+$ by $U_x$'s we can deduce that $Y$ can be covered be $d + 1$ affine open subsets, as desired.
\end{proof}

In \cref{rslt:descendability-for-relative-compactification-of-td-spaces} we used as well the following isomorphism of topological spaces (similar to \cite[Lemma 3.6.1]{mann-mod-p-6-functors}). Note that for any Tate-Huber pair $(A,A^+)$ there exists a natural specialization map
\[
 \mathrm{sp}=\mathrm{sp}_{(A,A^+)}\colon |\Spa(A,A^+)|\to |\Spa(A^{\circ}/A^{\circ\circ},A^+/A^{\circ\circ})|,
\]
which is uniquely determined by naturality and the requirement that if $(A,A^+)=(K,K^+)$ for a non-archimedean valued field $K$ and an open and bounded valuation subring $K^+\subseteq K$, then $\mathrm{sp}_{(K,K^+)}$ sends a point $x\in \Spa(K,K^+)$ corresponding to a prime ideal $\mathfrak{p}_x$ containing $K^{\circ\circ}$ to the valuation on $\mathcal{O}_K/K^{\circ\circ}$ corresponding to the prime ideal $\mathfrak{p}_x/K^{\circ\circ}$. 

\begin{lemma} \label{rslt:homeo-to-reduction-of-family-of-ZR-spaces}
Let $\Spa(A, A^+)$ be an affinoid Tate adic space such that every connected component is of the form $\Spa(K, A^+_K)$, where $K$ is a non-archimedean field and $A^+_K \subset K$ is a subring. Then the specialiation map
\begin{align*}
  \mathrm{sp}\colon |\Spa(A, A^+)| \to |\Spa(A^\circ/A^{\circ\circ}, A^+/A^{\circ\circ})|
\end{align*}
is a homeomorphism.
\end{lemma}
\begin{proof}
  Without loss of generality $A^+\subseteq A$ is an open and integrally closed subring in $A^\circ$. We denote $X = \Spa(A, A^+)$, $\overline A = A^\circ/A^{\circ\circ}$, $\overline A^+ = A^+/A^{\circ\circ}$ and $\overline X = \Spa(\overline A, \overline A^+)$. Then the claim is that the specialization map induces a homeomorphism $\abs X \isom \abs{\overline X}$.
  
  First note that there is a canonical homeomorphism $\pi_0(X) = \pi_0(\overline X)$, as both can be derived from the idempotent elements in $A^+$.
  Fix a connected component $c = \Spa(K, A^+_K)$ of $X$ and let $\overline c$ be the corresponding connected component of $\overline X$. Let us write $\overline c = \Spa(\overline K, \overline A^+_K)$. Then $K^\circ$ is the completed filtered colimit of $\ri^\circ_X(W)$ for clopen neighbourhoods $W \subset X$ of $c$ as the clopen neighborhoods of $c$ are cofinal among all neighborhoods of $c$ (by the argument in \cite[Lemma 7.5]{etale-cohomology-of-diamonds}). Similarly, $\overline K$ is the filtered colimit of $\ri_{\overline X}(\overline W)$ for clopen neighbourhoods $\overline W \subset \overline X$ of $\overline c$. By the homeomorphism of spaces of connected components, the $W$ correspond to $\overline W$ and under this correspondence we have $\ri_{\overline X}(\overline W) = \ri^\circ_X(W)/\ri^{\circ\circ}_X(W)$. It follows that $\overline K = K^\circ/K^{\circ\circ}$, which is a discrete field. 

With $c$ and $\overline c$ as above, we note that the points of $c$ are precisely the open and bounded valuation rings $V$ of $K$ such that $A^+_K \subset V$. Similarly, the points of $\overline c$ are the valuation rings $\overline V$ of $\overline K$ such that $\overline A^+_K \subset \overline V$. These two sets agree as each open and bounded valuation ring $V$ of $K$ must contain $K^{\circ\circ}$. As $c$ was arbitrary (and $\pi_0(X)=\pi_0(\overline X)$) this shows that the specialization map is bijective, and as well specializing. It remains to show that this bijection is a homeomorphism. For this it suffices to see that the specialization map is continuous and quasi-compact, e.g., using \cite[Lemma 2.5]{etale-cohomology-of-diamonds}. Note that $\overline X$ (and $X$) satisfy the assumption of \cref{rslt:topology-of-adic-space-made-of-ZR-spaces}. Hence the topology of $X$ has a basis given by $U_{\overline f,\overline W} = \{ |\overline{ f}| \le 1 \} \isect W$ for $f \in A^\circ$ (with reduction $\overline{f}\in A^{\circ}/A^{\circ\circ}$) and clopen $\overline{W} \subset \overline{X}$. Now, $\mathrm{sp}^{-1}(\{|\overline{f}|\leq 1\})=\{|f|\leq 1\}$ as can be checked by reducing to the case that $(A,A^+)$ is an affinoid field, and $\mathrm{sp}^{-1}(\overline{W})=W$ where $W\subseteq X$ is the clopen subset corresponding to $\overline{W}$ under the homeomorphism $\pi_0(X)\cong \pi_0(\overline X)$. In particular, we can conclude that $\mathrm{sp}^{-1}(U_{\overline{f},\overline{W}})$ is open and quasi-compact. This finishes the proof.
\end{proof}

The proof of \cref{rslt:homeo-to-reduction-of-family-of-ZR-spaces} made use of the following slightly technical result about the topology on families of Zariski--Riemann spaces:

\begin{lemma} \label{rslt:topology-of-adic-space-made-of-ZR-spaces}
Let $X = \Spa(A, A^+)$ be an affinoid adic space such that every connected component is of the form $\Spa(K, A^+_K)$, where $K$ is a field and $A^+_K \subset K$ is a subring. Then $\abs X$ has a basis of open subsets given by $\{ \abs f \le 1 \} \isect W$ for varying $f \in A^\circ$ and clopen subsets $W \subset X$.
\end{lemma}
\begin{proof}
Suppose first that $X$ is connected, so that $A$ is a field. A basis of open subsets of $X$ is given by the rational subsets, i.e. subsets of the form $\{ \abs{f_1} \le \abs{g} \ne 0, \dots, \abs{f_n} \le \abs{g} \ne 0 \} \subset X$ for certain $f_1, \dots, f_n, g \in A$. We can assume that $g \ne 0$, so that $g$ is invertible. Then the above subset agrees with $\{ \abs{\frac{f_1}g} \le 1, \dots \abs{\frac{f_n}g} \le 1 \}$ and hence is an intersection of subsets of the form $\{ \abs f \le 1 \}$ for varying $f \in A$. Note that if $f$ is not in $A^0$ (i.e. is not powerbounded) then $\{ \abs f \le 1 \} = \emptyset$. Indeed, $A$ is either discrete or non-archimedean. If $A$ is discrete, then $A=A^\circ$, and if $A$ is non-archimedean, then $\{\abs f\leq 1\}\neq \emptyset$ implies that the point given by the non-trivial rank $1$ valuation on $A$ lies in $\{\abs \leq 1\}$, which in turn yields $f\in A^\circ$. This finishes the proof in the case that $X$ is connected.

Let now $X$ be general. We argue as in the proof of \cite[Lemma 3.6.1]{mann-mod-p-6-functors}. Fix some open subset $U \subset X$. Pick a point $x \in U$ and let $c = \Spa(K, A^+_K) \in \pi_0(X)$ denote the connected component of $X$ containing $x$. By what we have shown above, there is some $\overline f \in A^+_K$ such that $x \in \{ \abs{\overline f} \le 1 \} \subset U \isect c$. Note that $c$ is the cofiltered limit of its clopen neighbourhoods in $X$, which implies that $A^+_K$ is the completed filtered colimit of $\ri^+_X(V)$ for the clopen neighbourhoods $W \subset X$ of $x$. Therefore, after potentially modifying $\overline f$ by some topologically nilpotent element in $A^+_K$ (which does not affect $\{ \abs{\overline f} \le 1 \}$) we can extend $\overline f$ to some $f_W \in \ri^+_X(W)$ for some $W$ as before. Then
\begin{align*}
  \bigisect_{c \subset W' \subset W} (W' \isect X \setminus U \isect \{ \abs{f_W} \le 1 \}) = \emptyset,
\end{align*}
where $W'$ ranges through all clopen neighbourhoods of $c$ contained in $W$. Each term in the above intersection is closed in the construcible topology on $X$, so by compactness of the constructible topology, already a finite intersection must be empty. This implies that one of the terms must be empty, i.e. there is some clopen neighbourhood $c \subset W' \subset W$ such that $W' \isect X \setminus U \isect \{ \abs{f_W} \le 1 \} = \emptyset$. Now let $f \in A^+$ be the element which restricts to $f_W$ on $W'$ and to $0$ outside of $W'$. Let furthermore $f' \in A^+$ be the elment which is $0$ on $W'$ and $1$ outside of $W'$. Then $\{ \abs f \le 1 \} \isect W' \subset U$. Note that the left-hand side contains $x$ by construction, so we have constructed an open neighbourhood of $x$ inside $U$ which belongs to the claimed basis for the topology. This finishes the proof.
\end{proof}

\subsection{Descent for relative compactifications of totally disconnected spaces}
\label{sec:desc-relat-comp}

In this section we want to identify $\D^a_{\hat\solid}(\mathcal{O}^+_X)$ for certain relative compactifications of totally disconnected spaces. The precise statement is the following.

\begin{theorem} \label{rslt:compute-Dqcohri-on-rel-compactification-of-td-space}
Let $X \in \AffPerfd$ whose tilt admits a map $f\colon X^\flat \to Z$ to some totally disconnected space $Z$ with $\dimtrg f < \infty$. Let $X' \to X$ be a quasi-pro-étale map, where $X'$ is a totally disconnected space. Then for $Y := \overline X'^{/X} = \Spa(B, B^+)$ the natural functor
\begin{align*}
  \D^a_{\hat\solid}(B^+) \to \D^a_{\hat\solid}(\ri^+_Y)
\end{align*}
is an equivalence.
\end{theorem}
\begin{proof}
  Set $Z':=(X')^\flat$. We let $X'\to Y'$ be the unique map of perfectoid spaces whose tilt is the map $Z'\to \overline{Z'}^{/Z}$ (note that $\mathcal{O}(X')=\mathcal{O}(Y')$ as this holds for the tilt). We note that there exists a natural inclusion $Y\to Y'$ whose tilt is the natural map $\overline{Z'}^{/X}\to \overline{Z'}^{/Z}$. In fact, $Y$ is an intersection of rational open subsets in $Y'$ as this holds on the tilt. By \cref{rslt:descendability-for-relative-compactification-of-td-spaces} there is a pro-étale cover $W' \to Y'$ which is of universal almost $+$-descent, where $W'$ is a totally disconnected space, i.e., almost $+$-modules (in the sense of \cref{sec:defin-d_hats-2-definition-of-almost-plus-category}) descend along $W'\to Y'$ and all its base changes to affinoid perfectoid spaces. Then $W := Y \cprod_{Y'} W \to Y$ is a pro-étale cover with $+$-descent for $Y$, and $W$ is still totally disconnected (as it is an intersection of rational open subsets in the totally disconnected space $W'$). By \cref{sec:bound-cond-1-criterion-for-tilde-to-be-an-equivalence} this finishes the proof.
\end{proof}

In the next proposition, an untilt of a perfectoid space $Z$ of characteristic $p$ is a perfectoid space $X$ together with an identification $X^\flat\cong Z$. \cref{rslt:descendability-for-relative-compactification-of-td-spaces} is similar to \cite[Proposition 3.1.11]{mann-mod-p-6-functors}. Again the essential point is to show that open covers of relative compactifications are descendable \textit{with uniformly bounded index}. This will as in \cite[Proposition 3.1.11]{mann-mod-p-6-functors} be reduced to open covers of (families of) Riemann--Zariski spaces. The necessary results on these kind of adic spaces were discussed in \cref{sec:suppl-discr-adic}.

\begin{proposition}
  \label{rslt:descendability-for-relative-compactification-of-td-spaces}
  Let $f\colon Z' \to Z$ be a map of totally disconnected spaces in $\AffPerfd_{\mathbb{F}_p}$ with $\dimtrg f < \infty$. Let $X$ be an untilt of $\overline Z'^{/Z}$. Then there is a pro-étale cover $Y \to X$ by a totally disconnected space $Y$ such that almost $+$-modules descend along $Y\times_XX'\to X'$ for any morphism $X'\to X$ of affinoid perfectoid spaces.
\end{proposition}
\begin{proof}
Let $Y := X^\wl$ be the w-localization of $X$, which by \cite[Lemma 7.13]{etale-cohomology-of-diamonds} is a totally disconnected space of the form $Y = \varprojlim_i Y_i \surjto X$ such that each $Y_i$ is a disjoint union of rational open subsets of $X$. We will show that the map $Y \to X$ has the claimed descent properties and similarly any base change. Write $X=\Spa(A,A^+)$, $Y=\Spa(C,C^+)$. It suffices to check that the map $(A^\circ,A^+)_\solid\to (C^\circ,C^+)_\solid$ is weakly adically descendable of an index $\le c(d)$, where $c(d)$ is a constant only depending on $d:=\dimtrg f$. Indeed, then almost $+$-modules will descent universally along $Y\to X$ because $\D_{\hat\solid}(A^\circ,A^+)$ and $\D_{\hat\solid}(A^+)$ define the same almost category using $(A^\circ)^a=(A^+)^a$ (and the similar statement for $\D_{\hat\solid}(C^\circ,C^+)$ and $\D_{\hat\solid}(C^+)$). By \cref{rslt:filtered-colim-of-adic-desc-maps-is-weakly-desc} it is enough to show that for each map $Y_i=\Spa(C_i,C_i^+) \to X$ the map $(A^\circ,A^+)_\solid \to (C_i^\circ,C_i^+)_\solid$ of adic analytic rings is adically descendable with an index $\le c(d)$. By definition of adic descendability it is enough to prove this statement modulo $\pi$ for every pseudouniformizer $\pi$ on $X$. Now we can argue as in \cite[Proposition 3.1.11]{mann-mod-p-6-functors}, using \cref{rslt:descendability-of-open-cover-of-ZR-space-family} in place of \cite[Proposition 2.10.10]{mann-mod-p-6-functors} in order to get descendability (and not just fs-descendability) of the covering. In the following we provide the details.

We are given a map $U:=Y_i = \bigdunion_{j=1}^n U_j \to X$ for a cover of $X$ by qcqs open subsets $U_j \subset X$. Write $X = \Spa(A, A^+)$, $Z = \Spa(B, B^+)$ and $Z' = \Spa(B', B'^+)$, where we note that $B' = A^\flat$ and that $A^{\flat+}$ is the completed integral closure of $B^+ + B'^{\circ\circ}$ in $B'^+$. Let $\overline B' := B'^\circ/B'^{\circ\circ}$, $\overline B^+ := B^+/B^{\circ\circ}$ and $\overline X := \Spa(\overline B', \overline B^+)$. Note that $\overline B' = A^{\flat\circ}/A^{\flat\circ\circ}$, that $A^{\flat+}/A^{\flat\circ\circ}$ is the integral closure of the image of $\overline B^+$ in $\overline B'$ and that $X^\flat$ satisfies the hypothesis of \cref{rslt:homeo-to-reduction-of-family-of-ZR-spaces}. Hence by \cref{rslt:homeo-to-reduction-of-family-of-ZR-spaces} we obtain a canonical homeomorphism $\abs X = \abs{X^\flat} \isom \abs{\overline X}$. In particular, the open cover $X = \bigunion_j U_j$ corresponds to an open cover $\overline X = \bigunion_j \overline U_j$ and we define $\overline U := \bigdunion_j \overline U_j$.

We now claim that $\overline X$ satisfies the condition of \cref{rslt:descendability-of-open-cover-of-ZR-space-family}. Indeed, the connected components of $Z'$ are of the form $\Spa(K', K'^+)$ for a perfectoid field $K'$, hence the connected components of $\Spec \overline B'$ are of the form $\Spec(\ri_{K'}/\mm_{K'})$ (we note for later that this implies that $|\Spec(\overline B')|$ is a profinite set). Similarly, if a connected component of $Z$ is of the form $\Spa(K, K^+)$ then the corresponding connected component of $\Spec \overline B^+$ is of the form $K^+/\mm_K$, which is a valuation ring. This implies that the condition (a) of \cref{rslt:descendability-of-open-cover-of-ZR-space-family} is satisfied. Similarly, condition (b) is satisfied using the same $d$ (\cite[VI.10.3.Corollaire 1]{bourbaki2007algebre}). Altogether we deduce that the map $\overline U \to \overline X$ is descendable of index bounded by a constant only depending on $d$.
More precisely, by the proof of \cref{rslt:descendability-of-open-cover-of-ZR-space-family} there is a reduced projective $\overline B^+$-scheme $\overline S$ together with a dominant map $\Spec \overline B' \to \overline S$ such that the map $\overline U \to \overline X$ comes via base-change from an open covering of $\overline S$ along the map $\overline X \to \overline S$. Moreover, $\overline S$ can be covered by $d + 1$ open affine subschemes $\overline W_i = \Spec \overline R_i \subset \overline S$. As $\overline S$ is separated, the morphism $\Spec \overline B'\to \overline{S}$ is affine. Note that we have a surjective map $A^\circ/\pi \surjto A^\circ/A^{\circ\circ} = A^{\flat\circ}/A^{\flat\circ\circ} = \overline B'$ with locally nilpotent kernel. By \cite[Lemma 07RT]{stacks-project} we can therefore form the pushout $S:=\Spec(A^\circ/\pi)\cup_{\Spec(\overline B')}\overline{S}$ in the category of schemes. By the construction in \cite[Lemma 07RT]{stacks-project}, the scheme $S$ is covered by the $d+1$ open affine subschemes $\Spec(A'_i)\cup_{\Spec(\overline B'_i)} \overline W_i$, where $\overline{Z}_i=\Spec(\overline B'_i)\subseteq \Spec(\overline B')$ is the affine open preimage of $W_i$ along $\Spec(\overline B')\to \overline{S}$, and $\Spec(A'_i)\subseteq \Spec(A^\circ/\pi)$ is the unique affine open subscheme with underlying topological space $\overline Z_i$.

The morphism $\overline X \to \overline S$ lifts to a morphism $X_\pi := \Spa(A^\circ/\pi, A^+/\pi) \to S$ which is the same map on underlying topological spaces. Now the given open covering $U=\Spa(C_i,C_i^+) \to X$ induces an open covering $U_\pi:=\Spa(C_i^\circ/\pi,C_i^+/\pi)\to X_\pi$ (using that \cref{rslt:homeo-to-reduction-of-family-of-ZR-spaces} applies to $U$ as well), and $U_\pi\to X_\pi$ reduces to the covering $\overline U \to \overline X$. Thus the covering $U_\pi\to X_\pi$ is an open covering of discrete adic spaces which comes via base-change from an open covering of $S$ (as can be checked after base change to $\overline{X}$). Since $S$ is covered by $d + 1$ open affine subsets, the map $U_\pi \to X_\pi$ is therefore descendable of index bounded by a constant only depending on $d$ (\cite[Corollary 2.10.7]{mann-mod-p-6-functors}). This finishes the proof.
\end{proof}

\subsection{Cohomological boundedness conditions} \label{sec:compare-p-bounded-and-+-bounded}

In \cref{sec:bound-cond} we introduced the notion of $p$-bounded morphisms of small v-stacks. In \cite[\S3.5]{mann-mod-p-6-functors} defines a related notion (also called $p$-bounded in loc. cit.) and proves strong descent results along such maps. In the present subsection we compare both notions of boundedness and in particular show that they are essentially equivalent up to the generality in which they are defined. By combining this observation with \cref{sec:bound-cond-1-totally-disconnected-space-is-p-bounded} and \cite[Theorem~3.5.21]{mann-mod-p-6-functors}, we get a powerful descent result for $\D^a_\solid(A^+/\pi)$ on affinoid perfectoid spaces, which will enter the proof of our main result. In order to distinguish both boundedness conditions, let us rename the one from \cite[\S3.5]{mann-mod-p-6-functors}:

\begin{definition} \label{sec:bound-cond-new--1-definition-plus-bounded}
A map $f\colon Y^\prime\to Y$ of small $v$-stacks is called $+$-bounded if it is $p$-bounded in the sense of \cite[Definition 3.5.5.(b)]{mann-mod-p-6-functors}.
\end{definition}

\begin{example}
  \label{sec:bound-cond-1-example-for-+-bounded-maps}
  Let $f\colon Y'\to Y$ be a morphism of affinoid perfectoid spaces in characteristic $p$. If $f$ is quasi-pro-\'etale or weakly of perfectly finite type, then $f$ is $+$-bounded (see \cite[Lemma 3.5.10.(v)]{mann-mod-p-6-functors} and \cite[Lemma 3.5.13]{mann-mod-p-6-functors}). 
\end{example}

We now come directly to the promised comparison result of $p$-bounded and $+$-bounded morphisms:

\begin{proposition} \label{sec:main-theor-new-texorpdfs-+-bounded-equals-p-bounded}
  Let $f\colon Y\to X$ be a morphism of small $v$-stacks which is locally separated and representable in prespatial diamonds. Then $f$ is $+$-bounded if and only if it is $p$-bounded.
\end{proposition}
\begin{proof}
  Let us first assume that $f$ is $+$-bounded. We verify that $f$ satisfies the conditions of $p$-bounded maps in \cref{def:ell-bounded-map}.  Condition (i) is true by assumption on $f$ and condition (ii), i.e. $\dimtrg(f) < \infty$ after pullback to every prespatial diamond, is part of the definition of $+$-bounded maps. It remains to prove condition (iii). By stability of $+$-bounded morphisms under base change (\cite[Lemma 3.5.10.(iii)]{mann-mod-p-6-functors}) we may assume that $X$ is a strictly totally disconnected space and by \cref{sec:bound-cond-new-absolute-and-relative-notion-of-ell-boundedness} it is enough to show that $Y$ is $p$-bounded. Since $f$ is locally separated and $Y$ is qcqs, we can reduce to the case that $f$ is separated. Using \cref{sec:bound-cond-1-quasi-pro-etale-over-ell-bounded-implies-ell-bounded}, \cite[Proposition~3.6]{mod-ell-stacky-6-functors} and \cite[Lemma~3.5.10.(iv)]{mann-mod-p-6-functors} we can further pass to the relative compactification of $Y$ over $X$ in order to assume that $f$ is proper.

  Now let $y\in Y$ be a maximal point, and $j\colon Y_y\to Y$ the induced subdiamond. We need to show that $\cd_p(y)$ is bounded by some constant $d$ independent of $y$. For this it suffices to see that for every static étale sheaf $\mathcal F$ on $Y_y$ the sheaf cohomology of $\mathcal F$ is bounded to the right by $d$. Denoting $A := j_\ast(\mathcal F)$ (which is static by \cref{sec:bound-cond-1-quasi-pro-etale-over-ell-bounded-implies-ell-bounded}), we can reformulate the problem to showing that $\Hom_{\D_\et(Y,\F{p})}(\F{p}, A) \in \D(\F{p})$ is bounded to the right by $d$. But $A$ is an overconvergent étale sheaf and by \cite[Theorem 3.9.23]{mann-mod-p-6-functors} the category $\D_\et(Y,\F{p})^\oc$ of overconvergent étale sheaves on $Y$ embeds fully faithful into $\D^a_\solid(\ri^+_Y/\pi)^\varphi$ of $\varphi$-modules in $\D^a_\solid(\ri^+_Y/\pi)$ (here $\pi\in \mathcal{O}_X(X)$ is a pseudo-uniformizer). As this functor is given by tensoring with $\mathcal{O}^{+a}_Y/\pi$ it preserves containment in $\D^{\leq 0}$. As $Y$ is $+$-bounded and proper over $X$, we can conclude that $\Hom_{\D^a_\solid(\ri^+_Y/\pi)}(\ri^+_Y/\pi,-)$ has finite cohomological dimension by \cite[Lemmas~3.5.7.(i),~3.5.9]{mann-mod-p-6-functors}. This implies the same claim for $\varphi$-modules, and hence the assertion.

  We now prove the converse implication, so assume that $f$ is $p$-bounded. In order to show that $f$ is $+$-bounded, we can by \cite[Lemmas~3.5.10.(i),(ii)]{mann-mod-p-6-functors} assume assume that $X = \Spa(A, A^+)$ is a strictly totally disconnected space, which is furthermore of characteristic $p$. We may also reduce to the case that $Y$ is separated (as $f$ is locally separated and $+$-boundedness can be checked locally on the source by \cite[Lemma~3.5.10.(i)]{mann-mod-p-6-functors}). By \cref{sec:funct-da_h--1-compactification-of-prespatial-diamond} and \cite[Lemma 3.5.10.(iv)]{mann-mod-p-6-functors} we can first replace $Y\to X$ by its canonical compactification, and then reduce to the case that $Y$ is spatial and separated.
  We have to see now that for a pseudo-uniformizer $\pi\in A$, the pushforward
  \begin{align*}
    f_\ast\colon \D^a_\solid(\ri^+_{Y}/\pi,\Z)\to \D^a_\solid(A^+/\pi,\Z)=\D^a(A^+/\pi)\otimes_{\D(\Z)}\D_\solid(\Z)
  \end{align*}
  has finite cohomological dimension. Here, $\D^a_\solid(\ri^+_{Y}/\pi,\Z)$ refers to the category defined in \cite[Definition 3.5.2]{mann-mod-p-6-functors}. In particular, the implicit $t$-structures are the natural one on $\D^a_\solid(A^+/\pi,\Z)$ and the one on $\D^a_\solid(\ri^+_{Y}/\pi,\Z)$ defined by descent from the totally disconnected case, i.e., $M\in \D^a_\solid(\ri^+_{Y}/\pi,\Z)$ lies in $\D_{\geq 0}$ (resp.\ $\D_{\leq 0}$) if and only if this holds after pullback to any totally disconnected perfectoid space over $Y$. There is a natural equivalence
  \begin{align*}
    \D^a_\solid(\ri^+_Y/\pi,\Z) \cong \D_\et(Y,(\ri^+/\pi,\F{p})_\solid)^\oc := \Mod_{\ri^+/\pi}(\D_\et(Y,\F{p})^\oc \tensor_{\D(\F{p})} \D^a_\solid(A^+/\pi,\F{p})),
  \end{align*}
  generalizing \cite[Proposition 3.3.16]{mann-mod-p-6-functors} to not necessarily discrete objects. Indeed, this follows by descent from the case of totally disconnected $Y$ (see the proof of \cref{rslt:comparison-of-oc-nuc-sheaves-to-O+-modules-on-p-bd-perfd} below for a very similar argument for $\D^a_\solid(A^+)$ in place of $\D^a_{\solid}(A^+/\pi,\Z)$ and use that $\D_\nuc(Y,\F{p}) \cong \D_\et(Y,\F{p})^\oc$ as can be checked by descent from the strictly totally disconnected case).

  Using transport of structure along the above equivalence, we obtain a $t$-structure on $\D_\et(Y,(\ri^+/\pi,\F{p})_\solid)^\oc$. We claim that this $t$-structure arises by restricting a $t$-structure on
  \[
    \D_\et(Y,(A^+/\pi,\F{p})_\solid)^\oc:=\D_\et(Y,\F{p})^\oc\otimes_{\D(\F{p})} \D^a_\solid(A^+/\pi)
  \]
  to $\ri^+/\pi$-module objects. Namely, if $Y$ is strictly totally disconnected, then one checks easily that
  \[
    \D_\et(Y,(A^+/\pi,\F{p})_\solid)^\oc \cong \Mod_{C(\pi_0(Y),\F{p})\otimes_{\F{p}}A^{+a}/\pi}(\D^a_\solid(A^+/\pi))
  \]
  (see \cref{rslt:A+-linear-sheaves-on-std-space} below for details in a slightly different version of that statement). Setting the $\D_{\geq 0}$-part (resp.\ the $\D_{\leq 0}$-part) to be the one of objects whose underlying object in $\D^a_\solid(A^+/\pi)$ lies in $\D_{\geq 0}$ (resp.\ $\D_{\leq 0}$) defines a $t$-structure on $\D_\et(Y,(A^+/\pi,\F{p})_\solid)^\oc$ for which each pullback functor for a morphism $Y'\to Y$ of strictly totally disconnected spaces is $t$-exact (as $C(\pi_0(Y),\F{p})\to C(\pi_0(Y'),\F{p})$ is flat). If $Y$ is not assumed to be a strictly totally disconnected space, we therefore obtain a $t$-structure on $\D_\et(Y,(A^+/\pi,\F{p})_\solid)$ by descent, such that the pullback functors are $t$-exact. Clearly, this $t$-structure induces the given one for $\ri^{+}/\pi$-module objects. Altogether we reduce the $+$-boundedness of $f$ to showing that the functor
  \[
    \widetilde{f}_\ast\colon \D_\et(Y,(A^+/\pi,\F{p})_\solid)^\oc\to \D^a_\solid(A^+/\pi,\F{p}),
  \]
  which is right adjoint to the natural symmetric monoidal functor
  \[
    f^\ast\otimes \D^a_\solid(A^+/\pi)\colon \D^a_\solid(A^+/\pi,\F{p})\to \D_\et(Y,(A^+/\pi,\F{p})_\solid)^\oc,
  \]
  has finite cohomological dimension for the $t$-structures on both sides. In fact, we may even restrict along $\D_\solid(\F{p})\to \D^a_\solid(A^+/\pi,\F{p})$, compatibly with the $t$-structure. Namely, $\widetilde{f}_\ast$ is the base change along $\D_\solid(\F{p})\to \D^a_\solid(A^+/\pi,\F{p})$ of the functor
  \[
    \widetilde{f}_{\ast,\F{p}}\colon \D_\et(Y,{\F{p}}_\solid)^\oc:=\D_\et(Y,\F{p})^\oc\otimes_{\D(\F{p})} \D_\solid(\F{p})\to \D_\solid(\F{p}),
  \]
  which is right adjoint to the natural symmetric monoidal functor $\D_\solid(\F{p})\to \D_\et(Y,{\F{p}}_\solid)^\oc$. Moreover, $\D_\et(Y,{\F{p}}_\solid)^\oc$ acquires again a $t$-structure (by descent from the strictly totally disconnected case) such that it suffices to show that $\widetilde{f}_{\ast,\F{p}}$ has finite cohomological dimension for this $t$-structure.
  It suffices to see that there exists some $d\geq 0$ such that for any ($\kappa$-small) profinite set $S$ the composition
  \[
    \Gamma(S,-)\circ \widetilde{f}_{\ast,\F{p}}\colon \D_\et(Y,{\F{p}}_\solid)^\oc\to \D(\F{p})
  \]
  has cohomological dimension bounded by $d$. Indeed, the functors $\Gamma(S,-)\cong \Hom({\F{p}}_\solid[S],-)$ for $S$ a profinite set are $t$-exact and a conservative family of functors on $\D_\solid(\F{p})$. The diagram
\[\begin{tikzcd}
  {\D_\et(Y,{\F{p}}_\solid)^\oc} & {\D_\solid(\F{p})} \\
  {\D_\et(Y,\F{p})} & {\D(\F{p})}
  \arrow["{\Gamma(S,-)}", from=1-2, to=2-2]
  \arrow["{f_\ast}", from=2-1, to=2-2]
  \arrow["{\widetilde{f}_{\ast,\F{p}}}", from=1-1, to=1-2]
  \arrow["{G:=\mathrm{Id}_{\D_\et(Y,\F{p})^\oc}\otimes \Gamma(S,-)}"', from=1-1, to=2-1]
      \end{tikzcd}\]
    commutes because the diagram of their left adjoints commutes (the left adjoints $(-)\otimes_{\F{p}}{\F{p}}_\solid[S]$ respectively pullback along $f$ act on different tensor factors). Note that we used that $\Gamma(S,-)$ commutes with colimits and is $\D(\F{p})$-linear to write the right adjoint of the pullback functor $\D_\et(Y,\F{p})\to \D_\et(Y,{\F{p}}_\solid)^\oc$ as $G=\mathrm{Id}_{\D_\et(Y,\F{p})^\oc}\otimes \Gamma(S,-)$.

    We claim that the functor $G$ is $t$-exact for the given $t$-structure on the source and the natural one on the target. The functor $G$ commutes with pullbacks in $Y$ (as pullback along some $Y'\to Y$ and $\Gamma(S,-)$ act on different tensor factors), which reduces the $t$-exactness claim to the case that $Y$ is strictly totally disconnected. By construction, $t$-exactness can then be checked after applying the functor
    \[
      \Gamma(Y,-)\otimes \mathrm{Id}_{\D_\solid(\F{p})}\colon \D_\et(Y,{\F{p}}_\solid)^\oc\to \D_\solid(\F{p}),
    \]
    which commutes with $\Gamma(S,-)$ and $G$. Indeed, $\D_\et(Y,{\F{p}}_\solid)^\oc\cong \mathrm{Mod}_{C(Y,\F{p})}(\D_\solid(\F{p}))$ with $t$-structure induced by the one on $\D_\solid(\F{p})$, and similarly for $\D_\et(Y,\F{p})$. Thus, the claimed exactness of $G$ reduces to the $t$-exactness of $\Gamma(S,-)\colon \D_\solid(\F{p})\to \D(\F{p})$, which is true.
    Given the $t$-exactness of $G$ we can finish the proof in the general case, i.e., $Y$ is not assumed to be strictly totally disconnected anymore. Then the pushforward $f_\ast\colon \D_\et(Y,\F{p})\to \D(\F{p})$ has cohomological dimension bounded by some $d\geq 0$ because $Y$ is $p$-bounded (by \cref{sec:bound-cond-new-absolute-and-relative-notion-of-ell-boundedness}). Thus, $f_\ast\circ G\cong \Gamma(S,-)\circ \widetilde{f}_{\ast,\F{p}}$ has cohomological dimension bounded by $d$, which is independent of $S$, as desired.
\end{proof}

The previous proof used the following lemma to reduce the prespatial case to the spatial case. We extract it for reusability.

\begin{lemma} \label{sec:funct-da_h--1-compactification-of-prespatial-diamond}
  Let $X$ be a prespatial diamond, and let $X_0\subseteq X$ be a spatial subdiamond with the property that $X_0(K,\ri_K)=X(K,\ri_K)$ for all perfectoid fields $K$ of characteristic $p$. Assume that $X\to Z$ is a quasi-compact separated morphism to a small $v$-sheaf. Then the natural morphism $h\colon \overline{X_0}^{/Z}\to \overline{X}^{/Z}$ is an isomorphism.
\end{lemma}
\begin{proof}
  As $X\to Z$ and $X_0\to X$ are quasi-compact and separated, the morphism $h$ is well-defined and proper (\cite[Proposition 18.7]{etale-cohomology-of-diamonds}). In particular, $h$ is qcqs and hence we may apply \cite[Lemma 12.5]{etale-cohomology-of-diamonds}. Hence, it suffices to check that for any perfectoid field $\Spa(K,K^+)$ over $Z$, where $K^+\subseteq K$ is an open and bounded valuation subring, the map $\overline{X_0}^{/Z}(K,K^+)\to \overline{X}^{/Z}(K,K^+)$
  is bijective. But by definition of canonical compactifications (\cite[Proposition 18.6]{etale-cohomology-of-diamonds}), we get
  \begin{align*}
    \overline{X_0}^{/Z}(K,K^+) &= X_0(K,\ri_K)\times_{Z(K,\ri_K)} Z(K,K^+)\\
    &\cong X(K,\ri_K)\times_{Z(K,\ri_K)}Z(K,K^+)\\
    &=\overline{X}^{/Z}(K,K^+)
  \end{align*}
  by our assumption.
\end{proof}

\subsection{Main theorem} \label{sec:d_hats-on-+-bounded-spaces}

In this section we will prove the remaining assertion of \cref{sec:introduction-1-main-theorem-introduction}, i.e., that if $X=\Spa(A,A^+)\in \AffPerfd$ admits a map of finite $\dimtrg$ to a totally disconnected space, then
\begin{align*}
  \D^a_{\hat\solid}(A^+) \isoto \D^a_{\hat\solid}(\mathcal{O}^+_X)
\end{align*}
is an equivalence. For this we will use the criterion of \cref{sec:bound-cond-1-criterion-for-sheafification-being-an-equivalence} by using the (overconvergent) nuclear $\mathcal C$-valued sheaves from \cref{sec:mathc-valu-nucl}. The strategy can roughly be described as realizing the category $\D^a_{\hat\solid}(\mathcal{O}^+_X)$ as a category of $\mathcal{O}^+$-modules in $\D_\nuc(X,\Lambda)$, where $\Lambda=\Z_p[[\pi]]$ with $\pi$ mapping to a pseudo-uniformizer in $A^+$.

In order to implement the strategy, we need to first introduce the following variant for nuclear $\mathcal{C}$-valued sheaves. In the following we recall the definition of $p$-bounded (pre)spatial diamonds from \cref{def:ell-bounded-prespatial-diamond}.

\begin{definition}
  \label{sec:main-theor-d_hats-nuclear-almost-sheaves}
Let $X = \Spa(A, A^+)$ be an affinoid perfectoid space with pseudouniformizer $\pi$. Then for every $p$-bounded spatial diamond $Y$ we denote
\begin{align*}
	\D_\nuc(Y, (A^+)^a_{\hat\solid}) := \D_\nuc(Y, \Z_p[[\pi]]) \tensor_{\D_\nuc(\Z_p[[\pi]])} \D^a_{\hat\solid}(A^+),
\end{align*}
with $\D_\nuc(Y,\Z_p[[\pi]])$ as defined in \cref{sec:overc-solid-sheav-basic-nuclear-and-conuclear}.
\end{definition}

Note that the category $\D^a_{\hat\solid}(A^+)$ is not compactly generated, and hence not a nicely generated $\D_\nuc(\Z_p[[\pi]])$-linear category in the sense of \cref{def:nicely-generated-nuc-linear-category}. Nevertheless we get the following assertions analogous to \cref{rslt:properties-of-C-valued-oc-nuc-sheaves}:

\begin{lemma} \label{rslt:properties-of-A+-linear-oc-nuc-sheaves}
Let $X = \Spa(A, A^+)$ be an affinoid perfectoid space with pseudouniformizer $\pi$. Then:
\begin{lemenum}
	\item \label{rslt:descent-for-A+-linear-oc-nuc-sheaves} The assignment $Y \mapsto \D_\nuc(Y, (A^+)^a_{\hat\solid})$ defines a hypercomplete sheaf of categories on the big quasi-pro-étale site of $p$-bounded spatial diamonds.

	\item \label{rslt:A+-linear-oc-nuc-sheaf-generated-by-pi-complete-obj} For every $p$-bounded spatial diamond $Y$ the category $\D_\nuc(Y, (A^+)^a_{\hat\solid})$ is generated under colimits by complete objects.

	\item \label{rslt:A+-linear-sheaves-on-std-space} If $Y$ is a strictly totally disconnected space then $\D_\nuc(Y, (A^+)^a_{\hat\solid}) = \D^a_{\hat\solid}(C(\pi_0(Y), A^+))$.
\end{lemenum}
\end{lemma}
\begin{proof}
Note that $\D_{\hat\solid}(A^+)$ is a nicely generated $\D_\nuc(\Z_p[[\pi]])$-linear category. To prove this, let $(X_i)_{i \in I}$ be the family of compact generators $A^+_{\hat\solid}[S]$ for profinite sets $S$, i.e., $A^+_{\hat\solid}[S]=\alpha^\ast(A^+[S])$ with $\alpha^\ast$ as in \cref{rslt:basic-properties-of-modified-modules}. We show that they satisfy properties (i), (ii) and (iii) from \cref{def:nicely-generated-nuc-linear-category}. Part (iii) is obvious by (\cref{eq:1:absolute-vs-internal-hom}). For (ii) we note that the functor $\Hom(A^+_{\hat\solid}[S], -)\colon \D_{\hat\solid}(A^+) \to \D_\nuc(\Z_p[[\pi]])$ factors as the composition of functors
\begin{align*}
	\D_{\hat\solid}(A^+) \xto{\IHom(A^+_{\hat\solid}[S], -)} \D_{\hat\solid}(A^+) \xto{\mathrm{forget}} \D_\solid(\Z_p[[\pi]]) \xto{(-)_\nuc} \D_\nuc(\Z_p[[\pi]]).
\end{align*}
The first functor preserves colimits because compact objects in $\D_{\hat\solid}(A^+)$ are stable under tensor products. The second functor obviously preserves colimits, while \cref{sec:defin-d_hats-1-nuclearization} shows that the third functor preserves colimits. This proves that property (ii) from \cref{def:nicely-generated-nuc-linear-category} is satisfied. For property (i) we note that for every right bounded and complete $M \in \D_\nuc(\Z_p[[\pi]])$ we have
\begin{align*}
	M \tensor A^+_{\hat\solid}[S] = (M \tensor_{\Z_p[[\pi]]_\solid} \Z_p[[\pi]]_\solid[S]) \tensor_{\Z_p[[\pi]]_\solid} A^+_{\hat\solid}, 
\end{align*}
so we conclude by \cref{rslt:adic-base-change-preserves-adic-completeness}. We have finally shown that $\D_{\hat\solid}(A^+)$ is a nicely generated $\D_\nuc(\Z_p[[\pi]])$-linear category. Thus by \cref{rslt:properties-of-C-valued-oc-nuc-sheaves} all claims are true for $\D_\nuc(-, A^+_{\hat\solid}):=\D_\nuc(-,\Z_p[[\pi]])\otimes_{\D_\nuc(\Z_p[[\pi]])}\D_{\hat\solid}(A^+)$ in place of $\D_\nuc(-, (A^+)^a_{\hat\solid})$. The almostification functor $(-)^a\colon \D_{\hat\solid}(A^+)\to \D^a_{\hat\solid}(A^+)$ (with left adjoint $(-)_!$) realizes $\D^a_{\hat\solid}(A^+)$ as a retract of $\D_{\hat\solid}(A^+)$, which shows that $\D^a_{\hat\solid}(A^+)$ is dualizable as a $\D_\nuc(\Z_p[[\pi]])$-linear category (by \cref{sec:mathc-valu-geom-remark-rigidity-of-d-nuc-r}). This implies (i) by the argument in \cref{rslt:properties-of-C-valued-oc-nuc-sheaves}. Furthermore, part (ii) and (iii) follow from \cref{rslt:properties-of-C-valued-oc-nuc-sheaves} by \cref{sec:glob-stably-unif-2-closed-open-immersions-are-stable-under-base-change-for-lurie-tensor-product}.
\end{proof}

We note that $\D_\nuc(Y,(A^+)^a_{\hat\solid})$ is naturally a symmetric monoidal $\infty$-category for a $p$-bounded spatial diamond. Indeed, \cite[Theorem 8.6]{condensed-complex-geometry} implies that the category $\D_\nuc(Y,\Z_p[[\pi]])$ is stable under tensor products in $\D_\solid(Y,\Z_p[[\pi]])_{\omega_1}$.

\begin{definition}
  \label{sec:main-theor-d_hats-definition-of-weird-o-+-modules}
Let $X = \Spa(A, A^+)$ be an affinoid perfectoid space and $Y$ a $p$-bounded spatial diamond over $X^\flat$. By \cref{rslt:properties-of-A+-linear-oc-nuc-sheaves} and \cref{rslt:v-hyperdescent-for-O+-modules} there is a unique algebra object $\ri^{+a} \in \D_\nuc(Y, (A^+)^a_{\hat\solid})$ such that for every strictly totally disconnected space $Z$ over $Y$ with untilt $Z^\sharp = \Spa(B, B^+)$ over $X$, the pullback of $\ri^{+a}$ to $Z$ is the object
\begin{align*}
	B^{+a} \in \D^a_{\hat\solid}(C(\pi_0(Z), A^+)) \cong \D_\nuc(Z, (A^+)^a_{\hat\solid})
\end{align*}
with the last isomorphism supplied by \cref{rslt:A+-linear-sheaves-on-std-space}.
We denote
\begin{align*}
	\D_\nuc(Y, (\ri^+, A^+)^a_{\hat\solid}) := \Mod_{\ri^{+a}}(\D_\nuc(Y, (A^+)^a_{\hat\solid})).
\end{align*}
\end{definition}

The category $\D_\nuc(Y,(\ri^+,A^+)^a_{\hat\solid})$ is a precise version of the heuristic category of ``$\ri^+$-modules in $\D_\nuc(Y,\Z_p[[\pi]])$''. The extra factor $\D^a_{\hat\solid}(A^+)$ accomodates the almost mathematics.

\begin{lemma}
Let $X = \Spa(A, A^+)$ be a $p$-bounded affinoid perfectoid space with pseudouniformizer $\pi$. Then:
\begin{lemenum}
	\item The assignment $Y \mapsto \D_\nuc(Y, (\ri^+, A^+)^a_{\hat\solid})$ defines a hypercomplete quasi-pro-\'etale sheaf of categories on the big quasi-pro-\'etale site of $p$-bounded spatial diamonds over $X^\flat$. 

	\item \label{rslt:O+-linear-oc-nuc-sheaves-generated-by-pi-cplt} For every $p$-bounded spatial diamond $Y$ over $X^\flat$ the category $\D_\nuc(Y, (\ri^+, A^+)^a_{\hat\solid})$ is generated under colimits by complete objects. 

	\item \label{rslt:O+-linear-oc-nuc-sheaves-on-std-space} If $Y=\Spa(B,B^+)$ over $X$ is strictly totally disconnected, then
	\begin{align*}
		\D_\nuc(Y^\flat, (\ri^+, A^+)^a_{\hat\solid}) = \D^a_{\hat\solid}(B^+, A^+).
	\end{align*}

	\item \label{rslt:O+-linear-oc-nuc-sheaves-on-compactification} For every separated $p$-bounded spatial diamond $Y$ over $X^\flat$ we have
	\begin{align*}
		\D_\nuc(\overline Y^{/X^\flat}, (\ri^+, A^+)^a_{\hat\solid}) = \D_\nuc(Y, (\ri^+, A^+)^a_{\hat\solid}).
	\end{align*}
\end{lemenum}
\end{lemma}
\begin{proof}
First note that the $p$-boundedness of $X$ implies that every $Y \in X_\qproet^\flat$ is $p$-bounded by \cref{sec:bound-cond-1-quasi-pro-etale-over-ell-bounded-implies-ell-bounded}. 
Part (i) follows from \cref{rslt:descent-for-A+-linear-oc-nuc-sheaves} as in the proof of \cref{sec:universal-descent-1-descent-implies-universal-descent-for-analytic-rings}; we note that the pullback functors for $\D_\nuc(-, (\ri^+, A^+)^a_{\hat\solid})$ along quasi-pro-étale maps are the same as before on underlying objects in $\D_\nuc(-, (A^+)^a_{\hat\solid})$ (because $\ri^{+a}$ is preserved under these pullbacks by definition). Part (ii) follows from \cref{rslt:A+-linear-oc-nuc-sheaf-generated-by-pi-complete-obj} by noting that $\D_\nuc(Y, (\ri^+, A^+)^a_{\hat\solid})$ is generated under colimits by objects of the form $\ri^{+a} \tensor \mathcal M$ for $\mathcal M \in \D_\nuc(Y, (A^+)^a_{\hat\solid})$ running through a set of generator. Part (iii) follows from \cref{rslt:A+-linear-sheaves-on-std-space} and the definition of $\ri^{+a}$. Indeed,
\begin{align*}
  &\mathrm{Mod}_{B^{+a}}(\D^a_{\hat\solid}(C(\pi_0(Y), A^+)))\cong \mathrm{Mod}_{B^{+a}}(\mathrm{Mod}_{C(\pi_0(Y), A^+)}(\D^a_{\hat\solid}(A^+)))\cong \mathrm{Mod}_{B^{+a}}(\D^a_{\hat\solid}(A^+)) \\&\qquad\cong \D^a_{\hat\solid}(B^+,A^+)
\end{align*}
using in the first isomorphism that $A^+\to C(\pi_0(Z),A^+)$ is integral mod $\pi$ (and hence induces the analytic ring structure).
Part (iv) follows from $\D_\nuc(\overline Y^{/X^\flat}, \Z_p[[\pi]]) = \D_\nuc(Y, \Z_p[[\pi]])$ (and the overconvergence of $\ri^{+a}$). Indeed, the claim can (by \cref{sec:nucl-objects-overc-3-overconvergent-and-nuclear-objects} be checked in the case that $X$ is strictly totally disconnected and $Y$ qcqs. Then $\overline{Y}^{/X^\flat}, Y$ are strictly totally disconnected as they are quasi-pro-\'etale over $X$. Then the claim follows from \cref{sec:nucl-objects-overc-1-nuclear-objects-are-pulled-back-from-pi-0} as $\pi_0(\overline{Y}^{/X^\flat})=\pi_0(Y)$.
\end{proof}

The next corollary is the culmination of all of our results on nuclear sheaves and the descent results of $\ri^{+a}_X$-modules from \cref{sec:desc-relat-comp}. It lies at the heart of the proof of our main descent result. Note that in that corollary already the case $Y=X^\flat$ is interesting and non-trivial. We refer the reader to \cref{def:D-ri-on-untilted-small-vstacks} for the notation used in the corollary.

\begin{corollary} \label{rslt:comparison-of-oc-nuc-sheaves-to-O+-modules-on-p-bd-perfd}
Let $X = \Spa(A, A^+)$ be an affinoid perfectoid space with pseudouniformizer $\pi$ such that $X^\flat$ admits a map of finite $\dimtrg$ to some totally disconnected space. Let $Y$ be a separated $p$-bounded spatial diamond and let $f\colon Y\to X^\flat$ be a morphism of finite $\dimtrg$. Then there is a natural equivalence
\begin{align*}
	\D^a_{\hat\solid}(\ri^+_{(\overline{Y}^{/X^\flat})^\sharp}) = \D_\nuc(Y^{\sharp}, (\ri^+, A^+)^a_{\hat\solid}).
\end{align*}
In particular, the following is true:
\begin{corenum}
	\item The category $\D^a_{\hat\solid}(\ri^+_{(\overline{Y}^{/X^\flat})^\sharp})$ is generated under colimits by right bounded, complete objects. The same holds true for $\D^a_{\hat\solid}(\ri^+_X)$.
	\item The pushforward $\widetilde{f}_\ast \colon \D^a_{\hat\solid}(\ri^+_{(\overline{Y}^{/X^\flat})^\sharp}) \to \D^a_{\hat\solid}(A^+)$ preserves all small colimits.
\end{corenum}
\end{corollary}
\begin{proof}
  Note that the functor $Y\mapsto (\overline{Y}^{/X^\flat})^\sharp$ sends quasi-pro-\'etale covers to $v$-covers and preserves fiber products. Thus by \cref{rslt:v-hyperdescent-for-O+-modules} the functor $Y\mapsto \D^a_{\hat\solid}(\ri^+_{(\overline{Y}^{/X^\flat})^\sharp})$ satisfies quasi-pro-\'etale hyperdescent.
We first construct a natural map of hypercomplete quasi-pro-étale sheaves
\begin{align*}
	\alpha(-)\colon \D_\nuc(-, (\ri^+, A^+)^a_{\hat\solid}) \to \D^a_{\hat\solid}(\ri^+_{(\overline{(-)}^{/X^\flat})^\sharp})
\end{align*}
on the quasi-pro-\'etale site of $p$-bounded spatial diamonds over $X^\flat$. It is enough to construct this map on strictly totally disconnected spaces $Y$ with untilt $Y^\sharp=\Spa(B, B^+)$ over $X$, where by \cref{rslt:v-hyperdescent-for-O+-modules} and \cref{rslt:O+-linear-oc-nuc-sheaves-on-std-space} it boils down to the natural map $\alpha(Y)\colon \D^a_{\hat\solid}(B^+, A^+) \to \D^a_{\hat\solid}(\ri^{+}_{(\overline{Y}^{/X^\flat})^\sharp})$ (which is is compatible with pullback in $Y$).

In the following, we will identify quasi-pro-\'etale sheaves on perfectoid spaces over $X$ with quasi-pro-\'etale sheaves on perfectoid spaces over $X^\flat$ and view both sides as sheaves on perfectoid spaces over $X$. Now pick a quasi-pro-\'etale hypercover $Y^\prime_\bullet \to Y$ by strictly totally disconnected spaces $Y^\prime_n$ over $X$. Then each morphism $Y^\prime_n\to X$ has finite $\dimtrg$, and it suffices to see that $\alpha(Y^\prime_n)$ is an equivalence for any $n\in \Delta$. Thus, we may assume that $Y=\Spa(B,B^+)$ is strictly totally disconnected. Set $Z:=\Spa(B,A^+)$. Then we may replace $X$ by $Z$ as this does not change both sides of $\alpha$ (using \cref{rslt:O+-linear-oc-nuc-sheaves-on-compactification}) and $Z^\flat$ still admits a morphism of finite $\dimtrg$ to a totally disconnected space. Hence, we may assume that $Y\to X$ is quasi-pro-\'etale. We claim that $\alpha(Z)$ is an equivalence. By \cref{rslt:O+-linear-oc-nuc-sheaves-on-std-space,rslt:O+-linear-oc-nuc-sheaves-on-compactification} we have
\begin{align*}
	\D_\nuc(Z, (\ri^+, A^+)^a_{\hat\solid}) \cong \D_\nuc(Y,(\ri^+,A^+)^a_{\hat\solid})\cong \D^a_{\hat\solid}(B^+, A^+).
\end{align*}
On the other hand, letting $C^+$ be the minimal ring of integral elements in $B$ containing the image of $A^+$, we obtain from \cref{rslt:compute-Dqcohri-on-rel-compactification-of-td-space} that
\begin{align*}
	\D^a_{\hat\solid}(\ri^+_Z) = \D^a_{\hat\solid}(C^+) = \D^a_{\hat\solid}(B^+, A^+).
\end{align*}
By going through the natural identifications one checks that $\alpha(Z)$ is the obvious equivalence.

It remains to prove the ``in particular'' claims. Part (i) is an immediate consequence of \cref{rslt:O+-linear-oc-nuc-sheaves-generated-by-pi-cplt} (noting that the right boundedness holds for the constructed complete generators). The case $Y=X^\flat$ implies the case of $\D^a_{\hat\solid}(\ri^+_X)$, where we note that $X^\flat$ is $p$-bounded by \cref{sec:bound-cond-1-totally-disconnected-space-is-p-bounded}. It remains to prove (ii). Under the equivalence $\alpha(Y)$ the functor $\widetilde{f}^\ast$ can be written as the composition
\begin{align*}
	\D^a_{\hat\solid}(A^+) \to \D_\nuc(Y, (A^+)^a_{\hat\solid}) \xto{- \tensor \ri^{+a}} \D_\nuc(Y, (\ri^+, A^+)^a_{\hat\solid}),
\end{align*}
where the first functor is induced (by tensoring over $\D_\nuc(\Z_p[[\pi]])$ with $\D^a_{\hat\solid}(A^+)$) from the pullback $\rho^*\colon \D_\nuc(\Z_p[[\pi]]) \to \D_\nuc(Y, \Z_p[[\pi]])$ for the morphism of sites $\rho\colon Y_\qproet\to \ast_\qproet.$ Thus $\widetilde{f}_\ast$ can be written as the composition of the right adjoints of the above two functors, so it is enough to show that each of these right adjoints preserves all small colimits. The right adjoint of $- \tensor \ri^{+a}$ is the forgetful functor, which clearly preserves all small colimits. The right adjoint $G$ of the first functor can be obtained from the global sections functor $\rho_*\colon \D_\nuc(Y, \Z_p[[\pi]]) \to \D_\nuc(\Z_p[[\pi]])$ by tensoring with $\D^a_{\hat\solid}(A^+)$. Indeed, $\rho_\ast$ commutes with colimits and is $\D_\nuc(\Z_p[[\pi]])$-linear because $\D_\nuc(\Z_p[[\pi]])$ is rigid, the left adjoint $\rho^\ast$ is symmetric monoidal and the unit in $\D_\nuc(Y,\Z_p[[\pi]])$ is compact (\cref{sec:mathc-valu-geom-remark-rigidity-of-d-nuc-r}, \cref{sec:omeg-solid-sheav-2-properties-of-omega-1-solid-sheaves}). This implies that $\rho_*\otimes_{\D_\nuc(\Z_p[[\pi]])}\D^a_{\hat\solid}(A^+)$ defines the right adjoint $G$ of $\D^a_{\hat\solid}(A^+)$ by functoriality in the 2-category of $\D_\nuc(\Z_p[[\pi]])$-linear presentable categories.\footnote{For this argument, we don't need any theory of $(\infty,2)$-categories as the necessary natural transformation can be constructed directly by functoriality of $(-)\otimes_{\D_\nuc(\Z_p[[\pi]])}\D^a_{\hat\solid}(A^+)$.} This proves the claim.
\end{proof}

We can finally show the main descent result of this paper, generalizing \cite[Theorem 3.5.21]{mann-mod-p-6-functors} and finishing the proof of \cref{sec:introduction-1-main-theorem-introduction}:

\begin{theorem} \label{sec:main-theor-d_hats-main-descent-theorem}
Assume that $X = \Spa(A, A^+) \in \AffPerfd$ and that $X^\flat$ admits a map with finite $\dimtrg$ to some totally disconnected space. Then
\begin{align*}
	\D^a_{\hat\solid}(\ri^+_X) = \D^a_{\hat\solid}(A^+).
\end{align*}
\end{theorem}
\begin{proof}
  We apply the criterion in \cref{sec:bound-cond-1-criterion-for-sheafification-being-an-equivalence}. Condition (a) of that criterion follows from \cite[Theorem~3.5.21]{mann-mod-p-6-functors} using \cref{sec:main-theor-new-texorpdfs-+-bounded-equals-p-bounded,sec:bound-cond-1-totally-disconnected-space-is-p-bounded}. Conditions (b) and (c) follow from \cref{rslt:comparison-of-oc-nuc-sheaves-to-O+-modules-on-p-bd-perfd} (and its proof to see the right boundedness of $\Gamma(X,\mathcal{N})$ for the constructed complete generators in \cref{rslt:comparison-of-oc-nuc-sheaves-to-O+-modules-on-p-bd-perfd}).  
\end{proof}

We finish this section by providing another application of \cref{rslt:comparison-of-oc-nuc-sheaves-to-O+-modules-on-p-bd-perfd} which will be important in \cite{anschuetz_mann_lebras_6_functors_for_solid_sheaves_on_fargues_fontaine_curves}. In the following, recall the notation from \cref{def:D-ri-on-untilted-small-vstacks}. We start with the following lemma:

\begin{lemma}
  \label{sec:main-theor-texorpdfs-base-change-for-quasi-pro-etale-maps-and-conservativity}
  Let $f\colon Y'\to Y$ be a morphism of untilted small $v$-stacks. Assume that $f$ is qcqs, quasi-pro-\'etale and locally separated. Then $f_\ast\colon \D^a_{\hat\solid}(\ri^+_{Y'}) \to \D^a_{\hat\solid}(\ri^+_Y)$ commutes with colimits and base change in $Y$. If $f$ is separated, then $f_\ast$ is conservative.
\end{lemma}
\begin{proof}
  As $f$ is qcqs, the commutation with colimits and the base change in $Y$ are local on $Y'$. In particular, we may assume that $f$ is separated.
  Moreover, we may by $v$-descent assume that $Y$ is a strictly totally disconnected perfectoid space (and we only have to check base changes to another strictly totally disconnected space). More precisely, compatibility with base change implies that conservativity can be checked $v$-locally. If now $Y=\Spa(A,A^+)$ is a strictly totally disconnected perfectoid space, then $Y'=\Spa(B,B^+)$ is as well by \cite[Proposition 9.6]{etale-cohomology-of-diamonds}. But then, $f_\ast\colon \D^a_{\hat\solid}(B^+)\to \D_{\hat\solid}^a(A^+)$ is just the forgetful functor. This implies commutation with colimits. Base change follows from steadiness of adic maps (see \cref{rslt:steadiness-for-modified-modules}). Conservativity of $f_\ast$ follows because by construction $f^\ast$ sends a set of generators of $\D^a_{\hat\solid}(A^+)$ to generators of $\D^a_{\hat\solid}(B^+)$, e.g., the almostifications of the compact objects in $\D_{\hat\solid}(A^+)$.   
\end{proof}

The promised application of \cref{rslt:comparison-of-oc-nuc-sheaves-to-O+-modules-on-p-bd-perfd} is now the following. In that result, we say that an object $M\in \D^a_{\hat\solid}(\ri^+_Y)$ is right-bounded if this is true after pullback to any strictly totally disconnected space over $Y$.

\begin{proposition}
  \label{sec:main-theor-texorpdfs-generation-by-complete-objects}
  Let $X = \Spa(A, A^+)$ be an affinoid perfectoid space with pseudouniformizer $\pi$ and assume that $X^\flat$ admits a map of finite $\dimtrg$ to some totally disconnected space. Then for every separated $p$-bounded morphism of small v-stacks $f\colon Y\to X$ the category $\D^a_{\hat\solid}(\ri^+_Y)$ is generated by right-bounded, $\pi$-complete objects, and the functor
  \begin{align*}
    f_\ast\colon \D^a_{\hat\solid}(\ri^+_Y)\to \D^a_{\hat\solid}(\ri^+_X)=\D^a_{\hat\solid}(A^+)
  \end{align*}
  commutes with colimits.
\end{proposition}
\begin{proof}
  Let $Z:=\overline{Y}^{/X}$ be the canonical compactification of $f$. Then the inclusion $j\colon Y\to Z$ is qcqs, quasi-pro-\'etale and separated. Thus, $j_\ast$ is conservative and commutes with colimits by \cref{sec:main-theor-texorpdfs-base-change-for-quasi-pro-etale-maps-and-conservativity}, hence $j^\ast$ sends a collection of generators to a collection of generators. Now, \cref{rslt:adic-base-change-preserves-adic-completeness} implies that $j^\ast$ preserves $\pi$-adic completedness of right-bounded objects. In particular, we may replace $Y$ by $Z$ (as $Z\to X$ is again separated and $+$-bounded). Let $Z_0\subseteq Z$ be a spatial subdiamond in the prespatial diamond $Z$, such that $Z_0(K,\ri_K)\cong Z(K,\ri_K)$ for any perfectoid field $K$. Then by \cref{sec:funct-da_h--1-compactification-of-prespatial-diamond} $Z=\overline{Z_0}^{/X}$. By \cref{sec:bound-cond-new-absolute-and-relative-notion-of-ell-boundedness} we can conclude that $Z$ (and hence $Z_0$) is $p$-bounded, and then we can apply \cref{rslt:comparison-of-oc-nuc-sheaves-to-O+-modules-on-p-bd-perfd} to $Z_0\to X$. This finishes the proof. 
\end{proof}

\bibliography{bibliography}
\addcontentsline{toc}{section}{References}

\end{document}